\documentclass[a4paper,twoside,11pt,english,intlimits]{article}

\usepackage[utf8]{inputenc}
\usepackage[T1]{fontenc}
\usepackage[english]{babel}

\usepackage{amsmath,amsthm,amsfonts,amssymb}
\usepackage{newtxtext, courier, euler}
\usepackage[scaled=0.95]{helvet}
\usepackage[cal=euler, scr=boondoxo, scrscaled=1.05]{mathalfa}

\usepackage{setspace}

\usepackage{euler,stmaryrd,upgreek}

\usepackage{enumerate}
\usepackage{fancyhdr,titlesec,url}
\usepackage[all,knot,poly]{xy}
\usepackage{array}
\usepackage{hyperref}
\usepackage{graphicx}
\usepackage[numbers,sort&compress]{natbib}
\usepackage{comment}
\usepackage{multirow}
\usepackage{ifthen}
\usepackage{time}
\usepackage{algorithm2e}

\usepackage{etoolbox}

\usepackage{cancel}
\usepackage{tikz}

\newcommand{\periodafter}[1]{\ifstrempty{#1}{}{#1.}}
\titleformat{\section}[block]{\scshape\filcenter\LARGE}{\thesection.}{.5em}{}
\titleformat{\subsection}[block]{\bfseries\filcenter\large}{\thesubsection.}{.5em}{\medskip}
\titleformat{\subsubsection}[runin]{\bfseries}{\thesubsubsection.}{.5em}{\periodafter}
\titlespacing{\subsubsection}{0pt}{\topsep}{.5em}

\newtheoremstyle{ntheorem}%
	{\topsep}{\topsep}{\itshape}{0pt}{\bfseries}{.}{.5em}%
	{\thmnumber{#2.\hspace{.5em}}\thmname{#1}\thmnote{ (#3)}}
	
\newtheoremstyle{ndefinition}%
	{\topsep}{\topsep}{\normalfont}{0pt}{\bfseries}{.}{.5em}%
	{\thmnumber{#2.\hspace{.5em}}\thmname{#1}\thmnote{ (#3)}}
	
\newtheoremstyle{nremark}%
	{\topsep}{\topsep}{\normalfont}{0pt}{\itshape}{.}{.5em}%
	{\thmnumber{}\thmname{#1}\thmnote{ (#3)}}

\theoremstyle{ntheorem}
  	\newtheorem{theorem}[subsubsection]{Theorem}
  	\newtheorem{proposition}[subsubsection]{Proposition}
	\newtheorem{lemma}[subsubsection]{Lemma}
  	\newtheorem{corollary}[subsubsection]{Corollary}

\theoremstyle{ndefinition}

	\newtheorem{remark}[subsubsection]{Remark}
	
\makeatletter
\def\@equationname{equation}
\newenvironment{eqn}[1]{%
    \def\mymathenvironmenttouse{#1}%
    \ifx\mymathenvironmenttouse\@equationname%
        \refstepcounter{subsubsection}%
    \else
        \patchcmd{\@arrayparboxrestore}{equation}{subsubsection}{}{}
        \patchcmd{\print@eqnum}{equation}{subsubsection}{}{}%
        \patchcmd{\incr@eqnum}{equation}{subsubsection}{}{}%
    \fi
    \csname\mymathenvironmenttouse\endcsname%
}{%
    \ifx\mymathenvironmenttouse\@equationname%
        \tag{\thesubsubsection}%
    \fi
    \csname end\mymathenvironmenttouse\endcsname%
}
\makeatother

\fancypagestyle{plain}{
\fancyhf{}\fancyfoot[LC,RC]{}
\fancyfoot[LE,RO]{$\thepage$}

}

\UseTips
\SelectTips{eu}{11}

\newdir{ >}{{}*!/-10pt/@{>}}
\newdir{ -}{{}*!/-10pt/@{}}
\newdir{> }{{}*!/+10pt/@{>}}

\makeatletter

\xyletcsnamecsname@{dir4{}}{dir{}}
\xydefcsname@{dir4{-}}{\line@ \quadruple@\xydashh@}
\xydefcsname@{dir4{.}}{\point@ \quadruple@\xydashh@}
\xydefcsname@{dir4{~}}{\squiggle@ \quadruple@\xybsqlh@}
\xydefcsname@{dir4{>}}{\Tttip@}
\xydefcsname@{dir4{<}}{\reverseDirection@\Tttip@}

\xydef@\quadruple@#1{%
	\edef\Drop@@{%
		\dimen@=#1\relax
		\dimen@=.5\dimen@
		\A@=-\sinDirection\dimen@
		\B@=\cosDirection\dimen@
		\setboxz@h{%
			\setbox2=\hbox{\kern3\A@\raise3\B@\copy\z@}%
			\dp2=\z@ \ht2=\z@ \wd2=\z@ \box2
			\setbox2=\hbox{\kern\A@\raise\B@\copy\z@}%
			\dp2=\z@ \ht2=\z@ \wd2=\z@ \box2
			\setbox2=\hbox{\kern-\A@\raise-\B@\copy\z@}%
			\dp2=\z@ \ht2=\z@ \wd2=\z@ \box2
			\setbox2=\hbox{\kern-3\A@\raise-3\B@ \noexpand\boxz@}%
			\dp2=\z@ \ht2=\z@ \wd2=\z@ \box2
		}%
		\ht\z@=\z@ \dp\z@=\z@ \wd\z@=\z@ \noexpand\styledboxz@
	}%
}

\xydef@\Tttip@{\kern2pt \vrule height2pt depth2pt width\z@
	\Tttip@@ \kern2pt \egroup
	\U@c=0pt \D@c=0pt \L@c=0pt \R@c=0pt \Edge@c={\circleEdge}%
	\def\Leftness@{.5}\def\Upness@{.5}%
	\def\Drop@@{\styledboxz@}\def\Connect@@{\straight@{\dottedSpread@\jot}}}
	
\xydef@\Tttip@@{%
	\dimen@=.25\dimen@
 	\B@=\cosDirection\dimen@
	\setboxz@h\bgroup\reverseDirection@\line@ \wdz@=\z@ \ht\z@=\z@ \dp\z@=\z@
	{\vDirection@(1,-1)\xydashl@ \xyatipfont\char\DirectionChar}%
	{\vDirection@(1,+1)\xydashl@ \xybtipfont\char\DirectionChar}%
}

\xydef@\ar@form{
	\ifx \space@\next \expandafter\DN@\space{\xyFN@\ar@form}%
	\else\ifx ^\next \DN@ ^{\xyFN@\ar@style}\edef\arvariant@@{\string^}%
	\else\ifx _\next \DN@ _{\xyFN@\ar@style}\edef\arvariant@@{\string_}%
	\else\ifx 0\next \DN@ 0{\xyFN@\ar@style}\def\arvariant@@{0}%
	\else\ifx 1\next \DN@ 1{\xyFN@\ar@style}\def\arvariant@@{1}%
	\else\ifx 2\next \DN@ 2{\xyFN@\ar@style}\def\arvariant@@{2}%
	\else\ifx 3\next \DN@ 3{\xyFN@\ar@style}\def\arvariant@@{3}%
	\else\ifx 4\next \DN@ 4{\xyFN@\ar@style}\def\arvariant@@{4}%
	\else\ifx \bgroup\next \let\next@=\ar@style
	\else\ifx [\next \DN@[##1]{\ar@modifiers{[##1]}}
	\else\ifx *\next \DN@ *{\ar@modifiers}%
	\else\addLT@\ifx\next \let\next@=\ar@slide
	\else\ifx /\next \let\next@=\ar@curveslash
	\else\ifx (\next \let\next@=\ar@curveinout 
	\else\addRQ@\ifx\next \addRQ@\DN@{\ar@curve@}%
	\else\addLQ@\ifx\next \addLQ@\DN@{\xyFN@\ar@curve}%
	\else\addDASH@\ifx\next \addDASH@\DN@{\defarstem@-\xyFN@\ar@}%
	\else\addEQ@\ifx\next \addEQ@\DN@{\def\arvariant@@{2}\defarstem@-\xyFN@\ar@}%
	\else\addDOT@\ifx\next \addDOT@\DN@{\defarstem@.\xyFN@\ar@}%
	\else\ifx :\next \DN@:{\def\arvariant@@{2}\defarstem@.\xyFN@\ar@}%
	\else\ifx ~\next \DN@~{\defarstem@~\xyFN@\ar@}%
	\else\ifx !\next \DN@!{\dasharstem@\xyFN@\ar@}%
	\else\ifx ?\next \DN@?{\ar@upsidedown\xyFN@\ar@}%
	\else \let\next@=\ar@error
	\fi\fi\fi\fi\fi\fi\fi\fi\fi\fi\fi\fi\fi\fi\fi\fi\fi\fi\fi\fi\fi\fi\fi \next@}

\makeatother

\newcommand{\fl}{\rightarrow}
\newcommand{\fll}{\longrightarrow}

\newcommand{\dfl}{\Rightarrow}

\newcommand{\tfl}{\Rrightarrow}
\newcommand{\qfl}{\xymatrix@1@C=10pt{\ar@4 [r] &}}

\newcommand{\cl}[1]{\overline{#1}}
\newcommand{\rep}[1]{\widehat{#1}}

\newcommand{\tck}[1]{#1^{\top}}

\renewcommand{\phi}{\varphi}
\renewcommand{\epsilon}{\varepsilon}

\newcommand{\Cr}{\mathcal{C}}

\renewcommand{\Pr}{\mathcal{P}}

\newcommand{\Sr}{\mathcal{S}}

\newcommand{\Vr}{\mathcal{V}}

\newcommand{\B}{\mathbf{B}}
\newcommand{\C}{\mathbf{C}}
\newcommand{\D}{\mathbf{D}}

\newcommand{\G}{\mathbf{G}}

\def\catego#1{\mathsf{#1}}

\newcommand{\Cat}{\catego{Cat}}
\newcommand{\Pol}{\catego{Pol}}
\newcommand{\Grpd}{\catego{Grpd}}
\newcommand{\DbCat}{\catego{DbCat}}
\newcommand{\DbGrpd}{\catego{DbGrpd}}

\newcommand{\DiPol}{\catego{DiPol}}

\newcommand{\Set}{\catego{Set}}

\newcommand{\ifthen}[2]{\ifthenelse{#1}{#2}{}}

\DeclareMathOperator{\Sq}{Sq}

\renewcommand{\leq}{\leqslant}
\renewcommand{\geq}{\geqslant}

\def\Sph{\mathrm{Sph}}

\def\squier#1{\mathcal{S}(#1)}

\def\N{\mathbb{N}}

\def\T{\mathbf{T}}

\def\norm#1{||#1||}

\newcount\hh
\newcount\mm
\mm=\time
\hh=\time
\divide\hh by 60
\divide\mm by 60
\multiply\mm by 60
\mm=-\mm
\advance\mm by \time
\def\hhmm{\number\hh:\ifnum\mm<10{}0\fi\number\mm}

\definecolor{vert}{rgb}{0,0.45,0}
\definecolor{rouge}{rgb}{0.89,0.04,0.36}

\def\irr{\mathrm{Irr}}
\def\nf{\mathrm{NF}}
\def\cdg{\diamond}
\newcommand{\st}{\partial}

\newcommand{\er}[2]{\,_{#1}{#2}}
\newcommand{\re}[2]{{#2}_{#1}}
\newcommand{\ere}[2]{\,_{#1}{#2}_{#1}}
\newcommand{\ER}{\er{E}{R}}
\newcommand{\RE}{\re{E}{R}}
\newcommand{\ERE}{\ere{E}{R}}

\def\dar[#1,#2,#3]{\ar@<0.8ex>[#1] ^{#3} \ar@<-0.8ex>[#1] _{#2} }

\newcommand{\ddck}{\rotatebox[origin=c]{-90}{{$\vDash$}}}
\newcommand{\dck}[1]{#1^{\ddck}}
\newcommand{\dvck}[1]{#1^{\ddck,v}}
\newcommand{\cck}[1]{#1^{\oblong}}
\newcommand{\sqext}{E \rtimes \Gamma \cup \peiffer{\tck{E}}{S^\ast}}
\newcommand{\acy}[1]{\dvck{(E,S, E \rtimes #1 \cup \peiffer{\tck{E}}{S^\ast})}}
\newcommand{\bcy}[1]{\dck{(E,S,\squier{#1})}}
\newcommand{\sqconf}[1]{#1^{\curlyvee}}

\newcommand{\Saux}{S^{\amalg}}

\newcommand{\peiffer}[2]{\text{Peiff}(#1,#2)}

\newcommand{\dbpol}{\catego{DbPol}}

\newcommand{\Eo}{E^{\top (1)}}
\newcommand{\Ro}{R^{\ast (1)}}
\newcommand{\So}{S^{\ast (1)}}

\newcommand{\hatt}[1]{\widehat{\widetilde{#1}}}
\newcommand{\tilda}[1]{\widetilde{\widehat{#1}}}

\newcommand{\auteur}[3]{
\noindent
\begin{minipage}[t]{.45\textwidth}
\begin{flushright}
\textsc{#1} \\
{\footnotesize\textsf{#2}}
\end{flushright} 
\end{minipage}
\qquad
\begin{minipage}[t]{.45\textwidth}
#3
\end{minipage}
}

\newcommand{\ddott}[1]{
\mathord{
\begin{tikzpicture}[baseline = 0, scale=0.911]
	\draw[<-,thick,black] (0.08,-.3) to (0.08,.4);
      \node at (0.08,0.1) {$\color{black}\bullet$};
     \node at (0.08,.5) {$\scriptstyle{#1}$};
\end{tikzpicture}
}}

\newcommand{\udott}[1]{\mathord{
\begin{tikzpicture}[baseline = 0, scale=0.911]
	\draw[->,thick,black] (0.08,-.3) to (0.08,.4);
      \node at (0.08,0.05) {$\color{black}\bullet$};
   \node at (0.08,-.4) {$\scriptstyle{#1}$};
\end{tikzpicture}
}}

\newcommand{\crossdn}[2]{\mathord{
\begin{tikzpicture}[baseline = 0, scale=0.911]
	\draw[<-,thick,black] (0.28,-.3) to (-0.28,.4);
	\draw[<-,thick,black] (-0.28,-.3) to (0.28,.4);
   \node at (-0.28,.5) {$\scriptstyle{#1}$};
   \node at (0.28,.5) {$\scriptstyle{#2}$};
   \node at (.4,.05) {};
\end{tikzpicture}
}}

\newcommand{\sdrd}[1]{\mathord{
\begin{tikzpicture}[baseline = 0, scale=0.911]
  \draw[->,thick,black] (0.3,0) to (0.3,-.4);
	\draw[-,thick,black] (0.3,0) to[out=90, in=0] (0.1,0.4);
	\draw[-,thick,black] (0.1,0.4) to[out = 180, in = 90] (-0.1,0);
	\draw[-,thick,black] (-0.1,0) to[out=-90, in=0] (-0.3,-0.4);
	\draw[-,thick,black] (-0.3,-0.4) to[out = 180, in =-90] (-0.5,0);
  \draw[-,thick,black] (-0.5,0) to (-0.5,.4);
  \node at (-0.13,0.11) {$\bullet$};
   \node at (-0.5,.5) {$\scriptstyle{#1}$};
   \node at (0.5,0) { };
\end{tikzpicture}
}}

\newcommand{\surdd}[1]{\mathord{
\begin{tikzpicture}[baseline = 0, scale=0.911]
  \draw[->,thick,black] (0.3,0) to (0.3,.4);
	\draw[-,thick,black] (0.3,0) to[out=-90, in=0] (0.1,-0.4);
	\draw[-,thick,black] (0.1,-0.4) to[out = 180, in = -90] (-0.1,0);
	\draw[-,thick,black] (-0.1,0) to[out=90, in=0] (-0.3,0.4);
	\draw[-,thick,black] (-0.3,0.4) to[out = 180, in =90] (-0.5,0);
  \draw[-,thick,black] (-0.5,0) to (-0.5,-.4);
  \node at (-0.13,0.12) {$\bullet$};
   \node at (-0.5,-.5) {$\scriptstyle{#1}$};
   \node at (0.5,0) { };
\end{tikzpicture}
}}

\newcommand{\crossup}[2]{\mathord{
\begin{tikzpicture}[baseline = 0, scale=0.911]
	\draw[->,thick,black] (0.28,-.3) to (-0.28,.4);
	\draw[->,thick,black] (-0.28,-.3) to (0.28,.4);
   \node at (-0.28,-.4) {$\scriptstyle{#1}$};
   \node at (0.28,-.4) {$\scriptstyle{#2}$};
\end{tikzpicture}
}}

\newcommand{\cupdb}[2]{\mathord{
\begin{tikzpicture}[baseline = 0, scale=0.911]
	\draw[-,thick,black] (0.4,0.3) to[out=-90, in=0] (0.1,-0.1);
	\draw[-,thick,black] (0.1,-0.1) to[out = 180, in = -90] (-0.2,0.3);
    \node at (-0.2,.5) {$\scriptstyle{#1}$};
    \node at (0.5,0.5) {$\scriptstyle{#2}$};
\end{tikzpicture}
}}

\newcommand{\capdb}[2]{\mathord{
\begin{tikzpicture}[baseline = 0, scale=0.911]
	\draw[-,thick,black] (0.4,-0.1) to[out=90, in=0] (0.1,0.3);
	\draw[-,thick,black] (0.1,0.3) to[out = 180, in = 90] (-0.2,-0.1);
    \node at (-0.2,-.3) {$\scriptstyle{#1}$};
    \node at (0.5,-0.3) {$\scriptstyle{#2}$};
\end{tikzpicture}
}}

\newcommand{\cupr}[1]{\mathord{
\begin{tikzpicture}[baseline = 0, scale=0.911]
	\draw[<-,thick,black] (0.4,0.3) to[out=-90, in=0] (0.1,-0.1);
	\draw[-,thick,black] (0.1,-0.1) to[out = 180, in = -90] (-0.2,0.3);
    \node at (-0.2,.4) {$\scriptstyle{#1}$};
\end{tikzpicture}
}}

\newcommand{\capr}[1]{\mathord{
\begin{tikzpicture}[baseline = 0, scale=0.911]
	\draw[<-,thick,black] (0.4,-0.1) to[out=90, in=0] (0.1,0.3);
	\draw[-,thick,black] (0.1,0.3) to[out = 180, in = 90] (-0.2,-0.1);
    \node at (-0.2,-.2) {$\scriptstyle{#1}$};
\end{tikzpicture}
}}

\newcommand{\cupl}[1]{\mathord{
\begin{tikzpicture}[baseline = 0, scale=0.911]
	\draw[-,thick,black] (0.4,0.4) to[out=-90, in=0] (0.1,0);
	\draw[->,thick,black] (0.1,0) to[out = 180, in = -90] (-0.2,0.4);
    \node at (0.4,.5) {$\scriptstyle{#1}$};
\end{tikzpicture}
}}

\newcommand{\capl}[1]{\mathord{
\begin{tikzpicture}[baseline = 0,scale=0.911]
	\draw[-,thick,black] (0.4,0) to[out=90, in=0] (0.1,0.4);
	\draw[->,thick,black] (0.1,0.4) to[out = 180, in = 90] (-0.2,0);
    \node at (0.4,-.1) {$\scriptstyle{#1}$};
\end{tikzpicture}
}}

\newcommand{\sdld}[1]{\mathord{
\begin{tikzpicture}[baseline = 0, scale=0.911]
  \draw[-,thick,black] (0.3,0) to (0.3,-.4);
	\draw[-,thick,black] (0.3,0) to[out=90, in=0] (0.1,0.4);
	\draw[-,thick,black] (0.1,0.4) to[out = 180, in = 90] (-0.1,0);
	\draw[-,thick,black] (-0.1,0) to[out=-90, in=0] (-0.3,-0.4);
	\draw[-,thick,black] (-0.3,-0.4) to[out = 180, in =-90] (-0.5,0);
  \draw[->,thick,black] (-0.5,0) to (-0.5,.4);
  \node at (-0.14,-0.08) {$\bullet$};
   \node at (0.3,-.5) {$\scriptstyle{#1}$};
   \node at (0.5,0) { };
\end{tikzpicture}
}}

\newcommand{\suld}[1]{\mathord{
\begin{tikzpicture}[baseline = 0, scale=0.911]
  \draw[-,thick,black] (0.3,0) to (0.3,.4);
	\draw[-,thick,black] (0.3,0) to[out=-90, in=0] (0.1,-0.4);
	\draw[-,thick,black] (0.1,-0.4) to[out = 180, in = -90] (-0.1,0);
	\draw[-,thick,black] (-0.1,0) to[out=90, in=0] (-0.3,0.4);
	\draw[-,thick,black] (-0.3,0.4) to[out = 180, in =90] (-0.5,0);
  \draw[->,thick,black] (-0.5,0) to (-0.5,-.4);
  \node at (-0.13,0.12) {$\bullet$};
   \node at (0.3,.5) {$\scriptstyle{#1}$};
   \node at (0.5,0) { };
\end{tikzpicture}
}}

\newcommand{\dpd}[1]{\mathord{
\begin{tikzpicture}[baseline = 0, scale=0.911]
	\draw[<-,thick,black] (0.08,-.3) to (0.08,.4);
      \node at (0.08,0.1) {$\color{black}\bullet$};
     \node at (0.08,.5) {$\scriptstyle{#1}$};
\end{tikzpicture}
}}

\newcommand{\upd}[1]{\mathord{
\begin{tikzpicture}[baseline = 0, scale=0.911]
	\draw[->,thick,black] (0.08,-.3) to (0.08,.4);
      \node at (0.08,0.1) {$\color{black}\bullet$};
     \node at (0.08,.5) {$\scriptstyle{#1}$};
\end{tikzpicture}
}}

\newcommand{\cuprdl}[2]{\mathord{
\begin{tikzpicture}[baseline = 0, scale=0.911]
	\node at (-0.16,0.15) {$\color{black}\bullet$};
	\draw[<-,thick,black] (0.4,0.4) to[out=-90, in=0] (0.1,-0.1);
	\draw[-,thick,black] (0.1,-0.1) to[out = 180, in = -90] (-0.2,0.4);
    \node at (-0.2,.5) {$\scriptstyle{#1}$};
  \node at (0.3,-0.15) { };
      \node at (-0.36,0.15) {$\scriptstyle{#2}$};
\end{tikzpicture}
}}

\newcommand{\cuprdr}[2]{\mathord{
\begin{tikzpicture}[baseline = 0, scale=0.911]
	\draw[<-,thick,black] (0.4,0.4) to[out=-90, in=0] (0.1,-0.1);
	\draw[-,thick,black] (0.1,-0.1) to[out = 180, in = -90] (-0.2,0.4);
    \node at (-0.2,.5) {$\scriptstyle{#1}$};
      \node at (0.37,0.15) {$\color{black}\bullet$};
      \node at (0.57,0.15) {$\scriptstyle{#2}$};
\end{tikzpicture}
}}

\newcommand{\caprdl}[2]{\mathord{
\begin{tikzpicture}[baseline = 0, scale=0.911]
	\draw[<-,thick,black] (0.4,-0.1) to[out=90, in=0] (0.1,0.4);
	\draw[-,thick,black] (0.1,0.4) to[out = 180, in = 90] (-0.2,-0.1);
    \node at (-0.2,-.2) {$\scriptstyle{#1}$};
  \node at (0.3,0.5) { };
      \node at (-0.14,0.25) {$\color{black}\bullet$};
      \node at (-.34,0.25) {$\scriptstyle{#2}$};
\end{tikzpicture}
}}

\newcommand{\caprdr}[2]{\mathord{
\begin{tikzpicture}[baseline = 0, scale=0.911]
	\node at (0.32,0.25) {$\color{black}\bullet$};
	\draw[<-,thick,black] (0.4,-0.1) to[out=90, in=0] (0.1,0.4);
	\draw[-,thick,black] (0.1,0.4) to[out = 180, in = 90] (-0.2,-0.1);
    \node at (-0.2,-.2) {$\scriptstyle{#1}$};
  \node at (0.3,0.5) { };
      \node at (0.52,0.25) {$\scriptstyle{#2}$};
\end{tikzpicture}
}}

\newcommand{\cupldr}[2]{\mathord{
\begin{tikzpicture}[baseline = 0, scale=0.911]
	\draw[-,thick,black] (0.4,0.4) to[out=-90, in=0] (0.1,-0.1);
	\draw[->,thick,black] (0.1,-0.1) to[out = 180, in = -90] (-0.2,0.4);
    \node at (-0.2,.5) {$\scriptstyle{#1}$};
  \node at (0.3,-0.15) { };
      \node at (0.37,0.15) {$\color{black}\bullet$};
      \node at (0.57,0.15) {$\scriptstyle{#2}$};
\end{tikzpicture}
}}

\newcommand{\cupldl}[2]{\mathord{
\begin{tikzpicture}[baseline = 0, scale=0.911]
	\node at (-0.16,0.15) {$\color{black}\bullet$};
	\draw[-,thick,black] (0.4,0.4) to[out=-90, in=0] (0.1,-0.1);
	\draw[->,thick,black] (0.1,-0.1) to[out = 180, in = -90] (-0.2,0.4);
    \node at (-0.2,.5) {$\scriptstyle{#1}$};
  \node at (0.3,-0.15) { };
      \node at (-0.36,0.15) {$\scriptstyle{#2}$};
\end{tikzpicture}
}}

\newcommand{\capldl}[2]{\mathord{
\begin{tikzpicture}[baseline = 0, scale=0.911]
	\draw[-,thick,black] (0.4,-0.1) to[out=90, in=0] (0.1,0.4);
	\draw[->,thick,black] (0.1,0.4) to[out = 180, in = 90] (-0.2,-0.1);
    \node at (-0.2,-.2) {$\scriptstyle{#1}$};
  \node at (0.3,0.5) { };
      \node at (-0.14,0.25) {$\color{black}\bullet$};
      \node at (-.34,0.25) {$\scriptstyle{#2}$};
\end{tikzpicture}
}}

\newcommand{\capldr}[2]{\mathord{
\begin{tikzpicture}[baseline = 0, scale=0.911]
	\node at (0.32,0.25) {$\color{black}\bullet$};
	\draw[-,thick,black] (0.4,-0.1) to[out=90, in=0] (0.1,0.4);
	\draw[->,thick,black] (0.1,0.4) to[out = 180, in = 90] (-0.2,-0.1);
    \node at (-0.2,-.2) {$\scriptstyle{#1}$};
  \node at (0.3,0.5) { };
      \node at (0.52,0.25) {$\scriptstyle{#2}$};
\end{tikzpicture}
}}

\newcommand{\tdcrosslr}[2]{\mathord{
\begin{tikzpicture}[baseline = 0,scale=0.611]
	\draw[<-,thick,black] (0.3,.5) to (-0.3,-.5);
	\draw[-,thick,black] (-0.2,.2) to (0.2,-.3);
        \draw[-,thick,black] (0.2,-.3) to[out=130,in=180] (0.5,-.4);
        \draw[-,thick,black] (0.5,-.4) to[out=0,in=270] (0.8,.5);
        \draw[-,thick,black] (-0.2,.2) to[out=130,in=0] (-0.5,.5);
        \draw[->,thick,black] (-0.5,.5) to[out=180,in=-270] (-0.8,-.5);        
      \draw[<-,thick,black] (-0.3,-.5) to (0.3,-1.5);
         \draw[-,thick,black] (-0.5,-1.5) to[out=180,in=-90] (-0.8,-.5);
         \draw[-,thick,black] (-0.2,-1.2) to[out=230,in=0] (-0.5,-1.5);
         \draw[-,thick,black] (-0.2,-1.2) to (0.2,-0.7); 
         \draw[-,thick,black] (0.2,-0.7) to[out=50,in=180] (0.5,-0.6);
          \draw[->,thick,black] (0.5,-0.6) to[out=0,in=90] (0.7,-1.5);
\end{tikzpicture}
}}

\newcommand{\tdcrosslra}[2]{\mathord{
\begin{tikzpicture}[baseline = 0,scale=0.611]
	\draw[<-,thick,black] (0.3,.5) to (-0.3,-.5);
	\draw[-,thick,black] (-0.2,.2) to (0.2,-.3);
        \draw[-,thick,black] (0.2,-.3) to[out=130,in=180] (0.5,-.4);
        \draw[-,thick,black] (0.5,-.4) to[out=0,in=270] (0.8,.5);
        \draw[-,thick,black] (-0.2,.2) to[out=130,in=0] (-0.5,.5);
        \draw[->,thick,black] (-0.5,.5) to[out=180,in=-270] (-0.8,-.5);        
      \draw[<-,thick,black] (-0.3,-.5) to (0.3,-1.5);
         \draw[-,thick,black] (-0.5,-1.5) to[out=180,in=-90] (-0.8,-.5);
         \draw[-,thick,black] (-0.2,-1.2) to[out=230,in=0] (-0.5,-1.5);
         \draw[-,thick,black] (-0.2,-1.2) to (0.2,-0.7); 
         \draw[-,thick,black] (0.2,-0.7) to[out=50,in=180] (0.5,-0.6);
          \draw[->,thick,black] (0.5,-0.6) to[out=0,in=90] (0.7,-1.5);
          	\draw[-,thick,black] (1.4,0.5) to[out=90, in=0] (1.1,0.9);
	\draw[->,thick,black] (1.1,0.9) to[out = 180, in = 90] (0.8,0.5);
	\draw[-,thick,black] (1.4,0.5) to (1.4,-1.2);
\end{tikzpicture}
}}

\newcommand{\tdcrosslrb}[1]{\mathord{
\begin{tikzpicture}[baseline = 0,scale=0.911]
    \draw[->,thick,black] (0,0) to (0,0.4);
	\draw[-,thick,black] (0.8,0) to[out=90, in=0] (0.5,0.4);
	\draw[->,thick,black] (0.5,0.4) to[out = 180, in = 90] (0.2,0);
\end{tikzpicture}
}}

\newcommand{\tdcrosslrc}[2]{\mathord{
\begin{tikzpicture}[baseline = 0,scale=0.611]
	\draw[<-,thick,black] (0.3,.5) to (-0.3,-.5);
	\draw[-,thick,black] (-0.2,.2) to (0.2,-.3);
        \draw[-,thick,black] (0.2,-.3) to[out=130,in=180] (0.5,-.4);
        \draw[-,thick,black] (-0.2,.2) to[out=130,in=0] (-0.5,.5);
        \draw[->,thick,black] (-0.5,.5) to[out=180,in=-270] (-0.8,-.5);        
      \draw[<-,thick,black] (-0.3,-.5) to (0.3,-1.5);
         \draw[-,thick,black] (-0.5,-1.5) to[out=180,in=-90] (-0.8,-.5);
         \draw[-,thick,black] (-0.2,-1.2) to[out=230,in=0] (-0.5,-1.5);
         \draw[-,thick,black] (-0.2,-1.2) to (0.2,-0.7); 
         \draw[-,thick,black] (0.2,-0.7) to[out=50,in=180] (0.5,-0.6);
          \draw[->,thick,black] (0.5,-0.6) to[out=0,in=90] (0.7,-1.5);
\end{tikzpicture}
}}

\newcommand{\doublecroisementhaut}[2]{\begin{tikzpicture} \begin{scope} [ x = 10pt, y = 10pt, join = round, cap = round, thick, scale=1.5 ] \draw[<-] (0.00,1.50)--(0.00,1.25);
\draw[<-] (1.00,1.50)--(1.00,1.25) ; 
\draw (0.00,1.25)--(1.00,0.75) (1.00,1.25)--(0.00,0.75) ; \draw (0.00,0.75)--(0.00,0.50) (1.00,0.75)--(1.00,0.50) ; \draw (0.00,0.50)--(1.00,0.00) (1.00,0.50)--(0.00,0.00) ; \draw (0.00,0.00)--(0.00,-0.25) (1.00,0.00)--(1.00,-0.25) ; 
\node at (0.00,-0.75) {$\scriptstyle{#1}$};
\node at (1.00,-0.75) {$\scriptstyle{#2}$};
\end{scope} \end{tikzpicture}}

\newcommand{\ybgauchehaut}[3]{
\begin{tikzpicture} \begin{scope} [ x = 10pt, y = 10pt, join = round, cap = round, thick , scale=1.5 ] \draw[<-] (0.00,2.25)--(0.00,2.00) ;
\draw[<-] (1.00,2.25)--(1.00,2.00);
\draw[<-] (2.00,2.25)--(2.00,1.75) ; \draw (0.00,2.00)--(1.00,1.50) (1.00,2.00)--(0.00,1.50) ;  \draw (0.00,1.50)--(0.00,1.00) (1.00,1.50)--(1.00,1.25) (2.00,1.75)--(2.00,1.25) ;  \draw (1.00,1.25)--(2.00,0.75) (2.00,1.25)--(1.00,0.75) ; \draw (0.00,1.00)--(0.00,0.50) (1.00,0.75)--(1.00,0.50) (2.00,0.75)--(2.00,0.25) ; \draw (0.00,0.50)--(1.00,0.00) (1.00,0.50)--(0.00,0.00) ;  \draw (0.00,0.00)--(0.00,-0.25) (1.00,0.00)--(1.00,-0.25) (2.00,0.25)--(2.00,-0.25) ; 
\node at (0.00,-0.75) {$\scriptstyle{#1}$};
\node at (1.00,-0.75) {$\scriptstyle{#2}$};
\node at (2.00,-0.75) {$\scriptstyle{#3}$};
\end{scope} 
\end{tikzpicture}}

\newcommand{\ybdroithaut}[3]{
\begin{tikzpicture} 
\begin{scope} [ x = 10pt, y = 10pt, join = round, cap = round, thick , scale=1.5] \draw[<-] (0.00,2.00)--(0.00,1.50);
\draw[<-] (1.00,2.00)--(1.00,1.75) ;
\draw[<-] (2.00,2.00)--(2.00,1.75) ;  \draw (1.00,1.75)--(2.00,1.25) (2.00,1.75)--(1.00,1.25) ; \draw (0.00,1.50)--(0.00,1.00) (1.00,1.25)--(1.00,1.00) (2.00,1.25)--(2.00,0.75) ; \draw (0.00,1.00)--(1.00,0.50) (1.00,1.00)--(0.00,0.50) ;  \draw (0.00,0.50)--(0.00,0.00) (1.00,0.50)--(1.00,0.25) (2.00,0.75)--(2.00,0.25) ;  \draw (1.00,0.25)--(2.00,-0.25) (2.00,0.25)--(1.00,-0.25) ; \draw (0.00,0.00)--(0.00,-0.50) (1.00,-0.25)--(1.00,-0.50) (2.00,-0.25)--(2.00,-0.50) ; 
\node at (0.00,-0.75) {$\scriptstyle{#1}$};
\node at (1.00,-0.75) {$\scriptstyle{#2}$};
\node at (2.00,-0.75) {$\scriptstyle{#3}$};
\end{scope}
\end{tikzpicture}}

\newcommand{\doublecroisementbas}[2]{\begin{tikzpicture} \begin{scope} [ x = 10pt, y = 10pt, join = round, cap = round, thick, scale=1.5 ] \draw (0.00,1.50)--(0.00,1.25);
\draw (1.00,1.50)--(1.00,1.25) ; 
\draw (0.00,1.25)--(1.00,0.75) (1.00,1.25)--(0.00,0.75) ; \draw (0.00,0.75)--(0.00,0.50) (1.00,0.75)--(1.00,0.50) ; \draw (0.00,0.50)--(1.00,0.00) (1.00,0.50)--(0.00,0.00) ; 
\draw[->] (0.00,0.00)--(0.00,-0.25) ;
\draw[->] (1.00,0.00)--(1.00,-0.25) ; 
\node at (0.00,-0.75) {$\scriptstyle{#1}$};
\node at (1.00,-0.75) {$\scriptstyle{#2}$};
\end{scope} \end{tikzpicture}}

\newcommand{\ybgauchebas}[3]{
\begin{tikzpicture} \begin{scope} [ x = 10pt, y = 10pt, join = round, cap = round, thick , scale=1.5 ] 
\draw (0.00,2.25)--(0.00,2.00) ;
\draw (1.00,2.25)--(1.00,2.00);
\draw (2.00,2.25)--(2.00,1.75) ; \draw (0.00,2.00)--(1.00,1.50) (1.00,2.00)--(0.00,1.50) ;  \draw (0.00,1.50)--(0.00,1.00) (1.00,1.50)--(1.00,1.25) (2.00,1.75)--(2.00,1.25) ;  \draw (1.00,1.25)--(2.00,0.75) (2.00,1.25)--(1.00,0.75) ; \draw (0.00,1.00)--(0.00,0.50) (1.00,0.75)--(1.00,0.50) (2.00,0.75)--(2.00,0.25) ; \draw (0.00,0.50)--(1.00,0.00) (1.00,0.50)--(0.00,0.00) ;  \draw[->] (0.00,0.00)--(0.00,-0.25) ;
\draw[->] (1.00,0.00)--(1.00,-0.25) ;
\draw[->] (2.00,0.25)--(2.00,-0.25) ; 
\node at (0.00,-0.75) {$\scriptstyle{#1}$};
\node at (1.00,-0.75) {$\scriptstyle{#2}$};
\node at (2.00,-0.75) {$\scriptstyle{#3}$};
\end{scope} 
\end{tikzpicture}}

\newcommand{\ybdroitbas}[3]{
\begin{tikzpicture} 
\begin{scope} [ x = 10pt, y = 10pt, join = round, cap = round, thick , scale=1.5] 
\draw (0.00,2.00)--(0.00,1.50);
\draw (1.00,2.00)--(1.00,1.75) ;
\draw (2.00,2.00)--(2.00,1.75) ;  \draw (1.00,1.75)--(2.00,1.25) (2.00,1.75)--(1.00,1.25) ; \draw (0.00,1.50)--(0.00,1.00) (1.00,1.25)--(1.00,1.00) (2.00,1.25)--(2.00,0.75) ; \draw (0.00,1.00)--(1.00,0.50) (1.00,1.00)--(0.00,0.50) ;  \draw (0.00,0.50)--(0.00,0.00) (1.00,0.50)--(1.00,0.25) (2.00,0.75)--(2.00,0.25) ;  \draw (1.00,0.25)--(2.00,-0.25) (2.00,0.25)--(1.00,-0.25) ; 
\draw[->] (0.00,0.00)--(0.00,-0.50) ;
\draw[->] (1.00,-0.25)--(1.00,-0.50); 
\draw[->] (2.00,-0.25)--(2.00,-0.50) ; 
\node at (0.00,-0.75) {$\scriptstyle{#1}$};
\node at (1.00,-0.75) {$\scriptstyle{#2}$};
\node at (2.00,-0.75) {$\scriptstyle{#3}$};
\end{scope}
\end{tikzpicture}}

\newcommand{\doubleidentitehaut}[2]{
\begin{tikzpicture}
\begin{scope} [ x = 10pt, y = 10pt, join = round, cap = round, thick , scale=1.5]
\draw[->] (0,-0.25) -- (0,1.25);
\draw[->] (1,-0.25) -- (1,1.25);
\node at (0.00,-0.75) {$\scriptstyle{#1}$};
\node at (1.00,-0.75) {$\scriptstyle{#2}$};
\end{scope}
\end{tikzpicture}}

\newcommand{\doubleidentitebas}[2]{
\begin{tikzpicture}
\begin{scope} [ x = 10pt, y = 10pt, join = round, cap = round, thick , scale=1.5]
\draw[<-] (0,-0.25) -- (0,1.25);
\draw[<-] (1,-0.25) -- (1,1.25);
\node at (0.00,-0.75) {$\scriptstyle{#1}$};
\node at (1.00,-0.75) {$\scriptstyle{#2}$};
\end{scope}
\end{tikzpicture}}

\begin{document}
\thispagestyle{empty}

\begin{center}

\begin{doublespace}
\begin{huge}
{\scshape Coherent confluence modulo relations}
\end{huge}

\vskip+2pt

\begin{huge}
{\scshape and double groupoids}
\end{huge}

\bigskip
\hrule height 1.5pt 
\bigskip

\vskip+5pt

\begin{Large}
{\scshape Benjamin Dupont -- Philippe Malbos}
\end{Large}
\end{doublespace}

\vfill

\vskip+25pt

\begin{small}\begin{minipage}{14cm}
\noindent\textbf{Abstract --}
A coherent presentation of an $n$-category is a presentation by generators, relations and relations among relations. 
Confluent and terminating rewriting systems generate coherent presentations, whose relations among relations are defined by confluence diagrams of critical branchings. This article introduces a procedure to compute coherent presentations when the rewrite relations are defined modulo a set of axioms. Our coherence results are formulated using the structure of $n$-categories enriched in double groupoids, whose horizontal cells represent rewriting paths, vertical cells represent the congruence generated by the axioms and square cells represent coherence cells induced by diagrams of confluence modulo. We illustrate our constructions on rewriting systems modulo commutation relations in commutative monoids, isotopy relations in pivotal monoidal categories, and inverse relations in groups.

\smallskip


\smallskip

\smallskip\noindent\textbf{M.S.C. 2010 -- Primary:} 68Q42, 18D05. \textbf{Secondary:} 18D20, 68Q40.
\end{minipage}\end{small}
\end{center}

\vspace{-15pt}

\begin{center}
\begin{small}\begin{minipage}{12cm}
\renewcommand{\contentsname}{}
\setcounter{tocdepth}{2}
\tableofcontents
\end{minipage}
\end{small}
\end{center}

\section{Introduction}

Algebraic rewriting aims at providing constructive methods based on rewriting theory to prove properties on higher algebraic structures given through generators and relations. This approach has been used to compute linear bases \cite{Dupont21,DupontEbertLauda21}, coherent presentations \cite{GaussentGuiraudMalbos15,HageMalbos17} or higher-syzygies \cite{GuiraudMalbos12advances,GuiraudHoffbeckMalbos19}. 
For diagrammatic algebras, such as Temperley-Lieb algebras \cite{TemperleyLieb71}, Brauer algebras \cite{Brauer37}, Birman-Wenzl algebras \cite{BirmanWenzl89}, Jones' planar algebras \cite{Jones99}, Khovanov-Lauda-Rouquier algebras and the associated categorifications of quantum groups \cite{KhovanovLauda08,Rouquier08}, and Khovanov's categorification of the Heisenberg algebra \cite{Khovanov2010,Brundan2017}, the presentations have a great complexity due to the number of generators and relations. 
These structures have a huge number of relations, leading to a combinatorial explosion of cases in the proof of their rewriting  properties.
However, many of these relations are inherent to the algebraic structure itself. For instance, some of the above algebras can be interpreted as linear $2$-categories with an additional pivotal structure, whose string diagrams representing $2$-cells are drawn up to isotopy. The inherent relations create useless obstructions to rewriting properties, and the classical rewriting approach is not efficient to study these structures. It is thus necessary to define the rewriting rules modulo some relations and to study the properties of rewriting modulo these relations. 

This work is part of a larger project that aims to develop methods of rewriting modulo in algebraic rewriting contexts. Algebraic rewriting modulo has already been applied to the computation of linear bases of linear~$2$-categories that categorify associative algebras, equipped with a system of idempotents, see \cite{Lauda12}. Linear bases of the spaces of morphisms of a linear~$2$-category are relevant to prove the isomorphism between its Grothendieck group and the categorified algebra. Such bases can be computed from presentations that are confluent modulo a non-oriented part of the relations \cite{DUP19}.

This article presents a rewriting categorical approach based on rewriting modulo in order to compute coherent presentations modulo the inherent structure, where coherence encodes generating relations among relations modulo.
Our construction constitutes the first step in the computation of cofibrant replacements of these structures by polygraphic resolutions.

\subsection*{Coherence by rewriting}

A \emph{syzygy} of a presentation of a higher-dimensional (strict globular) $n$-category is a relation among relations, \emph{i.e.} an equality between two relations that can be deduced from the presentation. For a given presentation, we would like to compute all the syzygies by making explicit some of them that generate all the others. The starting presentation extended by a generating family of syzygies is called a \emph{coherent presentation}, which we will describe explicitly from now on. First, generators of higher categories are described by \emph{polygraphs} \cite{Burroni93}, also called \emph{computads} \cite{Street76,Power91}. As a directed graph generates a $1$-category, an \emph{$(n-1)$-polygraph} $P$ generates an $(n-1)$-category. It is constructed by adjunction of generating $k$-cells, for $0\leq k \leq n-1$, whose source and target are composites of generating $(k-1)$-cells. The relations are defined by a set $R$ of \emph{rules}, described by generating $n$-cells that relate composites of $(n-1)$-cells in the free $(n-1)$-category $P^\ast$ on $P$. The data $(P,R)$ is called an \emph{$n$-polygraph}, including both generators and rules, that present a $(n-1)$-category, whose underlying $(n-2)$-category is freely generated by the $k$-cells of $P$, for $0\leq k \leq n-2$, and the $(n-1)$-cells are subject to the relations in $R$.
A rewriting path with respect to $R$ is interpreted by an $n$-cell in the free $n$-category $(P,R)^\ast$ generated by the $n$-polygraph~$(P,R)$.

A \emph{coherent extension} of the $n$-polygraph $(P,R)$ is a set $\Gamma$ of generating $(n+1)$-cells on the free $(n,n-1)$-category $\tck{(P,R)}$ on $(P,R)$ that generates all the syzygies of the presentation. In other words, any $n$-sphere in the quotient of the free $(n,n-1)$-category $\tck{(P,R)}$ by the congruence generated by $\Gamma$ is trivial. 
We say also that the extension $\Gamma$ is \emph{acyclic}.
When the $n$-polygraph $(P,R)$ is \emph{convergent}, that is confluent and terminating, it can be extended into a coherent presentation by $(n+1)$-cells keeping track of confluence diagrams of \emph{critical branchings}, which correspond to minimal overlappings of rules in $R$. Explicitly, any chosen family of \emph{generating confluences}, made of $(n+1)$-cells
\[
\xymatrix @C=2.5em @R=0.5em {
& v
	\ar @{.>} @/^1ex/ [dr] ^{f'}
	\ar@2 []!<0pt,-10pt>;[dd]!<0pt,10pt> ^-{A_{f,g}}
\\
u
	\ar @/^1ex/ [ur] ^{f}
	\ar @/_1ex/ [dr] _{g}
& &
w
\\
& v'
\ar @{.>}@/_1ex/ [ur] _{g'}
}
\]
given by a confluence diagram for each critical branching $(f,g)$, extends the $n$-polygraph $(P,R)$ to a coherent presentation. This construction was initiated by Squier in~\cite{Squier94} for monoids and generalized to $n$-categories in~\cite{GuiraudMalbos09,GuiraudMalbos18}. 

\subsection*{Coherence modulo by rewriting modulo}

This article extends these constructions to $n$-categories whose underlying $(n-1)$-categories are not free, using rewriting systems defined modulo a set of fixed relations. Rewriting modulo appears when studied reductions are defined modulo the axioms of an ambiant algebraic structure, eg. rewriting in commutative, groupoidal, linear, pivotal, weak structures. Furthermore, rewriting modulo makes the property of confluence easier to prove.
Indeed, the family of critical branchings that should be considered in the analysis of local confluence is reduced, and the non-orientation of certain relations allows more flexibility when reaching confluence.
The most naive approach of rewriting modulo is to consider the system $\ERE$ of rules of $R$ on congruence classes modulo $E$. This approach works for associative and commutative theories. However, it appears inefficient in general for the analysis of confluence. Indeed, the reducibility of an equivalence class needs to explore all the class, hence it requires all equivalence classes to be finite.
Another way has been considered by Huet in \cite{Huet80}, where rewriting paths involve only oriented rules and no equivalence steps, and the confluence property is formulated modulo an equivalence relation. However, for algebraic rewriting systems this approach is  too restrictive for computations~\cite{JouannaudLi12}. 
Peterson and Stickel introduced in~\cite{PetersonStickel81} an extension of Knuth-Bendix's completion procedure~\cite{KnuthBendix70}, to reach confluence of a rewriting system modulo an equational theory, for which a finite complete unification algorithm is known. They applied their procedure to rewriting systems modulo axioms of associativity and commutativity, in order to rewrite in free commutative groups, commutative unitary rings, and distributive lattices.
Jouannaud and Kirchner enlarged this approach in \cite{JouannaudKirchner84} with the definition of rewriting properties for any rewriting system modulo~$S$ such that $R\subseteq S \subseteq \ERE$.
They also proved a critical branching lemma and developed a completion procedure for a rewriting system modulo $\ER$, whose one-step reductions consist in application of a rule in $R$ using $E$-matching. Their completion procedure is based on a finite $E$-unification algorithm. Bachmair and Dershowitz developed a generalization of Jouannaud-Kirchner's completion procedure using inference rules \cite{BachmairDershowitz89}. Several other approaches have also been studied for term rewriting systems modulo to deal with various equational theories, see \cite{Viry95, Marche96}.

In Section~\ref{S:PolygraphsModulo}, we give a polygraphic formulation of rewriting modulo that is the main notion of this article:
\begin{quote}
{\bf Definition \ref{SSS:PolygraphsModulo}.} 
An \emph{$n$-polygraph modulo} is a data~$\Pr:=(R,E,S)$ made of
\begin{enumerate}[{\bf i)}]
\item an $n$-polygraph $R$, whose generating $n$-cells are called \emph{primary rules},
\item an $n$-polygraph $E$ such that $E_{\leq (n-2)} = R_{\leq (n-2)}$ and $E_{n-1} \subseteq R_{n-1}$, whose generating $n$-cells are called \emph{modulo rules},
\item a cellular extension $S$ of $R_{n-1}^\ast$ satisfying the following condition:
\[
R_n
\subseteq
S 
\subseteq 
\ERE.
\]
\end{enumerate}
\end{quote}
This means that $S$ contains all the generating $n$-cells of $R_n$ and that every generating $n$-cell in~$S$ can be written $(e,f,e')$ with $e,e'\in\tck{E}_n$ and $f$ in $R_n^{\ast (1)}$.
In this way, a presentation modulo is split into two parts: oriented rules defined by $R_n$ and non-oriented equations defined by~$E_n$. 
In Section~\ref{S:PolygraphsModulo}, we define the termination property for polygraphs modulo and we recall from \cite{Huet80} Huet's principle of double induction. We define confluence properties for polygraphs modulo following Huet \cite{Huet80} and Jouannaud-Kirchner \cite{JouannaudKirchner84}. We give a completion procedure for the $n$-polygraph modulo $(R,E,\ER)$ in terms of critical branchings that implements inference rules for completion modulo given by Bachmair and Dershowitz in \cite{BachmairDershowitz89}. 
Proofs of confluence modulo involve square diagrams whose horizontal edges correspond to rewriting paths with respect to $S$ and vertical edges correspond to congruences with respect to $E$. In this work, we extend this interpretation by considering faces which describe generating syzygies modulo. To this purpose, we formulate the notion of coherence modulo for an~$n$-category using the structure of $n$-categories enriched in double groupoids.

\subsection*{Confluence modulo and double categories}
 The notion of double category was first introduced by Ehresmann in~\cite{Ehresmann63} as an internal category in the category of categories. The notion of double groupoids, that is internal groupoids in the category of groupoids, and its higher-dimensional versions have been widely used in homotopy theory, \cite{BrownSpencer76, BrownHiggins78}, see \cite{BrownHigginsSivera11} and \cite{Brown04} for a complete account on the theory. A double category gives four related categories: a vertical and a horizontal one, and two categories of squares with either vertical or horizontal cells as sources and targets.  
A square cell $A$ is pictured by
\[
\xymatrix@R=2em @C=2em{ 
u 
\ar[r] ^{f} _{}="src" 
\ar[d] _{e} 
& 
v
\ar[d]^{e'} \\
 u' 
 \ar[r] _{g} ^{}="tgt" 
 &
 v' 
 \ar@2 "src"!<0pt,-7pt>;"tgt"!<0pt,7pt>  ^{A}
}
\]
where $f,g$ are horizontal cells, and $e,e'$ are vertical cells.
Section~\ref{S:PolygraphsModulo} formulates the property of confluence for an $n$-polygraph modulo $\Pr=(R,E,S)$ using cubical diagrams made of $n$-cells of the free $n$-category $(R_{\leq n-1},S)^\ast$ as horizontal edges, and $n$-cells of the free $(n,n-1)$-category $\tck{(R_{\leq n-1},E)}$ generated by $E$ as vertical edges. Thus, we formulate our coherence results in the structure of $(n-1)$-category enriched in double groupoids, obtained by adjoining \emph{coherence cells modulo}, that are given by square cells filling the confluence diagrams of polygraphs modulo. 
We define a \emph{branching} of $\Pr$ as a triple~$(f,e,g)$, where $f$ and $g$ are $n$-cells of $S^\ast$ and $e$ is an $n$-cell of $\tck{E}$, that we picture as follows
\[
\raisebox{0.55cm}{
\xymatrix @R=2em @C=2em{
u
  \ar[r] ^-{f}
  \ar[d] _-{e}
&
u' 
\\
v
  \ar[r] _-{g}
&
v'
}}
\]
Such a branching is \emph{confluent modulo $E$} if there exist $n$-cells $f'$ and $g'$ in $S^\ast$ and an $n$-cell $e'$ in $\tck{E}$ as in the following diagram:
\[
\raisebox{0.55cm}{
\xymatrix @R=2em @C=2em{
u
  \ar[r] ^-{f}
  \ar[d] _-{e}
&
u' 
  \ar@{.>}[r] ^-{f'} 
& 
u''
  \ar@{.>}[d] ^-{e'}
\\
v
  \ar[r] _-{g}
&
v'
  \ar@{.>}[r] _-{g'} 
&
w''
}}
\]

\subsection*{Coherent confluence modulo}

The notion of coherent presentation modulo that we introduce is based on an adaptation of the structure of polygraphs to the cubical setting. 
In Section~\ref{S:DoubleCoherentPresentation}, we define a \emph{double $(n+2,n)$-polygraph} as a data~$P=(P^v,P^h,P^s)$ made of two $(n+1)$-polygraphs $P^{v}$ and $P^{h}$ with the same underlying $n$-polygraph, and a square extension $P^{s}$ made of generating squares of the form
\[
\raisebox{0.55cm}{
\xymatrix @R=2em @C=2em{
u
  \ar[r] ^-{f}
  \ar[d] _-{e}
&
u' 
  \ar[d] ^-{e'}
\\
v
  \ar[r] _-{g}
&
v'
}}
\]
where $f,g$ are $(n+1)$-cells of the free $(n+1,n)$-category $\tck{(P^{v})}$ and $e,e'$ are $(n+1)$-cells of the free $(n+1,n)$-category $\tck{(P^{h})}$.
We define a \emph{double coherent presentation} of an $n$-category $\Cr$ as a double $(n+2,n)$-polygraph $P=(P^{v},P^{h},P^{s})$ such that the coproduct of the polygraphs $P^{v}$ and $P^{h}$ is a presentation of the category $\Cr$ and the square extension $P^{s}$ satisfies an acyclicity property defined in~\eqref{SSS:acyclicity}.

Section~\ref{S:CoherentConfluenceModulo} introduces the notion of confluence of an $n$-polygraph modulo $\Pr=(R,E,S)$ with respect to a square extension $\Gamma$ of $\Pr$, that is of the pair of $n$-categories $(\tck{E},S^\ast)$. We say that $\Pr$ is \emph{$\Gamma$-confluent} if, for every $S$-branching $(f,e,g)$, there exist $n$-cells~$f',g'$ in~$S^*$, $e'$ in~$\tck{E}$ and an~$(n+1)$-cell
\[
\raisebox{0.55cm}{
\xymatrix @R=2em @C=2em{
u
  \ar[r] ^-{f}
  \ar[d] _-{e}
&
u' 
  \ar@{.>}[r] ^-{f'} 
  \ar@2 []!<0pt,-12pt>; [d]!<0pt,+12pt> ^-{A}
& 
w
  \ar@{.>}[d] ^-{e'}
\\
v
  \ar[r] _-{g}
&
v'
  \ar@{.>}[r] _-{g'} 
&
w'
}}
\]
in the free $(n-1)$-category enriched in double categories generated by the square extension $\Gamma$ and an action of $E$ on $\Gamma$ as defined in Subsection~\ref{SS:CoherentNewmanModulo}.
We deduce coherent confluence of an $n$-polygraph modulo from local coherent confluence properties. In particular, we prove a coherent version of Newman's lemma modulo: 

\begin{quote}
{\bf Theorem \ref{T:ConfluenceTheorem}.}
\emph{
Let $\Pr$ be a terminating $n$-polygraph modulo, and $\Gamma$ be a square extension of $\Pr$. If $\Pr$ is locally $\Gamma$-confluent, then it is $\Gamma$-confluent.
}
\end{quote}
Finally, we prove a coherent version of the critical branching lemma modulo, deducing coherent local confluence from coherent confluence of some critical branchings modulo:

\begin{quote}
{\bf Theorem \ref{T:CoherentCriticalBranchingTheorem}}
\emph{
Let $\Pr$ be a terminating $n$-polygraph modulo, and $\Gamma$ be a square extension of $\Pr$. Then $\Pr$ is $\Gamma$-locally confluent if, and only if, it is $\Gamma$-critically confluent.
}
\end{quote}

\subsection*{Coherent completion modulo}

Section~\ref{S:CoherentCompletionModulo} presents several ways to extend a presentation of an $(n-1)$-category by a polygraph modulo into a double coherent presentation of this category. 
As stated above, a convergent $n$-polygraph~$E$ extends to a globular coherent presentation of the $(n-1)$-category it presents by adjunction of a chosen family of generating confluences of $E$. We define a family of \emph{generating confluences} of an $n$-polygraph modulo $\Pr=(R,E,S)$ as a square extension of $\Pr$ made of square $(n+1)$-cells of the form
\[
\xymatrix @R=2em @C=2em{
u
  \ar [r] ^-{f}
  \ar [d] _-{e}
&
u'
  \ar [r] ^-{f'}
  \ar@2 []!<0pt,-12pt>; [d]!<0pt,+12pt> ^-{}
&
w
  \ar[d] ^-{e'}
\\
u
  \ar [r] _-{g}
&
v
  \ar [r] _-{g'}
&
w'
}
\]
for every critical branching $(f,e,g)$ of $\Pr$. The main theorem of this section gives conditions to extend a square extension of $\Pr$ to a coherent extension: 
\begin{quote}
{\bf Theorem \ref{T:CoherentAcyclicity}.}
\emph{
Let $\Pr=(R,E,S)$ be an $E$-normalizing $n$-polygraph modulo, and $\Gamma$ be a square extension of $\Pr$ such that $\Pr$ is $\Gamma$-diconvergent. Then any Squier extension $\squier{\Gamma}$ is coherent.
}
\end{quote}

As a consequence of this result, Corollary~\ref{C:mainApplication}, states that for a diconvergent and $E$-normalizing $n$-polygraph modulo $\Pr$, and a family $\Gamma$ of generating confluences of $\Pr$, then any Squier extension $\squier{\Gamma}$ is coherent.
In particular, when $E$ is empty, we recover Squier's coherence theorem for convergent $n$-polygraphs as given in~{\cite[Thm. 5.2.]{GuiraudMalbos09}}, see also \cite{GuiraudMalbos12mscs}.
We conclude Section~\ref{S:CoherentCompletionModulo} by giving another condition to compute a coherent extension in terms of commuting normalization strategies on $\Pr$.

\subsection*{Organization of the article}
In Section~\ref{S:Preiminaries}, we introduce notations and terminology on higher-dimensional globular $n$-categories and globular $n$-polygraphs. We refer the reader to \cite{GuiraudMalbos09} for a deeper presentation on rewriting properties of $n$-polygraphs. We also recall from \cite{Ehresmann63} the notions of double categories and of double groupoids.
In Section~\ref{S:DoubleCoherentPresentation} we define the notions of double polygraphs and dipolygraphs, giving double coherent presentations of globular $n$-categories. Following \cite{DawsonPare02}, we construct the free $n$-category enriched in double groupoids generated by a double $(n+2)$-polygraph, in which our coherence results will be formulated.
Finally, we explain how to deduce a globular coherent presentation from a double coherent presentation. As examples, we make explicit the notion of coherent presentation in the cases of groups, commutative monoids and pivotal categories.
Section~\ref{S:PolygraphsModulo} is devoted to the study of rewriting properties of polygraphs defined modulo relations. We formulate the notions of termination, confluence, local confluence and confluence modulo for these polygraphs. Following \cite{BachmairDershowitz89}, we give a completion procedure in terms of critical branchings for confluence modulo of the polygraph modulo $(R,E,\ER)$.
In Section~\ref{S:CoherentConfluenceModulo}, we develop the notion of coherent confluence modulo and we prove a coherent version of Newman's lemma and critical branching lemma for polygraphs modulo.
In Section~\ref{S:CoherentCompletionModulo}, we define the notion of coherent completion modulo, and we show how to compute a double coherent presentation of an $n$-category by coherent completion.
Section~\ref{S:GlobularCoherenceFromDoubleCoherence} shows how to deduce a globular coherent presentation for an $n$-category from a double coherent presentation generated by a polygraph modulo. Finally, we apply our constructions in the situation of commutative monoids, pivotal monoidal categories modulo isotopy relations, and groups modulo inverse relations.

\section{Preliminaries}
\label{S:Preiminaries}

This preliminary section introduces notations on higher-dimensional categories used in this article. We recall the structures of polygraphs from \cite{Street76,Power91,Burroni93} and of double categories from~\cite{Ehresmann63}. We refer the reader to~\cite{GuiraudMalbos09,GuiraudMalbos12advances,GuiraudMalbos18} for rewriting properties of polygraphs and to~\cite{BrownSpencer76, DawsonPare93, DawsonPare02} for deeper presentations on double categories and double groupoids.

\subsection{Higher-dimensional categories and polygraphs}

Throughout this article, $n$ denotes either a fixed natural number or $\infty$.

\subsubsection{Higher-dimensional categories}

We denote by $\Cat_n$ the category of (small, strict and globular) $n$-categories.
If $\Cr$ is an~$n$-category, we denote by $\Cr_k$ the set of $k$-cells of~$\Cr$. If $f$ is a $k$-cell of~$\Cr$, then $\st_{-,i}(f)$ and $\st_{+,i}(f)$ respectively denote the $i$-source and $i$-target of $f$, while $(k-1)$-source and $(k-1)$-target are denoted by $\st_-(f)$ and $\st_+(f)$ respectively. The source and target maps satisfy the \emph{globular relations}: 
\begin{eqn}{equation}
\label{E:GlobularRelations}
\st_{\alpha,i}\st_{\alpha,i+1} \:=\: \st_{\alpha,i}\st_{\beta,i+1},
\end{eqn}
for all $\alpha\neq\beta$ in $\{-,+\}$.
Two $k$-cells $f$ and $g$ are \emph{$i$-composable} when $\st_{+,i}(f)=\st_{-,i}(g)$. In that case, their $i$-composite is denoted by $f\star_i g$, or by $fg$ when $i=0$. The compositions satisfy the \emph{exchange relations}: 
\begin{eqn}{equation}
\label{E:ExchangeRelations}
(f _1\star_i g_1) \star_j (f_2 \star_i g_2) \:=\: (f_1 \star_j f_2) \star_i (g_1 \star_j g_2),
\end{eqn}
for all $i\neq j$ and cells $f_1,f_2,g_1,g_2$ such that both sides are defined. 
If $f$ is a $k$-cell, we denote by~$1_f$ its identity $(k+1)$-cell. When $1_f$ is composed with $l$-cells, we simply denote it by~$f$ for $l\geq k+1$. 

A $k$-cell $f$ of an $n$-category~$\Cr$ is \emph{$i$-invertible} when there exists a (necessarily unique) $k$-cell $g$ in $\Cr$, with $i$-source $\st_{+,i}(f)$ and $i$-target $\st_{-,i}(f)$, called the \emph{$i$-inverse of $f$}, that satisfies 
\[
f\star_i g \:=\: 1_{\st_{-,i}(f)}
\qquad\text{and}\qquad
g\star_i f \:=\: 1_{\st_{+,i}(f)}.
\]
When $i=k-1$, we say that $f$ is \emph{invertible} and we denote by $f^-$ its \emph{inverse}. As in higher-dimensional groupoids, if a $k$-cell $f$ is invertible and if its $i$-source $u$ and $i$-target $v$ are invertible, then $f$ is $(i-1)$-invertible, with $(i-1)$-inverse given by $v^- \star_{i-1} f^- \star_{i-1} u^-$.

For a natural number $p\leq n$, or for $p=n=\infty$, an \emph{$(n,p)$-category} is an $n$-category whose $k$-cells are invertible for every $k>p$. When $n<\infty$, this is an $n$-category enriched in $(n-p)$-groupoids and, when $n=\infty$, an $n$-category enriched in $\infty$-groupoids. In particular, an $(n,n)$-category is an $n$-category, and an $(n,0)$-category is an \emph{$n$-groupoid}, also called a \emph{groupoid} for $n=1$.

A \emph{$0$-sphere of $\Cr$} is a pair $(f,g)$ of $0$-cells of $\Cr$. For $1\leq k\leq n$, a \emph{$k$-sphere of $\Cr$} is a pair $S=(f,g)$ of $k$-cells of $\Cr$ such that $\st_-(f)=\st_-(g)$ and $\st_+(f)=\st_+(g)$. The $k$-cell $f$ (resp. $g$) is called the \emph{source} (resp. target) of $S$ denoted by $\st_-(S)$ (resp. $\st_+(S)$).
We denote by $\Sph_k(\Cr)$ the set of $k$-spheres of $\Cr$. If $f$ is a $k$-cell of $\Cr$, for $1\leq k\leq n$, the \emph{boundary of $f$} is the $(k-1)$-sphere $(\st_-(f),\st_+(f))$ denoted by $\partial(f)$.

\subsubsection{Cellular extensions}
Suppose $n<\infty$, a \emph{cellular extension} of an $n$-category $\Cr$ is a set $\Gamma$ equipped with a map $\gamma : \Gamma \fl \Sph_n(\Cr)$. By considering all the formal composites of elements of~$\Gamma$, seen as $(n+1)$-cells with source and target in~$\Cr$, we build the \emph{free $(n+1)$-category generated by~$\Gamma$ over $\Cr$}, and denoted by $\Cr[\Gamma]$. The \emph{size} of an $(n+1)$-cell $f$ of $\Cr[\Gamma]$, denoted by $\norm{f}_\Gamma$, is the number of $(n+1)$-cells of~$\Gamma$ it contains. We denote by $\Cr^{(1)}$ the set of $n$-cells in $\Cr$ of size $1$. 
We denote by $(\Cr)_\Gamma$ the quotient of the $n$-category $\Cr$ by the congruence generated by $\Gamma$, \emph{i.e.} the $n$-category we get from $\Cr$ by identification of the $n$-cells $\st_-(S)$ and $\st_+(S)$, for all $n$-sphere~$S$ of $\Gamma$. 

If $\Cr$ is an $(n,p)$-category and $\Gamma$ is a cellular extension of $\Cr$, then the \emph{free $(n+1,p)$-category generated by~$\Gamma$ over $\Cr$} is denoted by~$\Cr(\Gamma)$ and defined by: 
\[
\Cr(\Gamma) \::=\: \left(\Cr \left[ \Gamma,\; \Gamma^- \right] \right)_{\text{Inv}(\Gamma)},
\]
where $\Gamma^-$ contains the same $(n+1)$-cells as $\Gamma$, with source and target reversed, and $\text{Inv}(\Gamma)$ is the cellular extension of $\Cr \left[ \Gamma,\; \Gamma^- \right]$  made of two $(n+2)$-cells 
\[
f\star_{n}f^- \:\fl\: 1_{\st_-(f)}
\qquad\text{and}\qquad
f^-\star_{n}f \:\fl\: 1_{\st_+(f)},
\]
for every $(n+1)$-cell $f$ in $\Gamma$. 

For $p<n$, a cellular extension $\Gamma$ of $\Cr$ is \emph{acyclic} if the $(n,p)$-category $\Cr/\Gamma$ is aspherical, \emph{i.e.} such that, for every $n$-sphere $S$ of $\Cr$, there exists an $(n+1)$-cell with boundary~$S$ in the $(n+1,p)$-category~$\Cr(\Gamma)$. 

\subsubsection{Polygraphs}
Recall that an \emph{$n$-polygraph} is a data $P:=(P_0,P_1,\ldots,P_n)$ made of a set $P_0$ and, for every $0\leq k<n$, a cellular extension $P_{k+1}$ of the free $k$-category
\[
P_k^\ast \::=\: P_0[P_1]\ldots[P_k],
\]
whose elements are called \emph{generating $(k+1)$-cells of $P$}. We will use the following notations:
\begin{enumerate}[{\bf i)}]
\item for $0\leq k \leq n-1$, $P_{\leq k}$ denotes the underlying $k$-polygraph $(P_0,P_1,\ldots,P_k)$ of the $n$-polygraph $P$,
\item $P^\ast$ (resp. $\tck{P}$) denotes the free $n$-category (resp. $(n,n-1)$-category) generated by $P$,
\item $\cl{P}:= (P^\ast_{n-1})_{P_n}$ denotes the \emph{$(n-1)$-category presented by $P$}.
\end{enumerate}

Given two $n$-polygraphs $P$ and $Q$, a \emph{morphism} of $n$-polygraphs from $P$ to $Q$ is a pair $(\xi_{n-1} , f_n)$ where $\xi_{n-1}$ is a morphism of $(n-1)$-polygraphs from $P_{n-1}$ to $Q_{n-1}$, and $f_n : P_n \fl Q_n$ is a map such that the following diagrams commute:
\[ 
\xymatrix@R=2em@C=2em{
P_{n-1}^\ast \ar [d] _-{F_{n-1}(\xi_{n-1})} & P_n \ar [l] _-{s_{n-1}^P} \ar [d] ^-{f_n}  \\
Q_{n-1}^\ast & Q_n \ar [l] ^-{s_{n-1}^Q}  } 
\qquad  
\qquad
\xymatrix@R=2em@C=2em{
P_{n-1}^\ast \ar [d] _-{F_{n-1}(\xi_{n-1})} & P_n \ar [l] _-{t_{n-1}^P} \ar [d] ^-{f_n}  \\
Q_{n-1}^\ast & Q_n \ar [l] ^-{t_{n-1}^Q}  } 
\]
Equivalently, it is a sequence of maps $(f_k : P_k \fl Q_k)_k$ indexed by integers $0 \leq k \leq n-1$ such that the relations
\[ f_k s_k^P = s_k^Q f_{k+1} \quad \text{and} \quad f_k t_k^P = t_k^Q f_{k+1}
\]
holds for all $0\leq k \leq n-1$.
We denote by $\Pol_n$ the category of $n$-polygraphs and their morphisms, and by  $U^{Pol}_n: \Cat_n \fl \Pol_n$ the forgetful functor sending an~$n$-category to its underlying $n$-polygraph. 

For $0\leq p\leq n$, an \emph{$(n,p)$-polygraph} is a data $P$ made of an $p$-polygraph $(P_0,\dots,P_p)$, and for every $p \leq k < n$, a cellular extension $P_{k+1}$ of the free $(k,p)$-category 
\[
\tck{P}_k \::=\: P^*_p(P_{n+1})\cdots(P_{k}).
\]
Note that an $(n,n)$-polygraph is an $n$-polygraph.

\subsubsection{Contexts in $n$-categories}
A \emph{context} of an $n$-category $\Cr$ is a pair $(S,C)$, simply denoted by $C$, made of an $(n-1)$-sphere $S$ of $\Cr$ and an $n$-cell~$C$ in $\Cr[S]$ such that $\norm{C}_S=1$. 
Recall from {\cite[Prop. 2.1.3]{GuiraudMalbos09}} that every context of $\Cr$ has a decomposition
\[
f_n \star_{n-1} ( f_{n-1} \star_{n-2} 
\cdots 
\star_1 (f_1 \star_0 S\star_0 g_1 ) \star_1 
\cdots \star_{n-2} g_{n-1} )
\star_{n-1} g_n, 
\]
with $S\in\Sph_{n-1}(\Cr)$, and $f_k,g_k\in\Cr_n$, for every $1\leq k \leq n$.
Moreover, we choose these cells so that~$f_k$ and~$g_k$ are (the identities of) $k$-cells.  
A \emph{whisker of $\Cr$} is a context with a decomposition
\[
f_{n-1} \star_{n-2}  \cdots \star_1 ( f_1 \star_0 S\star_0 g_1 )\star_1 \cdots \star_{n-2} g_{n-1} 
\]
such that $f_k,g_k\in\Cr_k$, for every $1\leq k\leq n-1$.
Given an $n$-polygraph $P$, recall from {\cite[Prop. 2.1.5]{GuiraudMalbos09}} that every $n$-cell $f$ in $P^{\ast}$ with size $k\geq 1$ has a decomposition 
\[
f \:=\: C_1[\gamma_1] \star_{n-1} \cdots \star_{n-1} C_k[\gamma_k],
\] 
where $\gamma_1$, $\dots$, $\gamma_k$ are generating $n$-cells of $P$ and $C_1$, $\dots$, $C_k$ are whiskers of $P^*$.

\subsubsection{Rewriting paths}
\label{SSS:Reductions}
Let $P$ be an $n$-polygraph, and $(u,v)$ in $\Sph_{n-1}(P_{n-1}^\ast)$.
A \emph{$P_n$-rewriting step} from $u$ to $v$ is an $n$-cell $f$ in $P_n^{\ast(1)}$ such that $\partial_-(f)=u$ and $\partial_+ (f) = v$. It can be written $f=C[\gamma]$, where $\gamma$ is a generating $n$-cell of $P$, and $C$ is a whisker of $P_{n-1}^\ast$. An $(n-1)$-cell $u$ of $P_{n-1}^\ast$ is \emph{$P_n$-reducible} if there exists a $P_n$-rewriting step with source $u$, otherwise $u$ is \emph{$P_n$-irreducible.}

A \emph{$P_n$-rewriting path} is a sequence $(f_i)_{i\in I}$ of $P_n$-rewriting steps such that $\st_+(f_i)=\st_-(f_{i+1})$ for every $i\in I$. A rewriting path $(f_1,f_2,\ldots,f_k)$ yields an $n$-cell $f_1\star_{n-1} f_2 \star_{n-1} \ldots \star_{n-1} f_k$ in $P_n^\ast$. The polygraph $P$ is \emph{terminating} if there is no infinite rewriting path.
We say that the $(n-1)$-cell \emph{$u$ $P_n$-reduces to $v$} if there is a $P_n$-rewriting path with source $u$ and target $v$.

\subsection{Double groupoids}

In this subsection, we recall the notion of double category introduced in \cite{Ehresmann63}.
Recall that an \emph{internal category} $\C$ in a category $\Vr$ with finite limits is a data $(\C_1,\C_0,\st_-^\C,\st_+^\C,\circ_\C,i_\C)$, where 
\[
\st_-^\C,\st_+^\C : \C_1 \fll \C_0,
\qquad
i_\C : \C_0 \fll \C_1,
\qquad
\circ_\C : \C_1 \times_{\C_0} \C_ 1 \fll \C_1
\]
are morphisms of $\Vr$ satisfying the axioms of a category, and $ \C_1 \times_{\C_0} \C_ 1$ denotes the pullback in~$\Vr$ over morphisms~$\st_-^\C$ and $\st_+^\C$. 
An internal functor from $\C$ to $\D$ is a pair of morphisms $\C_1 \fl \D_1$ and $\C_0 \fl \D_0$ in $\Vr$ commuting in the obvious way. 
We denote by $\Cat(\Vr)$ the category of internal categories in~$\Vr$ and their functors.
In the same way, we define an \emph{internal groupoid} $\G$ in $\Vr$ as an internal category $(\G_1,\G_0,\st_-^\G,\st_+^\G,\circ_\G,i_\G)$ with an additional morphism 
\[
(\cdot)^{-}_{\G} : \G_1 \fl \G_1
\] 
satisfying the axioms of groups, that is
\begin{eqn}{equation}
\label{E:AxiomGroup1}
\st_{-}^{\G} \circ (\cdot )^{-}_{\G} = \st_+^{\G}, \quad \st_{+}^{\G} \circ (\cdot)^{-}_{\G} = \st_{-}^{\G},\quad 
\end{eqn}
\begin{eqn}{equation}
\label{E:AxiomGroup2}
i_{\G} \circ \st_{-}^{\G} = \circ_{\G} \circ (\text{id} \times (\cdot)^{-}_{\G}) \circ \Delta,\quad
i_{\G} \circ \st_{+}^{\G} = \circ_{\G} \circ ((\cdot)^{-}_{\G} \times \text{id}) \circ \Delta,
\end{eqn}
where $\Delta : \G_1 \fl \G_1 \times \G_1$ is the diagonal functor.
We denote by $\Grpd(\Vr)$ the category of internal groupoids in~$\Vr$ and their functors.

\subsubsection{Double categories and double groupoids}
Denote by $\Cat$ (resp. $\Grpd$) the category of (small) categories (resp. groupoids) and functors. 
The category of \emph{double categories} (resp. \emph{double groupoids}) is defined as the category $\Cat(\Cat)$ (resp. $\Grpd(\Grpd)$). Explicitly, a double category is an internal category $(\C_1,\C_0,\st_-^\C,\st_+^\C,\circ_\C,i_\C)$ in $\Cat$, that gives four related categories:
\begin{align*}
 \C^{sv}&\: := \: (\C^s,\C^v,\st_{-,1}^v,\st_{+,1}^v,\cdg^v,i^v_1),& \C^{sh}&\: :=\: (\C^s,\C^h,\st_{-,1}^h,\st_{+,1}^h,\cdg^h,i^h_1),\\
 \C^{vo}&\: :=\:  (\C^v,\C^o,\st_{-,0}^v,\st_{+,0}^v,\circ^v,i^v_0),&\C^{ho}&\: := \:(\C^h,\C^o,\st_{-,0}^h,\st_{+,0}^h,\circ^h,i^h_0),\\
\end{align*}
where $\C^{sh}$ is the category $\C_1$ and $\C^{vo}$ is the category $\C_0$. The sources, target and identity maps pictured in the following diagram
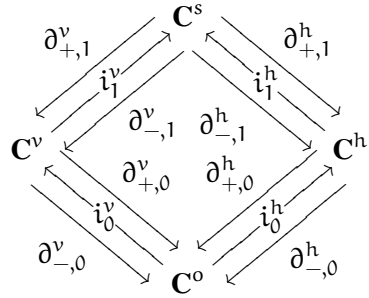
\[
\xymatrix@R=3.2em @C=3.8em{
& 
\C^s 
\ar@<+1.7ex>[dr] ^-{\st_{+,1}^h}
\ar@<-1.7ex>[dr] _-{\st_{-,1}^h}
\ar@<+1.7ex>[dl] ^-{\st_{-,1}^v}
\ar@<-1.7ex>[dl] _-{\st_{+,1}^v}
&
\\
\C^v 
\ar@<+1.7ex>[dr] ^-{\st_{+,0}^v}
\ar@<-1.7ex>[dr] _-{\st_{-,0}^v}
\ar[ur] |-{i_{1}^v}
&& 
\C^h
\ar@<+1.7ex>[dl] ^-{\st_{-,0}^h}
\ar@<-1.7ex>[dl] _-{\st_{+,0}^h}
\ar[ul] |-{i_{1}^h}
\\
&\C^o
\ar[ur] |-{i_{0}^h}
\ar[ul] |-{i_{0}^v}
&
}
\]
satisfy the following relations:
\begin{enumerate}[{\bf i)}]
\item 
$\st_{\alpha,0}^h\st_{\beta,1}^h = \st_{\beta,0}^v\st_{\alpha,1}^v$,
for all $\alpha,\beta$ in $\{-,+\}$,
\item
$\st_{\alpha,1}^\mu i_1^\eta = i_0^\mu\st_{\alpha,0}^\eta$, 
for all $\alpha$ in $\{-,+\}$ and $\mu,\eta$ in $\{v,h\}$,
\item $i_1^vi_0^v = i_1^hi_0^h$,
\item 
$\st_{\alpha,1}^{\mu}(A\cdg^{\mu} B) =\st_{\alpha,1}^{\mu}(A) \circ^{\mu} \st_{\alpha,1}^{\mu}(B)$, 
for all $\alpha\in\{-,+\}$, $\mu\in\{v,h\}$, and squares $A,B$ such that both sides are defined,
\item \emph{The middle four interchange law} :
\begin{eqn}{equation}
\label{E:MiddleFourIdentities}
(A\cdg^vA')\cdg^h(B\cdg^hB') = (A \cdg^h B) \cdg^v (A' \cdg^h B'), 
\end{eqn}
for all cells $A,A',B,B'$ in $\C^s$ such that both sides are defined. 
\end{enumerate}
Elements of $\C^o$ are called \emph{point cells}, the elements of $\C^h$ and $\C^v$ are called \emph{horizontal} and \emph{vertical cells}, and pictured by
\[
\xymatrix@C=1em@R=1em{ 
x_1
\ar[rr] ^{f} 
& & 
x_2
}
\qquad\text{and}\qquad
\raisebox{1cm}{
\xymatrix@C=1em@R=1em{ 
x_1
\ar[dd] _{e} 
\\
\\ 
x_2
}}
\]
respectively.
Following relations {\bf i)}, the elements of $\C^s$ are called \emph{square cells} and pictured as follows: 
\[
\xymatrix@R=2em @C=2em{ 
\cdot
\ar@1[r] ^{\st_{-,1}^h(A)} _{}="src" 
\ar@1[d] _{\st_{-,1}^v(A)} 
& 
\cdot
\ar@1[d] ^{\st_{+,1}^v(A)} \\
\cdot
\ar@1[r] _{\st_{+,1}^h(A)} ^{}="tgt" 
& 
\cdot
\ar@2 "src"!<0pt,-10pt>;"tgt"!<0pt,10pt> ^{A}
}
\]
and by the followings squares for identities
\[
\xymatrix@R=2em @C=2em{ 
x_1
\ar@1[r] ^{f} _{}="src" 
\ar@1[d] _{i_0^v(x_1)} 
& 
x_2
\ar@1[d] ^{i_0^v(x_2)} \\
x_1
\ar@1[r] _{f} ^{}="tgt" 
& 
x_2
\ar@2 "src"!<-7pt,-10pt>;"tgt"!<-7pt,10pt> ^{i_1^h(f)}
}
\qquad
\xymatrix@R=2em @C=2em{ 
x
\ar@1[r] ^{i_0^h(x)} _{}="src" 
\ar@1[d] _{e} 
& 
x
\ar@1[d] ^{e} \\
y
\ar@1[r] _{i_0^h(y)} ^{}="tgt" 
& 
y
\ar@2 "src"!<-7pt,-10pt>;"tgt"!<-7pt,10pt> ^{i_1^v(e)}
}
\qquad
\raisebox{-0.6cm}{\text{or simply by}}
\qquad
\xymatrix@R=2em @C=2em{ 
x_1
\ar@1[r] ^{f} _{}="src" 
\ar@1[d] _{\rotatebox{90}{=}} 
& 
x_2
\ar@1[d] ^{\rotatebox{90}{=}} \\
x_1
\ar@1[r] _{f} ^{}="tgt" 
& 
x_2
\ar@2@{} "src"!<0pt,-10pt>;"tgt"!<0pt,10pt> |-{i_1^h(f)}
}
\qquad
\xymatrix@R=2em @C=2em{ 
x
\ar@1[r] ^{=} _{}="src" 
\ar@1[d] _{e} 
& 
x
\ar@1[d] ^{e} \\
y
\ar@1[r] _{=} ^{}="tgt" 
& 
y
\ar@2@{} "src"!<0pt,-10pt>;"tgt"!<0pt,10pt> |-{i_1^v(e)}
}
\]

The compositions $\cdg^v$ (resp. $\cdg^h$) are called respectively \emph{vertical} and \emph{horizontal compositions}, and can be pictured as follows
\[
\xymatrix@C=1em@R=1em{ 
x_1
\ar[rr] ^{f_1} _{}="src" 
\ar[dd] _{e_1} 
& & 
x_2
\ar[dd]^{e_2}
\ar[rr] ^{f_2} _{}="sb" 
& & 
x_3 
\ar[dd]^{e_3}
\\
& &  & &
\\
y_1 
\ar[rr] _{g_1} ^{}="tgt" 
& & 
y_2
\ar@2 "src"!<0pt,-12pt>;"tgt"!<0pt,12pt>  ^{A} 
\ar[rr] _{g_2} ^{}="tb" 
& & 
y_3 
\ar@2 "sb"!<0pt,-12pt>;"tb"!<0pt,12pt>  ^{B} 
} 
\quad 
\raisebox{-10mm}{$\rightsquigarrow$} \quad 
\xymatrix@C=1em@R=1em{ 
x_1
\ar[rrrr] ^{f_1\circ^h f_2} _{}="src" 
\ar[dd] _{e_1} 
& &  & & 
x_3  
\ar[dd]^{e_3}\\
& &  & &\\
y_1 
\ar[rrrr] _{g_1\circ^h g_2} ^{}="tgt" 
& &  & & 
y_3 
\ar@2 "src"!<-12pt,-10pt>;"tgt"!<-12pt,10pt>  ^-{A \cdg^v B} }
\]
for all $x_i,y_i$ in $\C^o$, $f_i,g_i$ in $\C^h$, $e_i$ in $\C^v$ and $A,B$ in $\C^s$,
\[
\xymatrix @R=1em @C=1em{
x_1 
\ar[rr] ^{f} _{}="src" 
\ar[dd] _{e_1} 
& & 
x_2
\ar[dd]^{e_2} 
\\
& & 
\\
y_1 
\ar[rr] _{g} ^{}="tgt" 
\ar[dd]_{e'_1}  
& & 
y_2 
\ar[dd]^{e'_2}  
\ar@2 "src"!<0pt,-12pt>;"tgt"!<0pt,12pt>  ^{A}
\\
 & & 
\\
z_1 
\ar[rr] _{h} ^{}="tgtb" 
& & 
z_2 
\ar@2 "tgt"!<0pt,-15pt>;"tgtb"!<0pt,10pt> ^{A'} 
} 
\quad 
\raisebox{-18mm}{$\rightsquigarrow$} \quad \xymatrix@C=1em@R=1em{
x_1 
\ar[rr] ^{f} _{}="src" 
\ar[dddd] _{e_1 \circ^v e'_1} 
& & 
x_2
\ar[dddd]^{e_2 \circ^v e'_2}
\\
& & \\
& & \\
 & & \\
z_1
\ar[rr] _{h} ^{}="tgtb" 
& & 
z_2 
\ar@2 "src"!<0pt,-20pt>;"tgtb"!<0pt,20pt> |-{A \cdg^h A'}}
 \]
for all $x_i,y_i,z_i$ in $\C^o$, $f,g,h$ in $\C^h$, $e_i,e_i'$ in $\C^v$ and $A,A'$ in $\C^s$.

Similarly a double groupoid is given by the same data $(\G_1,\G_0,\st_-^\G,\st_+^\G,\circ_\G,i_\G)$, with an inverse operation $(\cdot)_\G^- : \G_1 \fl \G_1$ satisfying relations~\eqref{E:AxiomGroup1} and~\eqref{E:AxiomGroup2}. As a consequence the four related categories $\G^{sv}$, $\G^{sh}$, $\G^{vo}$ and $\G^{ho}$ are groupoids. 
For every square cell
\[
\xymatrix@R=2em @C=2em{ 
\cdot
\ar@1[r] ^{f} _{}="src" 
\ar@1[d] _{e} 
& 
\cdot
\ar@1[d] ^{e'} \\
\cdot
\ar@1[r] _{g} ^{}="tgt" 
& 
\cdot
\ar@2 "src"!<0pt,-10pt>;"tgt"!<0pt,10pt> ^{A}
}
\]
in $\G^s$, the inverse square cell with respect to $\cdg^\mu$, for $\mu\in\{v,h\}$, is denoted by $A^{-,\mu}$ and satisfy the following relations
\begin{eqn}{equation}
\label{E:RelationsInverseSquareCell}
A \cdg^{\mu} (A^{-,\mu}) = i_1^\mu(\st_{-,1}^\mu(A)),
\qquad
(A^{-,\mu}) \cdg^{\mu} A = i_1^\mu(\st_{+,1}^\mu(A)).
\end{eqn}
The sources and targets of the inverse squares $A^{-,v}$ and $A^{-,h}$ are given as follows
\[
\xymatrix@R=2.4em @C=2.4em{ 
\cdot
\ar@1[r] ^{f^-} _{}="src" 
\ar@1[d] _{e'} 
& 
\cdot
\ar@1[d] ^{e} \\
\cdot
\ar@1[r] _{g^-} ^{}="tgt" 
& 
\cdot
\ar@2 "src"!<-8pt,-10pt>;"tgt"!<-8pt,10pt> ^{A^{-,v}}
}
\qquad
\xymatrix@R=2.4em @C=2.4em{ 
\cdot
\ar@1[r] ^{g} _{}="src" 
\ar@1[d] _{e^-} 
& 
\cdot
\ar@1[d] ^{(e')^-} \\
\cdot
\ar@1[r] _{f} ^{}="tgt" 
& 
\cdot
\ar@2 "src"!<-8pt,-10pt>;"tgt"!<-8pt,10pt> ^{A^{-,h}}
}
\]

\subsubsection{Squares} 
A \emph{square} in a double category $\C$ is a quadruple $(f,g,e,e')$ such that $f,g$ are horizontal cells and $e,e'$ are vertical cells that compose as follows:
\[
\xymatrix@R=2em @C=2em{ 
u
\ar@1[r] ^-{f}
\ar@1[d] _-{e} 
& 
v
\ar@1[d] ^-{e'} \\
u'
\ar@1[r] _-{g}
& 
v'
}
\]
The \emph{boundary} of a square cell $A$ in $\C$ is the square $(\st_{-,h}(A),\st_{+,h}(A),\st_{-,v}(A),\st_{+,v}(A))$, denoted by~$\partial(A)$.
We denote by $\Sq(\C)$ the set of square cells of $\C$.

\subsubsection{$n$-categories enriched in double categories}
The coherence results for rewriting systems modulo of this article are formulated using $n$-categories enriched in double groupoids. Let us expand this notion for $n>0$.
We equip the category $\Grpd(\Grpd)$ with the cartesian product defined by 
\[
\C\times\D = (\C_1\times\D_1,\C_0\times\D_0,s_\C\times t_\C, c_\C\times c_\D, i_\C\times i_\D),
\]
for all double groupoids $\C$ and $\D$. The terminal double groupoid $\T$ has only one point cell, denoted by~$\bullet$, and identities $i_0^v(\bullet)$, $i_0^h(\bullet)$, $i_1^vi_0^h(\bullet) = i_1^hi_0^v(\bullet)$.

 An $n$-category enriched in double groupoids is an $n$-category $\Cr$ such that for all $x,y\in\Cr_{n-1}$ the homset $\Cr_{n}(x,y)$ has a double groupoid structure, whose point cells are the $n$-cells in $\Cr_{n}(x,y)$.
We denote by $\Cr_{n+1}^v$ (resp. $\Cr_{n+1}^h$, $\Cr_{n+2}^s$) the union of sets $\Cr_{n}(x,y)^v$ (resp. $\Cr_{n}(x,y)^h$, $\Cr_{n}(x,y)^s$) for all $x,y\in\Cr_{n-1}$.
An $(n+2)$-cell $A$ in $\Cr_{n+2}^s$ can be represented by the following diagram:
\[
\xymatrix@R=2em @C=2em{ 
u 
\ar[r] ^-{f} _{}="src" 
\ar[d] _-{e} 
& 
v
\ar[d] ^-{e'} \\
 u' 
 \ar[r] _-{g} ^{}="tgt" 
 &
 v' 
 \ar@2 "src"!<0pt,-10pt>;"tgt"!<0pt,10pt>  ^{A}
}
\]
with $u,u',v,v'\in\Cr_n$, $f,g\in\Cr_{n+1}^h$ and $e,e'\in\Cr_{n+1}^v$.
The composites of the $(n+2)$-cells and the identities $(n+2)$-cells are induced by the functors of double categories
\[
\star_{n-1}^{x,y,z} : \Cr_n(x,y) \times \Cr_n(y,z) \fl \Cr_n(x,z),
\qquad
\qquad
1_x : \T \fl \Cr_n(x,x),
\]
for all $x,y,z\in\Cr_{n-1}$.
The $(n-1)$-composite of an $(n+2)$-cell $A$ in $\Cr_n(x,y)$ with an $(n+2)$-cell~$B$ in $\Cr_n(y,z)$ of the form
\[
\xymatrix @R=2em @C=2em{ 
u_1
\ar[r] ^-{f_1} _{}="src" 
\ar[d] _-{e_1} 
& 
v_1
\ar[d]^-{e'_1}  
\\
u'_1 
\ar[r] _-{g_1} ^{}="tgt" 
&
v'_1 
\ar@2 "src"!<0pt,-10pt>;"tgt"!<0pt,10pt>  ^{A}
}
\qquad
\qquad
\xymatrix @R=2em @C=2em{ 
u_2
\ar[r] ^-{f_2} _{}="src" 
\ar[d] _-{e_2} 
& 
v_2
\ar[d]^-{e'_2} 
\\
u'_2 
\ar[r] _-{g_2} ^{}="tgt" 
&
v'_2 
\ar@2 "src"!<0pt,-10pt>;"tgt"!<0pt,10pt>  ^{B}
}
\]
is defined by $(n-1)$-composites of $n$-cells, vertical $(n+1)$-cells and horizontal $(n+1)$-cells as follows: 
\[
\xymatrix@C=2.5em@R=1.2em{ 
u_1\star_{n-1}u_2
\ar[rr] ^{f_1\star_{n-1}f_2} _{}="src" 
\ar[dd] _{e_1\star_{n-1}e_2} 
& & 
v_1\star_{n-1}v_2
\ar[dd]^{e'_1\star_{n-1}e'_2} 
\\
& & 
\\
u'_1 \star_{n-1} u'_2
\ar[rr] _{g_1\star_{n-1} g_2} ^{}="tgt" 
& & 
v'_1 \star_{n-1} v'_2
\ar@2 "src"!<0pt,-15pt>;"tgt"!<0pt,15pt>  ^{A\star_{n-1} B}
}
\]
By functoriality, the $(n-1)$-composite satisfies the following exchange relations:
\begin{eqn}{equation}
\label{E:ExchangeRelations1}
(A\cdg^\mu A') \star_{n-1} (B \cdg^\mu B')
\: = \: (A\star_{n-1} B) \cdg^\mu (A' \star_{n-1} B'),
\end{eqn}
\begin{eqn}{equation}
\label{E:ExchangeRelations2}
(A \cdg^\mu A') \star_{n-1} (B \cdg^\eta B') \: = \: 
((A \star_{n-1} B) \cdg^\mu (A' \star_{n-1} B)) \cdg^\eta ((A \star_{n-1} B') \cdg^\mu (A' \star_{n-1} B')).
\end{eqn}
Using the middle four interchange law~\eqref{E:MiddleFourIdentities}, the identity~\eqref{E:ExchangeRelations2} is equivalent to the following identity
\[
(A \cdg^\mu A') \star_{n-1} (B \cdg^\eta B') \: = \: ((A \star_{n-1} B) \cdg^{\eta} (A \star_{n-1} B')) \cdg^\mu ((A' \star_{n-1} B) \cdg^\eta (A' \star_{n-1} B')) 
\]
for all $\mu \ne \eta$ in $\{v,h\}$ and $(n+2)$-cells $A,A',B,B'$ such that both sides are defined.

We denote by $\DbCat$-$\Cat_n$ (resp. $\DbGrpd$-$\Cat_n$) the category of $n$-categories enriched in double categories (resp. double groupoids) and enriched $n$-functors.

\subsubsection{$2$-categories as double categories}
\label{SSS:2CatAsDoubleCategories}
From a $2$-category $\Cr$, we can deduce two canonical double categories, by setting the vertical or horizontal cells to be only identities of $\Cr$. In this way, $2$-categories can be considered as special cases of double categories.
The \emph{quintet construction} associates to $\Cr$ a double category, called the \emph{double category of quintets in $\Cr$} and denoted by ${\bf Q}(\Cr)$. The vertical and horizontal categories of ${\bf Q}(\Cr)$ are both equal to $\Cr_1$, and there is a square cell
\[
\xymatrix@R=2em @C=2em{
u
  \ar[r] ^-{f} ^-{}="1"
  \ar[d] _-{g}
&
u' 
  \ar[d] ^-{k}
\\
v
  \ar[r] _-{h} _-{}="2"
&
v' 
\ar@2 "1"!<0pt,-10pt>;"2"!<0pt,10pt> ^-{A}
}
\]
in ${\bf Q}(\Cr)$ whenever there is a $2$-cell $A: f \star_0 k \dfl g \star_0 h$ in $\Cr$. 
This defines a functor ${\bf Q}: \catego{Cat}_2 \fl \catego{DbCat}$.
Similarly, for $n \geq 2$ we associate to an $n$-category an $(n-2)$-category enriched in double categories by the quintet construction.

\section{Double coherent presentations}
\label{S:DoubleCoherentPresentation}

Recall that a \emph{coherent presentation} of an $n$-category $\Cr$ whose underlying $(n-1)$-category is free is an $(n+2,n)$-polygraph $P$ such that the underlying $(n+1)$-polygraph $P_{\leq (n+1)}$ is a presentation of $\Cr$ and $P_{n+2}$ is an acyclic extension of the free $(n+1,n)$-category $\tck{{P_{\leq (n+1)}}}$.
In this section, we introduce the structure of dipolygraph in order to define coherent presentations of $n$-categories whose underlying $(n-1)$-categories are not free. 
We also introduce double $(n+2)$-polygraphs as systems of generators for $n$-categories enriched in double groupoids. Finally, we define double coherent presentations of $n$-categories, that we will use to deduce coherent presentations from polygraphs modulo, whose generating horizontal cells describe primary rules, generating vertical cells are algebraic axioms, and square cells correspond to generating relations among primary reductions modulo the axioms.

\subsection{Double polygraphs and dipolygraphs}
\label{SS:DoublePolygraphDiPolygraphs}

\subsubsection{Square extensions}
\label{SSS:SquareExtensions}
Let $(\C^v,\C^h)$ be a pair of $n$-categories with the same underlying $(n-1)$-category $\B$. 
A \emph{square extension} of $(\C^v,\C^h)$ is a set $\Gamma$ equipped with four maps~$\st_{\alpha,n}^{\mu}$, with~$\alpha\in\{-,+\}$ and $\mu\in \{h,v\}$, as depicted by the following diagram: 
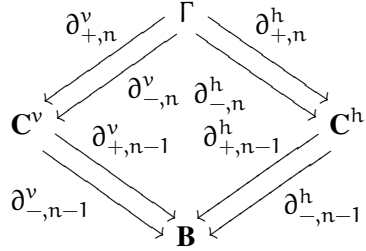
\[
\xymatrix@R=2.5em @C=4em{
& 
\Gamma
\ar@<+0.8ex>[dr] ^(.4){\st_{+,n}^h}
\ar@<-0.8ex>[dr] _(.4){\st_{-,n}^h}
\ar@<+0.8ex>[dl] ^(.4){\st_{-,n}^v}
\ar@<-0.8ex>[dl] _(.4){\st_{+,n}^v}
&
\\
\C^v
\ar@<+0.8ex>[dr] ^(.4){\st_{+,n-1}^v}
\ar@<-0.8ex>[dr] _(.4){\st_{-,n-1}^v}
&& 
\C^h
\ar@<+0.8ex>[dl] ^(.4){\st_{-,n-1}^h}
\ar@<-0.8ex>[dl] _(.4){\st_{+,n-1}^h}
\\
&\B&
}
\]
and satisfying the following relations:
\[
\st_{\alpha,n-1}^v\st_{\beta,n}^v = \st_{\beta,n-1}^h\st_{\alpha,n}^h,
\]
for all $\alpha,\beta$ in $\{-,+\}$.  
The point cells of a square $A$ in $\Gamma$ are the $(n-1)$-cells of $\B$ of the form
\[
\st_{\alpha,n-1}^\mu\st_{\beta,n}^\eta(A)
\]
with $\alpha,\beta$ in $\{-,+\}$, and $\eta,\mu$ in $\{h,v\}$. Note that by construction these four $(n-1)$-cells have the same $(n-2)$-source and $(n-2)$-target in $\B$ respectively denoted by $\st_{-,n-2}(A)$ and $\st_{+,n-2}(A)$.

A pair of $n$-categories $(\C^v,\C^h)$ has two canonical square extensions, the empty one, and the full one that contains all squares on $(\C^v,\C^h)$, denoted by $\Sq(\C^v,\C^h)$.
The \emph{Peiffer square extension} of $(\C^v,\C^h)$ is the square extension, denoted by $\peiffer{\C^v}{\C^h}$, made of squares of the forms
\[
\xymatrix @R=2em @C=3.5em {
u\star_i v 
  \ar[r] ^-{f\star_i v} 
  \ar[d] _-{u\star_i e}
& 
u'\star_i v
  \ar[d] ^-{u'\star_i e}
\\
u\star_i v' 
  \ar[r] _-{f\star_i v'} 
&
u'\star_i v'
}
\qquad\qquad
\xymatrix @R=2em @C=3.5em {
w\star_i t 
  \ar[r] ^-{w\star_i f'} 
  \ar[d] _-{e'\star_i t}
& 
w\star_i t'
  \ar[d] ^-{e'\star_i t'}
\\
w'\star_i t 
   \ar[r] _-{w'\star_i f'} 
&
w'\star_i t'
}
\]
for all $n$-cells $e,e'$ in $\C^v$ and $n$-cells $f,f'$ in $\C^h$. 

\subsubsection{Double polygraphs}
\label{SSS:DoublePolygraphDiPolygraphs}
We define a \emph{double $(n+2)$-polygraph} as a data~$P:=(P^v,P^h,P^s)$ made of 
\begin{enumerate}
\item two $(n+1)$-polygraphs $P^{v}$ and $P^{h}$, such that $P^{v}_{\leq n} = P^{h}_{\leq n}$,
\item a square extension $P^{s}$ of the pair of  free $(n+1)$-categories $((P^{v})^\ast,(P^{h})^\ast)$.
\end{enumerate}

For $0\leq k \leq n$, the generating $k$-cells of the polygraphs $P^v$ and $P^h$ are called the \emph{generating $k$-cells of~$P$}. The generating $(n+1)$-cells of $P^v$ (resp. $P^h$) are called the \emph{generating vertical} (resp. \emph{horizontal}) \emph{$(n+1)$-cells of $P$}, and the elements of $P^s$ are called the \emph{generating square $(n+2)$-cells} of $P$.

A \emph{morphism of double $(n+2)$-polygraphs} from $(P^v,P^h,P^s)$ to $(Q^v,Q^h,Q^s)$ is a triple $(f^v,f^h,f^s)$ made of two morphisms of $(n+1)$-polygraphs
\[
f^v : P^v \fl Q^v,
\qquad
f^h : P^h \fl Q^h,
\]
and a map $f^s : P^s \fl Q^s$ such that the following diagrams commute:
\[ 
\xymatrix@R=2.25em@C=3.5em{
P_{n+1}^\mu \ar [d] _-{f_{n+1}^\mu} & P^s \ar [l] _-{\st_{-,n-1}^{\mu,P}} \ar [d] ^-{f^s}  \\
Q_{n+1}^\mu & Q^s \ar [l] ^-{\st_{-,n-1}^{\mu,Q}}} 
\qquad
\qquad
\xymatrix@R=2.25em@C=3.5em{
P_{n+1}^\mu \ar [d] _-{f_{n+1}^\mu} & P^s \ar [l] _-{\st_{+,n-1}^{\mu,P}} \ar [d] ^-{f^s}  \\
Q_{n+1}^\mu & Q^s \ar [l] ^-{\st_{+,n-1}^{\mu,Q}}} \]
for $\mu$ in $\{ v,h \}$. We denote by $\dbpol_{n+2}$ the category of double $(n+2)$-polygraphs and their morphisms.

Let us define extensions of the notion of double polygraphs used in the sequel to formulate coherence and confluence results modulo.  We define a \emph{double $(n+2,n)$-polygraph} as a double $(n+2)$-polygraph whose square extension $P^s$ is defined on the pair of $(n+1,n)$-categories  $(\tck{(P^{v})},\tck{(P^{h})})$. We denote by $\dbpol_{(n+2,n)}$ the category of double $(n+2,n)$-polygraphs. In some situations, we will also consider double $(n+2)$-polygraphs whose square extension is defined on the pair of $(n+1)$-categories $(\tck{(P^{v})},(P^{h})^\ast)$ (resp. $((P^{v})^\ast,\tck{(P^{h})})$).
We respectively denote by $\dbpol_{n+2}^{v}$ (resp. $\dbpol_{n+2}^{h}$) the categories they form.

A double $(n+2)$-polygraph (resp. $(n+2,n)$-polygraph) $(P^v,P^h,P^s)$ can be pictured by the following diagram
\[
\xymatrix@C=3.15em@R=3.15em {
& 
P^{s}
\ar@<+0.5ex>[dr] ^(.525){\st_{+,n+1}^{h}}
\ar@<-0.5ex>[dr] _(.525){\st_{-,n+1}^{h}}
\ar@<+0.5ex>[dl] ^(.525){\st_{-,n+1}^{v}}
\ar@<-0.5ex>[dl] _(.525){\st_{+,n+1}^{v}}
&
\\
(P^{v})^\ast
\ar@<+0.5ex>[dr] ^(.35){\st_{-,n}^{v}}
\ar@<-0.5ex>[dr] _(.35){\st_{+,n}^{v}}
&& 
(P^{h})^\ast
\ar@<+0.5ex>[dl] ^(.35){\st_{+,n}^{h}}
\ar@<-0.5ex>[dl] _(.35){\st_{-,n}^{h}}
\\
P^{v}
\ar@<+0.5ex>[r] ^-{\st_{+,n}^{v}}
\ar@<-0.5ex>[r] _-{\st_{-,n}^{v}}
\ar[u] ^-{\iota_{n+1}^{v}}
&
P^\ast_{n}
\ar@<+0.5ex>[d] ^-{\st_{-,n-1}}
\ar@<-0.5ex>[d] _-{\st_{+,n-1}}
&
\ar@<+0.5ex>[l] ^-{\st_{-,n}^{h}}
\ar@<-0.5ex>[l] _-{\st_{+,n}^{h}}
\ar[u] _-{\iota_{n+1}^{h}}
P^{h}
\\
& P^\ast_{n-1} &
}
\qquad
\raisebox{-2cm}{
\text{(resp. }
}
\quad
\xymatrix@C=3.15em@R=3.15em {
& 
P^{s}
\ar@<+0.5ex>[dr] ^(.525){\st_{+,n+1}^{h}}
\ar@<-0.5ex>[dr] _(.525){\st_{-,n+1}^{h}}
\ar@<+0.5ex>[dl] ^(.525){\st_{-,n+1}^{v}}
\ar@<-0.5ex>[dl] _(.525){\st_{+,n+1}^{v}}
&
\\
(P^{v})^\top
\ar@<+0.5ex>[dr] ^(.35){\st_{-,n}^{v}}
\ar@<-0.5ex>[dr] _(.35){\st_{+,n}^{v}}
&& 
(P^{h})^\top
\ar@<+0.5ex>[dl] ^(.35){\st_{+,n}^{h}}
\ar@<-0.5ex>[dl] _(.35){\st_{-,n}^{h}}
\\
P^{v}
\ar@<+0.5ex>[r] ^-{\st_{+,n}^{v}}
\ar@<-0.5ex>[r] _-{\st_{-,n}^{v}}
\ar[u] ^-{\iota_{n+1}^{v}}
&
P^\ast_{n}
\ar@<+0.5ex>[d] ^-{\st_{-,n-1}}
\ar@<-0.5ex>[d] _-{\st_{+,n-1}}
&
\ar@<+0.5ex>[l] ^-{\st_{-,n}^{h}}
\ar@<-0.5ex>[l] _-{\st_{+,n}^{h}}
\ar[u] _-{\iota_{n+1}^{h}}
P^{h}
\\
& P^\ast_{n-1} &
}
\quad
\raisebox{-2cm}{\text{).}}
\]

\subsubsection{Dipolygraphs}
We define dipolygraphs as presentations by generators and relations for $\infty$-categories whose underlying $k$-categories are not necessarily free. Note that a similar notion was introduced by Burroni in~\cite{Burroni07}.
We define $n$-dipolygraphs by induction on $n\geq 0$. 
A \emph{$0$-dipolygraph} is a set. A \emph{$1$-dipolygraph} is a triple $((P_0,P_1),Q_1)$, where $(P_0,Q_1)$ is a $1$-polygraph and $P_1$ is a cellular extension of the quotient category $(P_0^\ast)_{Q_1}$. For $n\geq 2$, an \emph{$n$-dipolygraph} is a data $(P,Q)=((P_i)_{0\leq i \leq n},(Q_i)_{1 \leq i \leq n})$ made of
\begin{enumerate}[{\bf i)}]
\item a $1$-dipolygraph $((P_0,P_1),Q_1)$,
\item for every $2\leq k \leq n$, a cellular extension $Q_{k}$ of the $(k-1)$-category
\[
[P_{k-2}]_{Q_{k-1}}[P_{k-1}],
\]
where $[P_{k-2}]_{Q_{k-1}}$ denotes the $(k-2)$-category
\[
((((P_0^\ast)_{Q_1}[P_1])_{Q_2}[P_2])_{Q_3}\ldots [P_{k-2}])_{Q_{k-1}},
\]
\item for every $2\leq k \leq n$, a cellular extension $P_{k}$ of the $(k-1)$-category
\[
[P_{k-1}]_{Q_{k}}.
\]
\end{enumerate}
For $0\leq k \leq n-1$, we denote by $(P,Q)_{\leq k}$ the underlying $k$-dipolygraph $((P_i)_{0\leq i \leq k},(Q_i)_{1 \leq i \leq k})$.

\subsubsection{}
For $0\leq p\leq n$, an \emph{$(n,p)$-dipolygraph} is a data $((P_i)_{0\leq i \leq n},(Q_i)_{1 \leq i \leq n})$ such that:
\begin{enumerate}[{\bf i)}]
\item $((P_i)_{0\leq i \leq p+1},(Q_i)_{1 \leq i \leq p+1})$ is a $(p+1)$-dipolygraph,
\item for every $p+2 \leq k \leq n$, $Q_{k}$ is a cellular extension of the $(k-1,p)$-category
\[
([P_{p}]_{Q_{p+1}})(P_{p+1})_{Q_{p+2}} \ldots(P_{k-1}),
\]
\item for every $p+2\leq k \leq n$, $P_{k}$ is a cellular extension of the $(k-1,p)$-category
\[
((([P_{p}]_{Q_{p+1}})(P_{p+1}))_{Q_{p+2}} \ldots(P_{k-1}))_{Q_k}.
\]
\end{enumerate}

\subsubsection{}
We define a \emph{morphism of $(n,p)$-dipolygraphs}
\[
((P_i)_{0 \leq i \leq n}, (Q_i)_{1 \leq i \leq n})
\fl 
((P'_i)_{0 \leq i \leq n}, (Q'_i)_{1 \leq i \leq n})
\]
as a family of pairs $((f_k, g_k))_{1\leq k \leq n}$, where $f_k: P_k \fl P'_k$ and $g_k: Q_k \fl Q'_k$ are maps such that the following diagrams commute
\[ 
\xymatrix@R=2.5em@C=2.5em{
Q_k 
 \ar@<.5ex> [r] 
 \ar@<-.5ex>[r]
 \ar [d] _-{g_k}  & [P_{k-2}]_{Q_{k-1}} [P_{k-1}] 
 \ar [d] ^-{\widetilde{f}_{k-1}} \\
Q'_k 
  \ar@<.5ex> [r] 
 \ar@<-.5ex> [r] &
   [P'_{k-2}]_{Q'_{k-1}} [P'_{k-1}] } 
\qquad
\qquad
\xymatrix@R=2.5em@C=2.5em{
P_k 
 \ar@<.5ex> [r] 
 \ar@<-.5ex>[r]
 \ar [d] _-{f_k}  &
[P_{k-1}]_{Q_{k}}
 \ar [d] ^-{[f_{k-1}]_{g_k}} \\
P'_k 
  \ar@<.5ex> [r] 
 \ar@<-.5ex> [r] &
   [P'_{k-1}]_{Q'_{k}} }
\]
for every $1 \leq k \leq p+1$, and such that the following diagrams commute
\[ \xymatrix@R=2.5em@C=2.5em{
Q_k 
 \ar@<.5ex> [r] 
 \ar@<-.5ex>[r]
 \ar [d] _-{g_k}  & ([P_{p}]_{Q_{p}})(P_{p+1})_{Q_{p+2}} \ldots(P_{k-1})
 \ar [d] ^-{\widetilde{f}_{k-1}} \\
Q'_k 
  \ar@<.5ex> [r] 
 \ar@<-.5ex> [r] &
   ([P'_{p}]_{Q'_{p}})(P'_{p+1})_{Q'_{p+2}} \ldots(P'_{k-1}) } \qquad
\xymatrix@R=2.5em@C=2.5em{
P_k 
 \ar@<.5ex> [r] 
 \ar@<-.5ex>[r]
 \ar [d] _-{f_k}  &
((([P_{p}]_{Q_{p+1}})(P_{p+1}))_{Q_{p+2}} \ldots(P_{k-1}))_{Q_k}
 \ar [d] ^-{[f_{k-1}]_{g_k}} \\
P'_k 
  \ar@<.5ex> [r] 
 \ar@<-.5ex> [r] &
 ((([P'_{p}]_{Q'_{p+1}})(P'_{p+1}))_{Q'_{p+2}} \ldots(P'_{k-1}))_{Q'_k} }
\]
for every $p+2 \leq k \leq n$, where the map $\widetilde{f}_{k-1}$ is induced by the map $f_{k-1}$ and the map $[f_{k-1}]_{g_k}$ is defined by the following commutative diagram:
\[ 
\xymatrix@R=2.5em@C=2.5em{
(([P_{p}]_{Q_{p+1}})(P_{p+1}))_{Q_{p+2}} \ldots(P_{k-1}) 
 \ar [r] ^-{\pi}
 \ar [d] _-{\widetilde{f}_{k-1}}  &
((([P_{p}]_{Q_{p+1}})(P_{p+1}))_{Q_{p+2}} \ldots(P_{k-1}))_{Q_k}
 \ar [d] ^-{[f_{k-1}]_{g_k}} \\
(([P'_{p}]_{Q'_{p+1}})(P'_{p+1}))_{Q'_{p+2}} \ldots(P'_{k-1})
   \ar [r] _-{\pi '} &
 ((([P'_{p}]_{Q'_{p+1}})(P'_{p+1}))_{Q'_{p+2}} \ldots(P'_{k-1}))_{Q'_k} }
\]
We denote by $\DiPol_{(n,p)}$ the category of $(n,p)$-dipolygraphs and their morphisms.

\subsubsection{Presentations by dipolygraphs}
The $(n-1)$-category presented by an $n$-dipolygraph $(P,Q)$ is defined by
\[
\cl{(P,Q)} := ([P_{n-1}]_{Q_n})_{P_n}.
\]
A \emph{presentation of an $(n-1)$-category $\Cr$} is an $n$-dipolygraph $(P,Q)$ whose presented category $\cl{(P,Q)}$ is isomorphic to $\Cr$. 
A \emph{coherent presentation of $\Cr$} is an $(n+1,n-1)$-dipolygraph $(P,Q)$ such that
\begin{enumerate}[{\bf i)}]
\item the underlying $n$-dipolygraph $(P_{\leq n},Q_{\leq n})$ is a presentation of $\Cr$, 
\item the cellular extension $P_{n+1}$ is acyclic and the cellular extension $Q_{n+1}$ is empty.
\end{enumerate}

\subsection{Double coherent presentations}
\label{SS:DoubleCoherentPresentations}

We introduce the notion of a double coherent presentation of an $n$-category. 
We first make explicit the construction of a free $n$-category enriched in double categories generated by a double $(n+2)$-polygraph.

\subsubsection{}  {\bf What is a free double category like ?}
The question of the construction of free double categories was considered in several works  \cite{DawsonPare93,Dawson95,DawsonPare02,DawsonParePronk04}. In particular, Dawson and Pare gave in \cite{DawsonPare02} constructions of free double categories generated by double graphs and double reflexive graphs. Such free double categories always exist, and they show how to describe their cells explicitly in geometrical terms. However, they show that free double categories generated by double graphs cannot describe many of the possible compositions in free double categories.  They fixed this problem by considering double reflexive graphs as generators. 

\subsubsection{}
The coherence results in Section~\ref{S:CoherentCompletionModulo} will be formulated in free $n$-categories enriched in double categories generated by double $(n+2)$-polygraphs.
For every $n\geq 0$, let us consider the forgetful functor
\begin{eqn}{equation}
\label{E:forgetfulW}
W_n : \text{$\DbCat$-$\Cat_n$} \fl \dbpol_{n+2}
\end{eqn}
sending an $n$-category enriched in double categories $\Cr$ to the double $(n+2)$-polygraph, denoted by
\[
W_n(\Cr) = (W_{n+1}^v(\Cr), W_{n+1}^h(\Cr),W_{n+2}^s(\Cr)),
\] 
where $W_{n+1}^v(\Cr)$ (resp. $W_{n+1}^h(\Cr)$) is the underlying $(n+1)$-polygraph of the $(n+1)$-category obtained as the extension of the underlying $n$-category of $\Cr$ by the vertical (resp. horizontal) $(n+1)$-cells and~$W_{n+2}^s(\Cr)$ is the square extension generated by all squares of $\Cr$. Explicitly, for $\mu\in\{v,h\}$, consider~$\Cr_{n+1}^{\mu}$ the $(n+1)$-category whose
\begin{enumerate}
\item underlying $(n-1)$-category coincides with the underlying $(n-1)$-category of $\Cr$,
\item set of $n$-cells is given by
\[
(\Cr_{n+1}^{\mu})_n := \coprod_{x,y\in \Cr_{n-1}} (\Cr_{n}(x,y))^o,
\]
\item set of $(n+1)$-cells is given by
\[
(\Cr_{n+1}^{\mu})_{n+1} := \coprod_{x,y\in \Cr_{n-1}} (\Cr_{n}(x,y))^\mu.
\]
\end{enumerate}
The $(n-1)$-composite of $n$-cells and $(n+1)$-cells of $\Cr_{n+1}^{\mu}$ are defined by enrichment.
The $n$-composite of $(n+1)$-cells of  $\Cr_{n+1}^{\mu}$ are induced by the composition $\circ^{\mu}$.
We define $W_{n+1}^\mu(\Cr)$ as the underlying $(n+1)$-polygraph of the $(n+1)$-category $\Cr_{n+1}^{\mu}$ :
\[
W_{n+1}^\mu(\Cr) := U_{n+1}^{Pol}(\Cr_{n+1}^{\mu}). 
\]
Finally, the square extension $W_{n+2}^s(\Cr)$ is defined on the pair of $(n+1)$-categories $(\Cr_{n+1}^v, \Cr_{n+1}^h)$ by
\[ W_{n+2}^s(\Cr) := \coprod\limits_{x,y \in \Cr_{n-1}} \Cr_n(x,y)^s. \]

\begin{proposition}
For every $n\geq 0$, the forgetful functor $W_n$ defined in~\eqref{E:forgetfulW} admits a left adjoint functor~$F_n : \dbpol_{n+2} \fl \text{$\DbCat$-$\Cat_n$}$.
\end{proposition}

The proof is based on an explicit construction of the free $n$-category enriched in double categories generated by a double $(n+2)$-polygraph given in \eqref{SSS:freeNcategoryEnrichedDoubleGroupoids} and proof \eqref{SSS:UniversalPropertyFreeness} of universal property of free object. 

\subsubsection{Free $n$-categories enriched in double categories}
\label{SSS:freeNcategoryEnrichedDoubleGroupoids}
Let $P$ be a double $(n+2)$-polygraph. 
We construct the \emph{free $n$-category enriched in double categories on $P$}, denoted by $\cck{P}$, as follows:
\begin{enumerate}[{\bf i)}]
\item the underlying $n$-category of $\cck{P}$ is the free $n$-category $P^\ast_n$,
\item for all $(n-1)$-cells $x$ and $y$ of $P^\ast_{n-1}$, the hom-double category $\cck{P}(x,y)$ is constructed as follows
\begin{enumerate}[{\bf a)}]
\item its point cells are the $n$-cells in $P^\ast_n(x,y)$,
\item its vertical (resp. horizontal) cells are the $(n+1)$-cells of the free $(n+1)$-category $(P^v)^\ast$ (resp. $(P^h)^{\ast}$) with $(n-1)$-source $x$ and $(n-1)$-target $y$,
\item its square cells are defined recursively and contains
\begin{enumerate}[{\bf --}]
\item the square cells $A$ of $P^s$ such that $\st_{-,n-1}(A)=x$ and $\st_{+,n-1}(A)=y$,
\item the square cells $C[A]$ for all context $C$ of the $n$-category $P_n^\ast$ and $A$ in $P^s$, such that $\st_{-,n-1}(C[A])=x$ and $\st_{+,n-1}(C[A])=y$,
\item  identities square cells $i_1^h(f)$ and $i_1^v(e)$, for all $(n+1)$-cells $f$ in $(P^h)^\ast$ and $e$ in~$(P^v)^\ast$ whose $(n-1)$-source (resp. $(n-1)$-target) in $P_{n-1}^\ast$ is $x$ (resp. $y$), 
\item all formal pastings of these elements with respect to compositions $\cdg^h$ and $\cdg^v$.
\end{enumerate}
\item two square cells constructed as such formal pastings are identified modulo associativity, identity axioms of compositions $\cdg^v$ and $\cdg^h$ and the middle four interchange law~\eqref{E:MiddleFourIdentities},
\end{enumerate}
\item for all $(n-1)$-cells $x,y,z$ of $P^\ast_{n-1}$, the composition functor 
\[
\star_{n-1} : \cck{P}(x,y)\times \cck{P}(y,z) \fll \cck{P}(x,z)
\]
is defined for all
\[
\raisebox{0.8cm}{
\xymatrix @R=2em @C=2em{ 
u_1
\ar[r] ^-{f_1} _{}="src" 
\ar[d] _-{e_1} 
& 
v_1
\ar[d]^-{e'_1}  
\\
u'_1 
\ar[r] _-{g_1} ^{}="tgt" 
&
v'_1 
\ar@2 "src"!<0pt,-10pt>;"tgt"!<0pt,10pt>  ^{A_1}
}}
\quad\text{in $\cck{P}(x,y)$,\; and }\quad
\raisebox{0.8cm}{
\xymatrix @R=2em @C=2em{ 
u_2
\ar[r] ^-{f_2} _{}="src" 
\ar[d] _-{e_2} 
& 
v_2
\ar[d]^-{e'_2} 
\\
u'_2 
\ar[r] _-{g_2} ^{}="tgt" 
&
v'_2 
\ar@2 "src"!<0pt,-10pt>;"tgt"!<0pt,10pt>  ^{A_2}
}}
\quad\text{in $\cck{P}(y,z)$,}\quad
\]
by
\[
\xymatrix@C=2.5em@R=1.15em{ 
u_1\star_{n-1}u_2
\ar[rr] ^{f_1\star_{n-1}f_2} _{}="src" 
\ar[dd] _{e_1\star_{n-1}e_2} 
& & 
v_1\star_{n-1}v_2
\ar[dd]^{e'_1\star_{n-1}e'_2} 
\\
& & 
\\
u'_1 \star_{n-1} u'_2
\ar[rr] _{g_1\star_{n-1} g_2} ^{}="tgt" 
& & 
v'_1 \star_{n-1} v'_2
\ar@2 "src"!<-7pt,-15pt>;"tgt"!<-7pt,15pt>  ^{A_1\star_{n-1} A_2}
}
\]
where the square cell $A_1\star_{n-1} A_2$ is defined recursively using exchanges relations~(\ref{E:ExchangeRelations1}--\ref{E:ExchangeRelations2}) from functoriality of the composition $\star_{n-1}$, and the middle four interchange law~\eqref{E:MiddleFourIdentities},
\item for all $(n-1)$-cell $x$ of $P^\ast_{n-1}$, the identity map 
$\T \fll \cck{P}(x,x)$, where $\T$ is the terminal double groupoid, sends the one point cell $\bullet$ to $x$ and the identity $i_\alpha^\mu(\bullet)$ on $i_\alpha^\mu(x)$ for all $\mu\in\{v,h\}$ and $\alpha\in\{0,1\}$.
\end{enumerate}

\subsubsection{}
\label{SSS:UniversalPropertyFreeness}
The functor $F_n : \dbpol_{n+2} \fl  \text{$\DbCat$-$\Cat_n$}$, defined by $F_n(P)=\cck{P}$ for every double $(n+2)$-polygraph~$P$, satisfies the universal property of a free object in $\DbCat$-$\Cat_n$. 
Indeed, given a double $(n+2)$-polygraph $P$, a morphism $\eta_P : P \fl W_n(F_n(P))$ of double $(n+2)$-polygraphs, an~$n$-category enriched in double categories $\Cr$, and a morphism $\varphi : P \fl W_n(\Cr)$ of double $(n+2)$-polygraphs, there exists a unique enriched morphism $\widetilde{\varphi} : F_n(P) \fl \Cr$ such that the following diagram commutes
\[
\xymatrix@C=3em{
P \ar[r] ^-{\eta_P} 
  \ar[dr] _-{\varphi}
&
W_n(F_n(P))
  \ar[d] ^-{W_n(\widetilde{\varphi})}
\\
& W_n(\Cr)
}
\]
The functor $\widetilde{\varphi} = (\widetilde{\varphi}_k)_{0\leq k \leq n+2}$ is defined as follows.
\begin{enumerate}[{\bf i)}]
\item By construction, the morphism $\varphi$ induces morphisms of $(n+1)$-polygraphs $\varphi^{\mu} : P^{\mu} \fl W_{n+1}^{\mu}(\Cr)$, for $\mu\in\{v,h\}$. The morphism $\varphi^\mu$ extends by universal property of free $(n+1)$-categories into a functor $\widetilde{\varphi}^{\mu} : (P^{\mu})^\ast \fl \Cr_{n+1}^\mu$. 
We set $\widetilde{\varphi}_k = \varphi^{v}_k=\varphi^{h}_k$ for $0\leq k \leq n$, and
\[
\widetilde{\varphi}_{n+1}(f) = \varphi^{h}(f),
\qquad
\widetilde{\varphi}_{n+1}(e) = \varphi^{v}(e),
\]
for every horizontal $(n+1)$-cell $f$ and every vertical $(n+1)$-cell $e$.
\item By construction, every square $(n+2)$-cell $A$ in $F_n(P)$ is a composite of generating square $(n+2)$-cells in $P^s$ with respect to the compositions $\cdg^v$, $\cdg^h$ and $\star_{n-1}$.
Moreover, following {\cite[Thm. 1.2]{DawsonPare93}}, if a compatible arrangement of square cells in a double category is composable in two different ways, the results are equal modulo the associativity, identity axioms of compositions $\cdg^v$ and $\cdg^h$, and the middle four interchange law~\eqref{E:MiddleFourIdentities}. We extend the functor $\varphi$ to the functor $\widetilde{\varphi}$ by setting
\[
\widetilde{\varphi}(A \cdg^{\mu} B) = \varphi(A) \cdg^{\mu} \varphi(B),
\qquad
\widetilde{\varphi}(A \star_{n-1} B) = \varphi(A) \star_{n-1} \varphi(B),
\]
for all $\mu\in\{v,h\}$ and square generating $(n+2)$-cells $A,B$ in $P^s$ whenever the composites are defined.
\end{enumerate}

\subsubsection{Free $n$-categories enriched in double groupoids}
With a construction similar to the one for the free $n$-category enriched in double categories on a double $(n+2)$-polygraph given in \eqref{SSS:freeNcategoryEnrichedDoubleGroupoids}, we
construct the free $n$-category enriched in double groupoids generated by a double $(n+2,n)$-polygraph$P$, that we denote by $\dck{P}$.
It is obtained as the free $n$-category enriched in double categories~$\cck{P}$ having in addition 
\begin{enumerate}[{\bf i)}]
\item inverse vertical $(n+1)$-cells $e^-$ for every vertical $(n+1)$-cell $e$ in~$\cck{P}$,
\item inverse horizontal $(n+1)$-cells $f^-$ for every horizontal $(n+1)$-cell $f$ in~$\cck{P}$,
\item inverse square $(n+2)$-cells $A^{-,\mu}$ for every square $(n+2)$-cell $A$ in~$\cck{P}$, 
\end{enumerate}
\noindent satisfying the inverses axioms of groupoids for vertical and horizontal cells and relations~\eqref{E:RelationsInverseSquareCell} for square cells. 

Finally, we will also consider the free $n$-category enriched in double categories, whose vertical category is a groupoid, generated by a double $(n+2)$-polygraph $P$ in $\dbpol_{n+2}^v$, that we denote by $\dvck{P}$. In that case, we only require the invertibility of vertical $(n+1)$-cells and the invertibility of square $(n+2)$-cells with respect to $\cdg^h$-composition.

\subsubsection{Acyclicity}
\label{SSS:acyclicity}
Let $P$ be a double $(n+2,n)$-polygraph. The square extension $P^{s}$ of the pair of $(n+1,n)$-categories $(\tck{(P^{v})},\tck{(P^{h})})$ is \emph{acyclic} if for every square $S$ over $(\tck{(P^{v})},\tck{(P^{h})})$ there exists a square $(n+2)$-cell $A$ in the free $n$-category enriched in double groupoids $\dck{P}$ such that $\st(A)=S$.
For example, the set of squares over $(\tck{(P^{v})},\tck{(P^{h})})$ forms an acyclic extension.

\subsubsection{Double coherent presentations of $n$-categories}
Recall that a \emph{presentation of an $n$-category $\Cr$} is an $(n+1)$-polygraph $P$ whose presented category $\cl{P}$ is isomorphic to $\Cr$.
We define a \emph{double coherent presentation of $\Cr$} as a double $(n+2,n)$-polygraph $P$ satisfying the following two conditions:
\begin{enumerate}[{\bf i)}]
\item the $(n+1)$-polygraph $(P_{\leq n},P^{v}_{n+1}\cup P^{h}_{n+1})$ is a presentation of $\Cr$, 
\item the square extension $P^{s}$ is acyclic.
\end{enumerate}

\subsection{Globular coherent presentations from double coherent presentations}
\label{SS:GlobularCoherentPresentationsFromDoubleCoherentPresentations}

\subsubsection{}
\label{SSS:CoherenceFromDoubleCoherence}
We define a quotient functor 
\begin{eqn}{equation}
\label{E:QuotientFunctorV}
V: \dbpol_{(n+2,n)} \fl \DiPol_{(n+2,n)}
\end{eqn}
that sends a double $(n+2,n)$-polygraph $P$ to the $(n+2,n)$-dipolygraph
\begin{eqn}{equation}
\label{E:QuotientFunctorV(P)}
V(P)= ((P_0, \dots, P_{n+2}), (Q_1,\ldots,Q_{n+2})) 
\end{eqn}
defined as follows:
\begin{enumerate}[{\bf i)}]
\item $(P_0,\ldots, P_n)$ is the underlying $n$-polygraph $P_{\leq n}^v=P_{\leq n}^h$,
\item for every $1\leq i \leq n$, the cellular extension $Q_i$ is empty,
\item $Q_{n+1}$ is the cellular extension 
$\xymatrix{ P_{n+1}^v \ar@<.55ex> [r] ^-{\st_{-,n}^v}
\ar@<-.55ex> [r] _-{\st_{+,n}^v} & P_n^\ast}$,
\item $P_{n+1}$ is the cellular extension 
$\xymatrix{ P_{n+1}^h           
\ar@<.55ex> [r] ^-{\widetilde{\st}_{-,n}^h}
\ar@<-.55ex> [r] _-{\widetilde{\st}_{+,n}^h} & (P_n^\ast)_{P_{n+1}^v} }$, where $\widetilde{\st}_{\mu,n}^h := \st_{\mu,n}^h \circ \pi$, for every $\mu$ in $\{ -,+ \}$, and $\pi: P_n^\ast \twoheadrightarrow (P_n^\ast)_{P_{n+1}^v}$ denotes the canonical projection sending an~$n$-cell $u$ in $P_n^\ast$ to its class modulo $P_{n+1}^v$, denoted by $[u]^v$. For an $(n+1)$-cell $f : u \fl v$ in $P_{n+1}^h$, we denote by $[f]^v : [u]^v \fl [v]^v$ the corresponding element in $P_{n+1}$,
\item the cellular extension $Q_{n+2}$ is empty,
 \item $P_{n+2}$ is defined as the cellular extension $\xymatrix{ P^s        
\ar@<.55ex> [r] ^-{\check{s}}
\ar@<-.55ex> [r] _-{\check{t}} & (P_n^\ast)_{P_{n+1}^v} (P_{n+1}^h)}$, where the maps $\check{s}$ and $\check{t}$ are defined by the following commutative diagrams:
\[ \xymatrix@R=3em@C=3em{
P_s 
 \ar@<+0.5ex>[d] ^(.5){\st_{-,n+1}^h}
 \ar@<-0.5ex>[d] _(.5){\st_{+,n+1}^h}
 \ar@/^5ex/@<+0.5ex>[dr] ^-{\check{s}}
 \ar@/^5ex/@<-0.5ex>[dr] _-{\check{t}}
 & \\
 (P_{n+1}^h)^\top 
  \ar[r] _-{F}
  \ar@<+0.5ex>[d] ^(.5){\st_{-,n}^{h}}
  \ar@<-0.5ex>[d] _(.5){\st_{+,n}^{h}}  
 & 
 (P_n^\ast)_{P_{n+1}^v} (P_{n+1}^h) 
  \ar@<+0.5ex>[d] ^(.5){\cl{\widetilde{\st}}_{-,n}^h}
  \ar@<-0.5ex>[d] _(.5){\cl{\widetilde{\st}}_{+,n}^h} \\
 P_n^\ast \ar[r] _-{\pi} & (P_n^\ast)_{P_{n+1}^v} } \]
where the maps $\cl{\widetilde{\st}}_{-,n}^h$ and $\cl{\widetilde{\st}}_{+,n}^h$ are induced from $\widetilde{\partial}_{-,n}^h$ and $\widetilde{\partial}_{+,n}^h$, and the $(n+1)$-functor $F$ is defined by:
\begin{enumerate}[{\bf a)}]
\item $F$ is the identity functor on the underlying $(n-1)$-category $P_{n-1}^\ast$,
\item $F$ sends an $n$-cell $u$ in $P_n^\ast$ to its equivalence class $[u]^v$ modulo $P_{n+1}^v$,
\item $F$ sends an $(n+1)$-cell $f:u \fl v$ in $\tck{(P_{n+1}^h)}$ to the $(n+1)$-cell $[f]^v: [u]^v \fl [v]^v$ in $(P_n^\ast)_{P_{n+1}^v}(P_{n+1}^h)$ defined as follows
\begin{enumerate}[{\bf --}]
\item for every $f$ in $P_{n+1}^h$, $[f]^v$ is defined by {\bf iv)},
\item $F$ is extended to the $(n+1)$-cells of $\tck{(P_{n+1}^h)}$ by functoriality by setting
\[
[\,C[g]\,]^v
= 
[x_n]^v\star_{n-1} x_{n-1} \star_{n-2} \ldots \star_1(x_1 \star_0 [g]^v \star_0 y_1)\star_1 \ldots \star_{n-2} y_{n-1}\star_{n-1} [y_n]^v,  
\]
for all whisker $C=x_n \star_{n-1} \ldots \star_1(x_1\star_0 - \star_0 y_1)\star_1 \ldots \star_{n-1} y_n$ of $\tck{(P_{n+1}^h)}$ and $(n+1)$-cell $g$ in $\tck{(P_{n+1}^h)}$, and
\[
[f_1\star_n f_2 ]^v = [f_1]^v \star_n [f_2]^v,
\]
for all $(n+1)$-cells $f_1,f_2$ in $\tck{(P_{n+1}^h)}$.
\end{enumerate}
\end{enumerate}
\end{enumerate}

\subsubsection{}
\label{SSS:NotationCrochetV}
Given a generating square $(n+2)$-cell  
\[
\xymatrix @R=2em @C=2em{
u
  \ar[r] ^-{f} ^-{}="1"
  \ar[d] _-{g}
&
u' 
  \ar[d] ^-{k}
\\
v
  \ar[r] _-{h} _-{}="2"
&
v' 
\ar@2 "1"!<0pt,-10pt>;"2"!<0pt,10pt> ^-{A} }
\] 
of $P^s$, we denote by $[A]^v$ the generating $(n+2)$-cell of the globular cellular extension $P_{n+2}$ on $(P_n^\ast)_{P_{n+1}^v} (P_{n+1}^h)$ defined in~\eqref{E:QuotientFunctorV(P)} as follows:
\[ 
\xymatrix{
[u]^v = [u']^v 
     \ar @/^5ex/ [rr] ^{[f]^v} _-{}="1"
	\ar @/_5ex/ [rr] _{[g]^v}  ^-{}="2" 
	 & & 
	 [v]^v = [v']^v
\ar@2 "1"!<0pt,-10pt>;"2"!<0pt,10pt> ^-{[A]^v}
} 
\]
Note that by construction, in the $(n+2,n)$-category $((P_n^\ast)_{P_{n+1}^v}(P_{n+1}^h))(P_{n+2})$ the relations
\[ 
[A]^v \star_n [A']^v = [A \cdg^v A']^v, \qquad [A]^v \star_{n+1} [A']^v = [A \cdg^h A']^v,
\]
hold for all generating square $(n+2)$-cells $A,A'$ in $P^s$ such that these compositions make sense. 

\begin{proposition}
\label{P:QuotientAcyclic}
Let $P$ be a double $(n+2,n)$-polygraph. If the square extension $P^s$ is acyclic then the cellular extension $P_{n+2}$  of the $(n+1)$-category $(P_n^\ast)_{P_{n+1}^v} (P_{n+1}^h)$ defined in~\eqref{E:QuotientFunctorV(P)} is acyclic.

In particular, if $P$ is a double coherent presentation of an~$n$-category $\Cr$. 
Then, the $(n+2,n)$-dipolygraph $V(P)$ is a globular coherent presentation of the quotient $n$-category $(P_n^\ast)_{P_{n+1}^v}$, that is the $n$-category is isomorphic to $\cl{V(P)}_{\leq (n+1)}$ and $P_{n+2}$ is an acyclic extension of $(P_n^\ast)_{P_{n+1}^v}(P_{n+1}^h)$.
\end{proposition}
\begin{proof}
Given an $(n+1)$-sphere $\gamma := ([f]^v,[g]^v)$ in $(P_n^\ast)_{P_{n+1}^v} (P_{n+1}^h)$, by definition of the functor $V$ in~\eqref{E:QuotientFunctorV}, there exists an~$(n+1)$-square
\[
S := \raisebox{0.65cm}{
\xymatrix @R=2em @C=2em{
u
  \ar[r] ^-{f} ^-{}="1"
  \ar[d] _-{e}
&
u' 
  \ar[d] ^-{e'}
\\
v
  \ar[r] _-{g} _-{}="2"
&
v' 
}}
\]
in $(\tck{(P^v_{n+1})},\tck{(P^h_{n+1})})$, such that $F(f) = [f]^v$ and $F(g) = [g]^v$ and $V(S)=\gamma$.
By acyclicity assumption, there exists a square $(n+2)$-cell $A$ in the free $n$-category enriched in double groupoids $\dck{(P^v,P^h,P^s)}$ such that $\partial(A) = S$. Then $[A]^v$ is an~$(n+2)$-cell in $(P_n^\ast)_{P_{n+1}^v}(P_{n+1}^h))(P_{n+2})$ such that $\st([A]^v) = \gamma$.
Finally, the fact that $V(P)_{\leq (n+1)}$ is a presentation of the quotient $n$-category $(P_n^\ast)_{P_{n+1}^v}$ follows from the definition of the functor $V$ and the fact that the $(n+1)$-polygraph $(P_n,P_{n+1}^v\cup P_{n+1}^h)$ is a presentation of~$\Cr$.
\end{proof}

\subsection{Examples}

\def\gppol{\mathsf{Gp}}
\def\compol{\mathsf{Com}}
\def\pivpol{\mathsf{Piv}}

We show how to define coherent presentations of algebraic structures in terms of dipolygraphs in the cases of groups, commutative monoids and pivotal categories.

\subsubsection{Coherent presentations of groups}
\label{SSS:CoherentPresentationGroups}
A presentation of a group $G$ by generators $X$ and relations $R$ is a $(2,1)$-dipolygraph $((P_0,P_1,P_2),(Q_1,Q_2))$ such that
\begin{enumerate}[{\bf i)}]
\item $P_0$ is a singleton, the cellular extension $Q_1$ is empty, and the generating $1$-cells in $P_1$ are elements of~$X$ seen as loops on the $0$-cell,
\item the cellular extension $Q_2$ of $P_1^\ast$ is made of the generating $2$-cells 
\[
x x^- \dfl  1,
\qquad
x^- x \dfl 1,
\]
for every $x$ in $P_1$,
\item the cellular extension $P_2$ on the free group ${P_1}^\ast_{Q_2}$ is made of generating $2$-cells of the form $r \dfl 1$ and $r^- \dfl 1$, for $r=1$ being a relation in $R$. 
\end{enumerate}
A  coherent presentation of $G$ is a $(3,1)$-dipolygraph $((P_i)_{0 \leq i \leq 3},(Q_j)_{1 \leq j \leq 3})$ such that the underlying $(2,1)$-dipolygraph $((P_0,P_1,P_2),(Q_1,Q_2))$ is a presentation of $G$, the cellular extension $Q_3$ is empty, and $P_3$ is an acyclic extension of ${P_1}^\ast_{Q_2}(P_2)$.

\subsubsection{Coherent presentation of commutative monoids}
\label{SSS:CoherentPresentationCommutativeMonoid}

A presentation of a commutative monoid $M$ by generators $X$ and relations $R$ corresponds to a $(2,1)$-dipolygraph $((P_0,P_1,P_2),(Q_1,Q_2))$ such that
\begin{enumerate}[{\bf i)}]
\item $P_0$ is a singleton, the cellular extension $Q_1$ is empty, and the generating $1$-cells in $P_1$ are elements of~$X$ seen as loops on the $0$-cell,
\item the cellular extension $Q_2$ of $P_1^\ast$ is made of the generating $2$-cells \[
\alpha_{i,j} : x_i x_j \dfl x_j x_i,
\] 
for all $x_i,x_j$ in $P_1$, such that $x_i > x_j$ for a given total order $>$ on $P_1$,
\item the cellular extension $P_2$ on the free commutative monoid ${P_1}^\ast_{Q_2}$ is made of relations in $R$ with a chosen orientation.
\end{enumerate}
A  coherent presentation of $M$ is a $(3,1)$-dipolygraph $((P_i)_{0 \leq i \leq 3},(Q_j)_{1 \leq j \leq 3})$ such that the underlying $(2,1)$-dipolygraph $((P_0,P_1,P_2),(Q_1,Q_2))$ corresponds to a presentation of $M$, the cellular extension $Q_3$ is empty, and $P_3$ is an acyclic extension of the $2$-category ${P_1}^\ast_{Q_2}(P_2)$.

\subsubsection{Coherent presentation of monoidal pivotal categories}
Recall that a (strict monoidal) pivotal category $\Cr$ is a monoidal category, seen as $2$-category with only one $0$-cell, in which every $1$-cell $p$ has a right dual $1$-cell $\hat{p}$, which is also a left-dual, that is there are $2$-cells 
\begin{eqn}{equation}
\label{E:2-cellsEtaEpsilon}
\eta_p^- : 1 \dfl \hat{p} \star_0 p,
\quad
\eta_p^+ : 1 \dfl p \star_0 \hat{p},
\quad
\varepsilon_p^- : \hat{p} \star_0 p \dfl 1,
\quad
\text{and}
\quad
\varepsilon_p^+ : p \star_0 \hat{p} \dfl 1,
\end{eqn}
respectively represented by the following diagrams: 
\begin{eqn}{equation}
\label{E:UnitCounit}
\cupdb{\hat{p}}{p}, \quad \cupdb{p}{\hat{p}}, \quad \raisebox{2mm}{$\capdb{\hat{p}}{p}$}, \quad \text{and} \quad \raisebox{2mm}{$\capdb{p}{\hat{p}}$}. 
\end{eqn}
These $2$-cells satisfy the relations
\begin{align*}
(\varepsilon_p^+ \star_0 1_p ) \star_1 (1_p \star_0 \eta_p^-)
&=
1_p
=
(1_p \star_0 \varepsilon_p^-) \star_1 (\eta_p^+ \star_0 1_p),
\\
(\varepsilon_p^- \star_0 1_{\hat{p}} ) \star_1 (1_{\hat{p}} \star_0 \eta_p^+)
&=
1_{\hat{p}}
=
(1_{\hat{p}} \star_0 \eta_p^+) \star_1 (\eta_p^- \star_0 1_{\hat{p}}),
\end{align*}
that can be diagrammatically depicted as follows
\[  
\begin{tikzpicture}[baseline = 0,scale=1]
  \draw[-,black,thick] (0.3,0) to (0.3,.4);
	\draw[-,black,thick] (0.3,0) to[out=-90, in=0] (0.1,-0.4);
	\draw[-,black,thick] (0.1,-0.4) to[out = 180, in = -90] (-0.1,0);
	\draw[-,black,thick] (-0.1,0) to[out=90, in=0] (-0.3,0.4);
	\draw[-,black,thick] (-0.3,0.4) to[out = 180, in =90] (-0.5,0);
  \draw[-,black,thick] (-0.5,0) to (-0.5,-.4);
  \node at (-0.2,0.65) {$\varepsilon_p^+$};
  \node at (0.1,-0.65) {$\eta_p^-$};
  \node at (-0.5,-0.6) {$p$};
\end{tikzpicture}  = \; \begin{tikzpicture}[baseline=0, scale=1]
\draw[-,black,thick] (0,-0.4) to (0,0.4) ;
\node at (0,-0.65) {$p$} ;
\end{tikzpicture}
= \; \begin{tikzpicture}[baseline = 0, scale=1]
  \draw[-,black,thick] (0.3,0) to (0.3,-.4);
	\draw[-,black,thick] (0.3,0) to[out=90, in=0] (0.1,0.4);
	\draw[-,black,thick] (0.1,0.4) to[out = 180, in = 90] (-0.1,0);
	\draw[-,black, thick] (-0.1,0) to[out=-90, in=0] (-0.3,-0.4);
	\draw[-,black, thick] (-0.3,-0.4) to[out = 180, in =-90] (-0.5,0);
  \draw[-,black, thick] (-0.5,0) to (-0.5,.4);
   \node at (-0.2,-0.65) {$\eta_p^+$};
  \node at (0.1,0.65) {$\varepsilon_p^-$};
  \node at (0.5,-0.6) {$p$};
\end{tikzpicture},
\qquad \qquad 
\begin{tikzpicture}[baseline = 0,scale=1]
  \draw[-,black, thick] (0.3,0) to (0.3,.4);
	\draw[-,black, thick] (0.3,0) to[out=-90, in=0] (0.1,-0.4);
	\draw[-,black, thick] (0.1,-0.4) to[out = 180, in = -90] (-0.1,0);
	\draw[-,black, thick] (-0.1,0) to[out=90, in=0] (-0.3,0.4);
	\draw[-,black, thick] (-0.3,0.4) to[out = 180, in =90] (-0.5,0);
  \draw[-,black, thick] (-0.5,0) to (-0.5,-.4);
  \node at (-0.2,0.65) {$\varepsilon_p^-$};
  \node at (0.1,-0.65) {$\eta_p^+$};
  \node at (-0.5,-0.65) {$\hat{p}$};
\end{tikzpicture}  = \; \begin{tikzpicture}[baseline=0, scale=1]
\draw[-,thick,black] (0,-0.4) to (0,0.4) ;
\node at (0,-0.65) {$\hat{p}$} ;
\end{tikzpicture}
= \; \begin{tikzpicture}[baseline = 0, scale=1]
  \draw[-,black, thick] (0.3,0) to (0.3,-.4);
	\draw[-,black, thick] (0.3,0) to[out=90, in=0] (0.1,0.4);
	\draw[-,black, thick] (0.1,0.4) to[out = 180, in = 90] (-0.1,0);
	\draw[-,black, thick] (-0.1,0) to[out=-90, in=0] (-0.3,-0.4);
	\draw[-,black, thick] (-0.3,-0.4) to[out = 180, in =-90] (-0.5,0);
  \draw[-,thick,black, thick] (-0.5,0) to (-0.5,.4);
   \node at (-0.2,-0.65) {$\eta_p^-$};
  \node at (0.1,0.65) {$\varepsilon_q^+$};
  \node at (0.4,-0.65) {$\hat{p}$};
\end{tikzpicture}.
\]
Any $2$-cell $f: p \dfl q$ in $\Cr$ is cyclic with respect to the biadjunctions $\hat{p} \vdash p \vdash \hat{p}$ and $\hat{q} \vdash q \vdash \hat{q}$, defined respectively by the family of $2$-cells
$(\eta_p^-, \eta_p^+, \varepsilon_p^-,\varepsilon_p^+)$ and $(\eta_q^-,
\eta_q^+, \varepsilon_q^-,\varepsilon_q^+)$.
That is, $f^\ast = {}^\ast f$, where $f^\ast$ and ${}^\ast f$ are respectively the right and left duals of $f$, defined using the right and left adjunction as follows:
\[ ^* f := \begin{tikzpicture}[baseline = 0,scale=1]
  \draw[-,black, thick] (0.3,0) to (0.3,.4);
	\draw[-,black, thick] (0.3,0) to[out=-90, in=0] (0.1,-0.4);
	\draw[-,black, thick] (0.1,-0.4) to[out = 180, in = -90] (-0.1,0);
	\draw[-,black, thick] (-0.1,0) to[out=90, in=0] (-0.3,0.4);
	\draw[-,black, thick] (-0.3,0.4) to[out = 180, in =90] (-0.5,0);
  \draw[-,black, thick] (-0.5,0) to (-0.5,-.4);
  \node at (-0.25,0.65) {$\varepsilon_q^-$};
  \node at (0.13,-0.65) {$\eta_p^+$};
  \node at (-0.1,0) {$\bullet$};
  \node at (0.05,0) {$\scriptstyle{f}$};
  \node at (0.35,0.65) {$\hat{p}$};
  \node at (-0.5,-0.65) {$\hat{q}$};
\end{tikzpicture},
 \qquad \qquad \qquad
 f ^* := \begin{tikzpicture}[baseline = 0, scale=1]
  \draw[-,black, thick] (0.3,0) to (0.3,-.4);
	\draw[-,black, thick] (0.3,0) to[out=90, in=0] (0.1,0.4);
	\draw[-,black, thick] (0.1,0.4) to[out = 180, in = 90] (-0.1,0);
	\draw[-,black, thick] (-0.1,0) to[out=-90, in=0] (-0.3,-0.4);
	\draw[-,black, thick] (-0.3,-0.4) to[out = 180, in =-90] (-0.5,0);
  \draw[-,thick,black, thick] (-0.5,0) to (-0.5,.4);
   \node at (-0.1,0) {$\bullet$};
  \node at (-0.25,0) {$\scriptstyle{f}$};
  \node at (-0.5,0.65) {$\hat{p}$};
   \node at (-0.2,-0.65) {$\eta_p^-$};
  \node at (0.1,0.65) {$\varepsilon_q^+$};
  \node at (0.35,-0.65) {$\hat{q}$};
\end{tikzpicture}.
 \]
We refer the reader to \cite{JoyalStreet91,CockettKoslowskiSeely00} for more details about the notion of pivotal monoidal category.

A presentation of a (strict) monoidal pivotal category by generating $1$-cells $X_1$, generating $2$-cells $X_2$ and relations $R$ corresponds to a $(3,2)$-dipolygraph $((P_i)_{0 \leq i \leq 3}, (Q_j)_{1 \leq j \leq 3})$ such that
\begin{enumerate}[{\bf i)}]
\item $P_0$ is a singleton, the cellular extension $Q_1$ is empty, and $P_1 = X_1 \cup \widehat{X}_1$, where $\widehat{X}_1 :=\{\hat{p} \;|\; p\in X_1\}$ is the set of bi-duals of elements of $X_1$,
\item the cellular extension $Q_2$ on $P_1^\ast$ is empty, and $P_2 = X_2 \;\cup\; \{\eta_p^-, \eta_p^+, \epsilon_p^-, \epsilon_p^+ \;\; |\;\; p \in X_1 \},$ where the $2$-cells $\eta_p^-$, $\eta_p^+$, $\epsilon_p^-$, $\epsilon_p^+$ are defined by~\eqref{E:2-cellsEtaEpsilon},
\item the cellular extension $Q_3$ on $P_1^\ast[P_2]$ is made of the following generating $3$-cells: 
\[ \begin{tikzpicture}[baseline = 0,scale=1]
  \draw[-,black, thick] (0.3,0) to (0.3,.4);
	\draw[-,black, thick] (0.3,0) to[out=-90, in=0] (0.1,-0.4);
	\draw[-,black, thick] (0.1,-0.4) to[out = 180, in = -90] (-0.1,0);
	\draw[-,black, thick] (-0.1,0) to[out=90, in=0] (-0.3,0.4);
	\draw[-,black, thick] (-0.3,0.4) to[out = 180, in =90] (-0.5,0);
  \draw[-,black, thick] (-0.5,0) to (-0.5,-.4);
  \node at (-0.25,0.65) {$\varepsilon_q^-$};
  \node at (0.13,-0.65) {$\eta_p^+$};
  \node at (-0.1,0) {$\bullet$};
  \node at (0.05,0) {$\scriptstyle{f}$};
  \node at (0.35,0.65) {$\hat{p}$};
  \node at (-0.5,-0.65) {$\hat{q}$};
\end{tikzpicture} \Rrightarrow \begin{tikzpicture}[baseline=0, scale=1]
\draw[-,black,thick] (0,-0.4) to (0,0.4) ;
\node at (0,-0.65) {$\hat{q}$};
\node at (0,0.65) {$\hat{p}$};
\node at (0.23,0) {$\scriptstyle{{}^\ast f}$};
\node at (0,0) {$\bullet$};
\end{tikzpicture},
\qquad \qquad \qquad \begin{tikzpicture}[baseline = 0, scale=1]
  \draw[-,black, thick] (0.3,0) to (0.3,-.4);
	\draw[-,black, thick] (0.3,0) to[out=90, in=0] (0.1,0.4);
	\draw[-,black, thick] (0.1,0.4) to[out = 180, in = 90] (-0.1,0);
	\draw[-,black, thick] (-0.1,0) to[out=-90, in=0] (-0.3,-0.4);
	\draw[-,black, thick] (-0.3,-0.4) to[out = 180, in =-90] (-0.5,0);
  \draw[-,thick,black, thick] (-0.5,0) to (-0.5,.4);
   \node at (-0.1,0) {$\bullet$};
  \node at (-0.25,0) {$\scriptstyle{f}$};
  \node at (-0.5,0.65) {$\hat{p}$};
   \node at (-0.2,-0.65) {$\eta_p^-$};
  \node at (0.1,0.65) {$\varepsilon_q^+$};
  \node at (0.35,-0.65) {$\hat{q}$};
\end{tikzpicture} \Rrightarrow \begin{tikzpicture}[baseline=0, scale=1]
\draw[-,black,thick] (0,-0.4) to (0,0.4) ;
\node at (0,-0.65) {$\hat{q}$};
\node at (0,0.65) {$\hat{p}$};
\node at (0.23,0) {$\scriptstyle{f^\ast}$};
\node at (0,0) {$\bullet$};
\end{tikzpicture},
\]
for every generating $2$-cell $f$ in $X_2$ or $f$ is an identity cell,
\item the cellular extension $P_3$ is made of relations in $R$ with a chosen orientation.
\end{enumerate}
A coherent presentation of $\Cr$ is a $(4,2)$-dipolygrah $((P_i)_{1 \leq i \leq 4},(Q_j)_{1 \leq j \leq 4})$ such that the underlying $(3,2)$-dipolygraph $((P_i)_{1 \leq i \leq 3}, (Q_j)_{1 \leq j \leq 3})$ is a presentation of $\Cr$, the cellular extension $Q_4$  is empty and $P_4$ is an acyclic extension of $P_1^\ast [P_2](P_3)$.

\section{Polygraphs modulo}
\label{S:PolygraphsModulo}

We introduce the notion of polygraphs modulo and define their main rewriting properties.

\subsection{Polygraphs modulo}
\label{SS:PolygraphsModulo}

\subsubsection{Cellular extensions modulo}
\label{SSS:CellularExtensionModulo}
Let $E$ and $R$ be two $n$-polygraphs such that $E_{\leq n-2}=R_{\leq n-2}$ and $E_{n-1} \subseteq R_{n-1}$. Recall that $R_n^{\ast (1)}$ denotes the set of $R_n$-rewriting steps.
We define the cellular extension 
\[
\gamma^{\ER} : \ER \fl \Sph_{n-1}(R_{n-1}^\ast),
\]
where the set $\ER$ is defined by the following pullback in $\Set$:
\[ 
\xymatrix{ \tck{E_n} \times_{R_{n-1}^{\ast}} R_n^{\ast (1)} \ar [d] _-{\pi_1} \ar [r] ^-{\pi_2} & R_n^{\ast (1)} \ar [d] ^-{\st_{-}}  \\
\tck{E_n} \ar [r] _-{\st_{+}} & R_{n-1}^\ast }
\]
and the map $\gamma^{\ER}$ is defined by $\gamma^{\ER}(e,f) = ( \st_{-}(e) , \st_{+}(f) )$, for all $e$ in $\tck{E}_n$ and $f$ in $R_n^{\ast (1)}$.
Similarly, we define the cellular extension 
\[
\gamma^{\RE} : \RE \fl \Sph_{n-1}(R_{n-1}^\ast),
\]
where the set $\RE$ is defined by the following pullback in $\Set$:
\[ 
\xymatrix{ R_n^{\ast (1)} \times_{R_{n-1}^\ast} \tck{E_n} \ar [d] _-{\pi_1} \ar [r] ^-{\pi_2} & \tck{E_n} \ar [d] ^-{\st_{-}}  \\
R_n^{\ast (1)} \ar [r] _-{\st_{+}} & R_{n-1}^\ast }
\]
and the map $\gamma^{\RE}$ is defined by $\gamma^{\RE}(f,e) = ( \st_{-}(f) , \st_{+}(e) )$, for all $e$ in $\tck{E}_n$ and $f$ in $R_n^{\ast (1)}$.
Finally, we define the cellular extension
\[
\gamma^{\ERE} : \ERE \fl \Sph_{n-1}(R_{n-1}^\ast),
\]
where the set $\ERE$ is defined by the following composition of pullbacks in $\Set$:
\[
\xymatrix@C=4em@R=2.5em{
\tck{E_n} \times_{R_{n-1}^\ast} R_n^{\ast (1)} \times_{R_{n-1}^\ast} \tck{E_n} \ar [r] ^-{(\pi_2,\pi_3)} \ar [d] _-{(\pi_1,\pi_2)} & R_n^{\ast (1)} \times_{R_{n-1}^\ast} \tck{E_n} \ar [d] _-{\pi_1} \ar [r] ^-{\pi_2} & \tck{E_n} \ar [d] ^-{\st_{-}}  \\
\tck{E_n} \times_{R_{n-1}^{\ast}} R_n^{\ast (1)} \ar [d] _-{\pi_1} \ar [r] ^-{\pi_2} & R_n^{\ast (1)} \ar [r] _-{\st_{+}} \ar [d] ^-{\st_{-}} & R_{n-1}^\ast \\
\tck{E_n} \ar [r] _-{\st_{+}} & R_{n-1}^\ast & 
} 
\] 
and the map $\gamma^{\ERE}$ is defined by $\gamma^{\ERE}(e,f,e') = (\st_{-}(e),\st_{+}(e'))$. 

\subsubsection{Polygraphs modulo}
\label{SSS:PolygraphsModulo}
An \emph{$n$-polygraph modulo} is a data~$\Pr:=(R,E,S)$ made of
\begin{enumerate}[{\bf i)}]
\item an $n$-polygraph $R$, whose generating $n$-cells are called \emph{primary rules},
\item an $n$-polygraph $E$ such that $E_{\leq (n-2)} := R_{\leq (n-2)}$ and $E_{n-1} \subseteq R_{n-1}$, whose generating $n$-cells are called \emph{modulo rules},
\item a cellular extension $S$ of $R_{n-1}^\ast$ satisfying the following conditions:
\[
R_n
\subseteq
S 
\subseteq 
\ERE.
\]
\end{enumerate}
This means that $S$ contains all the generating $n$-cells of $R_n$ and that every generating $n$-cell in~$S$ can be written $(e,f,e')$ with $e,e'\in\tck{E}_n$ and $f$ in $R_n^{\ast (1)}$. The \emph{$(n-1)$-category presented by $\Pr$}, denoted by $\cl{\Pr}$, is the category presented by the $n$-polygraph $(R_{\leq n-1},E_n\cup R_n)$.

\subsubsection{} 
We will consider in the sequel the following categories with respect to $\Pr$:
\begin{enumerate}[{\bf i)}]
\item the free $n$-category $R_{n-1}^\ast [R_n, E_n \coprod E_n^{-1}] \slash \text{Inv}(E_n, E_n^{-1})$, denoted by $R^\ast(E)$.
\item the free $n$-category generated by $S$, denoted by $S^\ast$, 
\item the free $(n,n-1)$-category generated by $S$, denoted by $\tck{S}$.
\end{enumerate}

\subsubsection{Rewriting and normal forms}
Recall from \eqref{SSS:Reductions} that  the size of an $n$-cell $f$ in $S^\ast$ is the positive integer $\norm{f}_S$ corresponding to the number of elements of $S$ contained in $f$. Seen as a $S$-rewriting path, the size of $f$ corresponds to the length of the path, that is the minimal number of $n$-cells of $S$ needed to write $f$ as an $(n-1)$-composite of elements of $S$. By definition of $S$, we have $\norm{f}_S=\norm{f}_{R_n}$. Let $(u,v)\in\Sph_{n-1}(R_{n-1}^\ast)$, recall that a $S$-rewriting step from $u$ to $v$ is an $n$-cell $f$ in $S^\ast$ with source $u$ and target $v$ such that $\norm{f}_S=1$, and a $S$-reduction path is a sequence $(f_i)_{i\in I}$ of normal $S$-rewriting step such that, for every $i\in I$, $\st_+(f_i)=\st_-(f_{i+1})$. A \emph{$S$-normal form} of an~$(n-1)$-cell $u$ in $R_{n-1}^\ast$ is a $S$-irreducible $(n-1)$-cell $v$ such that $u$ $S$-reduces to~$v$.
We denote by $\irr(S)$ the set of $S$-irreducible $(n-1)$-cells of $R_{n-1}^\ast$, and by $\nf(S,u)$ the set of $S$-normal forms of $u$.

\subsubsection{Square extensions of a polygraph modulo}
A square extension of the pair of $n$-categories $(\tck{E},S^\ast)$ will be called a \emph{square extension} of $\Pr$.
A \emph{coherent extension} of $\Pr$ is an acyclic square extension of the pair of $(n,n-1)$-categories $(\tck{E},\tck{S})$.

\subsection{Termination modulo}
This subsection deals with the property of Noetherianity of polygraphs modulo. In particular, we give a method to prove the termination with respect to an order compatible modulo rules. We also recall the double induction principle introduced by Huet in \cite{Huet80} that we will use in many proofs in the sequel. Let $\Pr=(R,E,S)$ denote an $n$-polygraph modulo. 

\subsubsection{Termination}
The $n$-polygraph modulo $\Pr$ is \emph{terminating} if the $n$-polygraph $(R_{\leq n-1}, \ERE)$ is terminating. When $S \ne R$, the termination of $\Pr$ is equivalent to the termination of the $n$-polygraphs $(R_{\leq n-1}, \RE)$, $(R_{\leq n-1}, \ER)$, and $(R_{\leq n-1},S)$. 
An order relation $\prec$ on the set of $(n-1)$-cells of $R_{n-1}^\ast$ is \emph{compatible with $S$ modulo $E$} if it satisfies the following two conditions for $(n-1)$-cells $u,v$ in $R_{n-1}^\ast$:
\begin{enumerate}[{\bf i)}]
\item $v \prec u$ if there exists an $n$-cell $u \fl v$ in $S^\ast$,
\item if $v\prec u$,  then $v' \prec u'$ holds for all $(n-1)$-cells $u',v'$ in $R_{n-1}^\ast$ such that there exist $n$-cells $e:u \fl u'$ and $e': v \fl v'$ in $\tck{E}_n$. 
\end{enumerate}
A \emph{termination order for $\Pr$} is a well-founded order relation on $R_{n-1}^\ast$, compatible with $S$ modulo $E$.

In this work, many constructions will be based on the termination of the $n$-polygraph modulo $(R,E,\ERE)$, which can be proved by constructing a termination order for one of the $n$-polygraphs $(R_{\leq n-1},\ER)$, $(R_{\leq n-1},\RE)$ and~$(R_{\leq n-1},\ERE)$. It can be also proved by constructing a termination order for $\Pr$. Such an order can be constructed as an order $\prec$ on $R_{n-1}^\ast$, stable by context, satisfying $\partial_+(f) \prec \partial_f (f)$ for every $f$ in $R_n$, and stable by the $n$-cells of $E_n$. 

\subsubsection{Noetherian induction}
\label{SSS:InductionPrincipale}
If $\Pr$ is terminating, then every $(n-1)$-cell of $R_{n-1}^\ast$ has at least one $S$-normal form.
In that case, we can prove a property $\mathbb{P}$ on an $(n-1)$-cells $u$ of $R_{n-1}^\ast$ by \emph{Noetherian induction}. For that, we prove the property $\mathbb{P}$ on $S$-normal forms. Then, we assume that $\mathbb{P}$ holds for every $(n-1)$-cell
$v$ such that $u$ $S$-reduces to $v$, and we prove,
under those hypotheses, that the $(n-1)$-cell $u$ satisfies the
property $\mathbb{P}$.

Let us recall the double Noetherian induction principle introduced by Huet in \cite{Huet80}, and that we will use to prove properties of confluence modulo from local confluence modulo assumptions. 
We consider an auxiliary $n$-polygraph $\Saux$ as follows. 
For $0 \leq k \leq n-1$, we set
\[
\text{$\Saux_k := S_k \times S_k$ }, 
\] 
and $\Saux_n$ is the set of $n$-cells $(u,v) \fl (u',v')$, for all $(n-1)$-cells $u,u',v,v'$ in $R_{n-1}^\ast$ in any of the following situation:
\begin{enumerate}[{\bf i)}]
\itemsep0em
\item there exists an $n$-cell $u \fl u'$ in $S^\ast$ and $v=v'$,
\item there exists an $n$-cell $v \fl v'$ in $S^\ast$ and $u=u'$,
\item there exist $n$-cells $u \fl u'$ and $u \fl v'$ in $S^\ast$,
\item there exist $n$-cells $v \fl u'$ and $v \fl v'$ in $S^\ast$,
\item  there exist $n$-cells $e_1$, $e_2$, and $e_3$ in $\tck{E}$ as in the following diagram
\[ 
\xymatrix{ 
u \ar [r] ^-{e_1} 
& 
v 
\ar [r] ^-{e_2} 
& 
u' 
\ar [r] ^-{e_3} & v' 
} 
\]
such that $\norm{e_1}_E > \norm{e_3}_E$.
\end{enumerate}
Note that this definition implies that, if there exist  $n$-cells $u \fl u'$ and $v \fl v'$ in $S^\ast$, then there is an $n$-cell $(u,v) \fl (u',v')$ in $\Saux$ given by the following composite:
\[ 
(u,v) \fl (u',v) \fl (u',v') 
\]
Following \cite[Prop. 2.2]{Huet80}, if  $\Pr$ is terminating, then so is $\Saux$. 
In the sequel, we will apply this Noetherian induction on $\Saux$ with the following property: 
\begin{quote}
\emph{for all $n$-cells $f: u \fl u'$, $g: v \fl v'$ in $S^\ast$ and $e: u \fl v$ in $\tck{E}$, there exist $n$-cells $f': u' \fl u''$, $g': v' \fl w''$ in $S^\ast$ and $e': u'' \fl w''$ in $\tck{E}$, and a square $(n+1)$-cell $A$ in a given $(n-1)$-category enriched in groupoids, as depicted in the following diagram:
\[
\xymatrix @R=2em @C=2em{
u
  \ar[r] ^-{f}
  \ar[d] _-{e}
&
u' 
  \ar@{.>}[r] ^-{f'} 
& 
u''
  \ar@{.>}[d] ^-{e'}
\\
v
  \ar[r] _-{g}
&
v'
  \ar@{.>}[r] _-{g'} 
&
w''
\ar@2 "1,2"!<0pt,-8pt>;"2,2"!<0pt,8pt> ^-{A}
}
\]
}
\end{quote}
In Section \ref{S:CoherentConfluenceModulo}, we will formulate this property for a branching $(f,e,g)$ of $\Pr$ in terms of coherent confluence modulo.

\subsection{Confluence modulo}
\label{S:ConfluenceModulo}
This subsection deals with properties of confluence and local confluence of polygraphs modulo. We define the notion of branching for polygraphs modulo and we give a classification of the local branchings.
Let $\Pr=(R,E,S)$ denote an $n$-polygraph modulo.  

\subsubsection{Branchings}
A \emph{strict $S$-branching} is an pair~$(f,g)$, where $f,g$ are $n$-cells of $S^\ast$ such that $\st_{-,n-1}^h(f)=\st_{-,n-1}^h(g)$, and depicted by 
\begin{eqn}{equation}
\label{E:branching}
\raisebox{0.5cm}{
\xymatrix @R=2em @C=2em{
u 
  \ar[r] ^-{f} 
  \ar[d] _-{\rotatebox{90}{=}}
& 
u'
\\
u
  \ar[r] _-{g} 
&
v'
}}
\end{eqn}
Such a strict branching is also denoted by $(f,g) : u \dfl (u',v')$, and the $(n-1)$-cell $u$ is called the \emph{source} of the strict branching.
A \emph{$S$-branching} is a triple~$(f,e,g)$ where $f,g$ are $n$-cells of~$S^*$ with $f$ non trivial and $e$ is an $n$-cell of $\tck{E}$. Such a branching is depicted by
\begin{eqn}{equation}
\label{E:branchingModulo}
\raisebox{0.55cm}{
\xymatrix @R=2em @C=2em{
u 
  \ar[r] ^-{f} 
  \ar[d] _-{e}
& 
u'
\\
v
  \ar[r] _-{g} 
&
v'
}}
\qquad\qquad
\big(\text{resp.}
\quad
\raisebox{0.55cm}{
\xymatrix @R=2em @C=2em{
u 
  \ar[r] ^-{f} 
  \ar[d] _-{e}
& 
u'
\\
v
}}
\quad
\big)
\end{eqn}
when $g$ is non trivial (resp. trivial) and denoted by 
$(f,e,g) : (u,v) \dfl (u',v')$ (resp. $(f,e) : u \dfl (u',v)$). 
The pair of $(n-1)$-cells $(u,v)$ (resp. $(u,u)$) is called the \emph{source} of this branching. Note that any strict branching $(f,g)$ is a branching $(f,e,g)$ where $e=i_1^v(\st_{-,(n-1)}^h(f)) = i_1^v(\st_{-,(n-1)}^h(g))$.

\subsubsection{Confluences and confluences modulo}
\label{SSS:Confluences}
A \emph{strict $S$-confluence} is a pair $(f',g')$ of $n$-cells  of $S^\ast$ such that $\st_{+,(n-1)}^h(f')=\st_{+,(n-1)}^h(g')$, depicted by
\[
\xymatrix @R=2em @C=2em{
u' 
  \ar[r] ^-{f'} 
& 
w
  \ar[d] ^-{\rotatebox{90}{=}}
\\
v'
  \ar[r] _-{g'} 
&
w
}
\]
and denoted by $(f',g'):(u',v') \dfl w$.
A \emph{$S$-confluence} is a triple $(f',e',g')$, where $f',g'$ are $n$-cells of $S^\ast$ and $e'$ is an $n$-cell of $\tck{E}$ such that $\st_{+,(n-1)}^h(f')=\st_{-,(n-1)}^v(e')$ and $\st_{+,(n-1)}^h(g')=\st_{+,(n-1)}^v(e')$, depicted by
\[
\xymatrix @R=2em @C=2em{
u' 
  \ar[r] ^-{f'} 
& 
w
  \ar[d] ^-{e'}
\\
v'
  \ar[r] _-{g'} 
&
w'
}
\]
and also denoted by $(f',e',g'):(u',v') \dfl (w,w')$.
The strict $S$-branching~\eqref{E:branching} is \emph{strictly confluent}
(resp. \emph{confluent}) if there exists a strict $S$-confluence $(f',g')$ (resp. $S$-confluence $(f',e',g')$) as follows:
\[
\raisebox{0.55cm}{
\xymatrix @R=2em @C=2em{
u
  \ar[r] ^-{f}
  \ar[d] _-{\rotatebox{90}{=}}
&
u' 
  \ar@{.>}[r] ^-{f'} 
& 
w
  \ar[d] ^-{\rotatebox{90}{=}}
\\
u
  \ar[r] _-{g}
&
v'
  \ar@{.>}[r] _-{g'} 
&
w'
}}
\qquad\qquad
\big(\text{resp.}
\quad
\raisebox{0.55cm}{
\xymatrix @R=2em @C=2em{
u
  \ar[r] ^-{f}
  \ar[d] _-{\rotatebox{90}{=}}
&
u' 
  \ar@{.>}[r] ^-{f'} 
& 
w
  \ar@{.>}[d] ^-{e'}
\\
u
  \ar[r] _-{g}
&
v'
  \ar@{.>}[r] _-{g'} 
&
w'
}}
\quad
\big).
\]
The $S$-branching~\eqref{E:branchingModulo} is \emph{confluent} if there exists a $S$-confluence $(f',e',g')$ as follows:
\[
\raisebox{0.55cm}{
\xymatrix @R=2em @C=2em{
u
  \ar[r] ^-{f}
  \ar[d] _-{e}
&
u' 
  \ar@{.>}[r] ^-{f'} 
& 
w
  \ar@{.>}[d] ^-{e'}
\\
v
  \ar[r] _-{g}
&
v'
  \ar@{.>}[r] _-{g'} 
&
w'
}}.
\]
		
\subsubsection{Local branchings}
\label{SSS:LocalBranchings}
A strict $S$-branching $(f,g)$ is \emph{local} if $f,g\in S^{*(1)}$. 
A $S$-branching $(f,e,g)$ is \emph{local} if $f\in S^{*(1)}$, and the $n$-cells $g$ of $S^\ast$ and $e$ of~$\tck{E}$ satisfy $\norm{g}_S + \norm{e}_E = 1$.
Local $S$-branchings belong to one of the following families:
\begin{enumerate}[{\bf i)}]
\item \emph{local aspherical} strict $S$-branchings of the form:
\[
\xymatrix @R=2em @C=2em{
u 
  \ar[r] ^-{f} 
  \ar[d] _-{\rotatebox{90}{=}}
& 
v
  \ar[d] ^-{\rotatebox{90}{=}}
\\
u
  \ar[r] _-{f} 
&
v
}
\]
where $f$ is an $n$-cell of $S^{\ast(1)}$;
\item \emph{local Peiffer} strict $S$-branchings of the form:
\[
\xymatrix @R=2em @C=2em{
u\star_i v 
  \ar[r] ^-{f\star_i v} 
  \ar[d] _-{\rotatebox{90}{=}}
& 
u'\star_i v
\\
u\star_i v
  \ar[r] _-{u\star_i g} 
&
u\star_i v'
}
\]
where $0\leq i \leq n-2$, $f$ and $g$ are $n$-cells of $S^{\ast(1)}$,
\item \emph{local Peiffer} $S$-branchings of the forms:
\begin{eqn}{equation}
\label{E:LocalPeifferModulo}
\xymatrix @R=2em @C=2em{
u\star_i v 
  \ar[r] ^-{f\star_i v} 
  \ar[d] _-{u\star_i e}
& 
u'\star_i v
\\
u\star_i v' &
}
\qquad\qquad
\xymatrix @R=2em @C=2em{
w\star_i u 
  \ar[r] ^-{w\star_i f} 
  \ar[d] _-{e'\star_i u}
& 
w\star_i u'
\\
w'\star_i u &
}
\end{eqn}
where $0\leq i \leq n-2$, where $f$ is an $n$-cell of $S^{\ast(1)}$ and $e,e'$ are $n$-cells of $E^{\top (1)}$;
\item \emph{overlapping} strict $S$-branchings are the remaining local strict branchings:
\[
\xymatrix @R=2em @C=2em{
u 
  \ar[r] ^-{f} 
  \ar[d] _-{\rotatebox{90}{=}}
& 
v
\\
u
  \ar[r] _-{g} 
&
v'
}
\]
where $f$ and $g$ are $n$-cells of $S^{\ast(1)}$,
\item \emph{overlapping $S$-branchings} are the remaining local branchings:
\begin{eqn}{equation}
\label{E:OverlappingBranchingModulo}
\xymatrix @R=2em @C=2em{
u 
  \ar[r] ^-{f} 
  \ar[d] _-{e}
& 
v
\\
v' & }
\end{eqn}
where $f$ is an $n$-cell of $S^{\ast(1)}$ and $e$ is an $n$-cell of $E^{\top (1)}$.
\end{enumerate}

Let $(f,g)$ (resp. $(f,e,g)$) be a strict $S$-branching (resp. $S$-branching) with source $u$ (resp. $(u,v)$) and a whisker $C$ of $R_{n-1}^\ast$ composable with $u$ and $v$. Then, the pair $(C[f],C[g])$ (resp. triple $(C[f],C[e],C[g])$) is a strict $S$-branching (resp. $S$-branching). If the $S$-branching $(f,e,g)$ is local, then so is $(C[f],C[e],C[g])$.
We denote by $\sqsubseteq$ the order relation on $S$-branchings defined by $(f,e,g) \sqsubseteq (f',e',g')$ if there exists a whisker~$C$ of~$R_{n-1}^\ast$ such that $(C[f],C[e],C[g]) = (f',e',g')$. A strict $S$-branching (resp. $S$-branching) is \emph{minimal} if it is minimal for the order relation $\sqsubseteq$. A strict $S$-branching (resp. $S$-branching) is \emph{critical} if it is a minimal overlapping strict $S$-branching (resp. $S$-branching).

\subsubsection{Confluence properties of polygraphs modulo}
The $n$-polygraph modulo~$\Pr$ is called
\begin{enumerate}[{\bf i)}]
\item \emph{locally confluent} if each of its local $S$-branchings is confluent,
\item \emph{confluent} if each of its $S$-branchings is confluent, 
\item \emph{convergent} if it is both terminating and confluent,
\item \emph{diconvergent} when its is convergent and the $n$-polygraph $E$ is convergent,
\item \emph{JK confluent} if every strict $S$-branching is confluent,
\item \emph{JK coherent} if every $S$-branching of the form $(f,e) : u \dfl (u',v)$ is confluent:
\[
\xymatrix @R=2em @C=2em{
u
  \ar[r] ^-{f} 
  \ar[d] _-{e}
& 
v \ar@1@{.>} [r] ^-{f'} & v' \ar@{.>}@1 [d] ^-{e'} \\
u'
  \ar@{.>} [rr] _-{g'} 
&  & w  }
\]
in such a way that $g'$ is a non-trivial $n$-cell in $S^\ast$.
\end{enumerate}

Note that when $\Pr$ is confluent, every $(n-1)$-cell of~$R_{n-1}^\ast$ has at most one $S$-normal form. Under the confluence modulo hypothesis, an~$(n-1)$-cell may admit several $S$-normal forms, which are all equivalent modulo~$E$.
The notions of JK confluence and JK coherence were introduced by Jouannaud and Kirchner in \cite{JouannaudKirchner84}. Following \cite{JouannaudKirchner84}, there exists a local version of JK-confluence  $E$ and JK coherence, given
by properties {\bf a)} and {\bf b)} of Proposition~\ref{P:CoherentCriticalBranchingProposition}, and we will prove in the next section that all these notions are equivalent.

\subsection{Completion procedures for polygraphs modulo}
\label{S:CompletionProcedure}

In this subsection, we define a procedure that completes a non confluent $n$-polygraph modulo $(R,E,\ER)$ into a confluent $n$-polygraph modulo $(\check{R},E,{}_E \check{R})$. 

\subsubsection{Completion of $\ER$}
\label{SSS:RemarquesERconfluentModuloE}
The property of JK coherence is trivially satisfied for the $n$-polygraph modulo $(R,E,\ER$). Indeed, every $\ER$-branching of the form $(f,e)$ is trivially confluent as follows:  
\begin{eqn}{equation}
\label{E:JKCoherenceTrivialForER}
\xymatrix @R=2em @C=2em{
u 
  \ar[r] ^-{f} 
  \ar@{->} [d] _-{e}
& 
v \ar[d] ^-{\rotatebox{90}{=}}
\\
v'
  \ar[r] _-{e^- \cdot f} 
&
v
} 
\end{eqn}
where $e^- \cdot f$ is a $\ER$-rewriting step.
Following Theorem \ref{T:CoherentCriticalBranchingTheorem}, we define a completion procedure to reach confluence of the $n$-polygraph modulo $(R,E,\ER)$, similar to the Knuth-Bendix completion. From~\eqref{E:JKCoherenceTrivialForER} and Theorem \ref{T:CoherentCriticalBranchingTheorem}, when $(R,E,\ER)$ is terminating, it is confluent if, and only if, all its critical branchings $(f,g)$ with $f$ in $(\ER)^{\ast (1)}$ and $g$ in $\Ro$ are confluent, as depicted by:
\[
\xymatrix @R=2em @C=5em {
u
  \ar[r] ^-{f \in (\ER)^{\ast (1)}} 
  \ar[d] _-{\rotatebox{90}{=}}
& 
v \ar@1@{.>} [r] ^-{f' \in (\ER)^{\ast}} & v' \ar@{.>}@1 [d] ^-{e'} \\
u
  \ar[r] _-{g \in \Ro} 
& w \ar@1@{.>} [r] _-{g' \in (\ER)^{\ast}} & w'}
\]
We denote by $\text{CP}(\ER,R)$ the set of such critical branchings.
	
\subsubsection{Completion procedure for $\ER$}
\label{SSS:CompletionProcedure}
Consider an $n$-polygraph modulo $(R,E,\ER)$ with a termination order $\prec$. The following procedure computes a completion $\check{R}$ of the $n$-polygraph $R$ such that the $n$-polygraph modulo $(\check{R},E,\er{E}{\check{R}})$ is confluent. We denote by $\hat{u}^{\ER}$ a $\ER$-normal form of an element $u$ in $R_{n-1}^\ast$. For all $(n-1)$-cells~$u,v$ in $R_{n-1}^\ast$, we denote $u \approx_E v$ if there exists an $n$-cell $e: u \fl v$ in $\tck{E}$.  

\medskip

\begin{algorithm}[H]
		\SetAlgoLined
		\KwIn{
			
			\begin{tabular}{l}
				$R$ and $E$ two $n$-polygraphs such that $R_{\leq n-1} = E_{\leq n-1}$.\\
				$\prec$ a termination order for $(R,E,\ER)$ that is total on the set of $\ER$-irreducible $(n-1)$-cells.
			\end{tabular}
		}
		
		\BlankLine
		
		\Begin{
			
			\BlankLine

			$\Cr \leftarrow \text{CP}(\ER,R)$; \\
			\While{ $\Cr \ne \emptyset$} {
				Pick a branching $c=(f: u \dfl v,g: u \dfl w)$ in $\mathcal{C}$, with $f$ in $\ER^\ast$ and $g$ in $R^\ast$; \\
				Reduce $v$ to a $\ER$-normal form $\hat{v}^{\ER}$; \\
				Reduce $w$ to a $\ER$-normal form $\hat{w}^{\ER}$; \\
				$\Cr \leftarrow \Cr \backslash \{ c \} $ ;\\
				\If{ $\hat{v}^{\ER} \quad \cancel{\approx_E} \quad \hat{w}^{\ER}$ }
				{ \If{ $\hat{w}^{\ER} \prec \hat{v}^{\ER}$}
					{ $R \leftarrow R \cup \{\alpha : \hat{v}^{\ER}\dfl \hat{w}^{\ER} \}$;} 
					\If{ $\hat{v}^{\ER} \prec \hat{w}^{\ER}$}
					{ $R \leftarrow R \cup \{\alpha :  \hat{w}^{\ER} \dfl \hat{v}^{\ER} \}$;}}
				$\Cr \leftarrow \Cr \cup \{ \text{$(\ER,R)$-critical branchings created by $\alpha$} \}$;
			}}
			\BlankLine
\end{algorithm}

\medskip
		
This procedure may not be terminating. However, it does not fail because the order $\prec$ is total on the set of $\ER$-irreducible $(n-1)$-cells. 

\begin{proposition} 
When it terminates, procedure~\eqref{SSS:CompletionProcedure} returns a confluent~$n$-polygraph modulo.
\end{proposition}
\begin{proof}
The proof of soundness of the completion procedure for $(R,E,\ER)$ is a consequence of the inference system given by Bachmair and Dershowitz in \cite{BachmairDershowitz89} in order to get a set of rules $\check{R}$ such that $(\check{R},E,{}_E \check{R})$ is confluent. Given a termination order $\prec$ on $(R,E,\ER)$, their inference system is given by the following six elementary rules:
\begin{enumerate}[{\bf 1)}]
\item Orienting an equation:
$$ ( A \cup \{ s  = t \} , R)  \qquad \rightsquigarrow \qquad \textrm{$(A, R \cup \{ s \fl t \}$) \quad if\; $s \succ t$.} $$ 
\item Adding an equational consequence:
		$$ (A,R)  \qquad \rightsquigarrow \qquad \textrm{$(A \cup \{ s = t \} , R)$ if $s \overset{*}{\longleftarrow_{R \cup E}} u \overset{*}{\longrightarrow_{R \cup E}} t$}. $$ 
\item Simplifying an equation:
		$$ (A \cup \{ s = t \} , R) \qquad \rightsquigarrow \qquad \textrm{$(A \cup \{ u=t \} , R)$ \quad if\; $s \overset{\ER}{\fl} u$}. $$
\item Deleting an equation:
		$$ (A \cup \{ s = t \} , R) \qquad \rightsquigarrow \qquad \textrm{$(A , R)$ \quad if\; $s \approx_E t$}. $$
\item Simplifying the right-hand side of a rule:
		$$ (A, R \cup \{ s \fl t \}) \qquad \rightsquigarrow \qquad \textrm{$(A, R \cup \{ s \fl u \}) $\quad if\;$t \overset{\ER}{\fl} u$}. $$
\item Simplifying the left-hand side of a rule:
		\begin{align*} (A, R \cup \{ s \fl t \}) \qquad & \rightsquigarrow \qquad \textrm{$(A \cup \{ u=t \}, R) $ \quad if\; $s \overset{\ER}{\fl} u$}. \\ 
		\end{align*} 
\end{enumerate}
The soundness of Procedure~(\ref{SSS:CompletionProcedure}) is a consequence of the following arguments:
\begin{enumerate}[{\bf i)}]
			\item For every critical branching $(f: u \fl v,g: u \fl w)$ in $\text{CP}(\ER,R)$, we can add an equation $v = w$ using the elementary rule {\bf 2)}, and simplify it to $\hat{v}^{\ER} = \hat{w}^{\ER}$ using the elementary rule {\bf 3)}.
			\item If $\hat{v}^{\ER} \approx_E \hat{w}^{\ER}$, we can delete the equation using the elementary rule {\bf 4)}. 
			\item Otherwise, we can always orient it using the elementary rule {\bf 1)}.
\end{enumerate}
Thus, each step of the procedure comes from one of these inference rules. Following \cite{BachmairDershowitz89}, it returns a confluent $n$-polygraph modulo $(\check{R},E,\er{E}{\check{R}})$.
\end{proof} 

\subsubsection{Completion procedure for $\ERE$}
\label{SSS:CompletionProcedureERE}
By definition, the polygraph modulo $(R,E,\ER)$ is confluent if, and only if, the polygraph modulo $(R,E,\ERE)$ is confluent. The completion procedure~\eqref{SSS:CompletionProcedure} extends to polygraphs modulo $(R,E,\ERE)$. In that case, the critical branchings of the form $(f,e)$  with $f$ in $\ERE^{\ast (1)}$ and $e$ in $\Eo$ are still trivially confluent. The $\ERE$-critical branchings of the form $(f,g)$, with $f$ in $(\ERE)^{\ast (1)}$ and $g$ in $R^{\ast (1)}$ can be written as a pair $(f' \cdot e, g)$, where $(f',g)$ is a critical branching in $\text{CP}(\ER,R)$ and $e$ is an $n$-cell in $\tck{E}$.  
The completion procedure~\eqref{SSS:CompletionProcedure} for $\ER$ can therefore be adapted to the polygraph modulo $(R,E,\ERE)$. In that case, the procedure differs from~\eqref{SSS:CompletionProcedure} by the fact that when adding a rule $\alpha: u \dfl v$ in~$R$, we can choose as target of $\alpha$ any element of the $E$-equivalence class of $v$. We prove in the same way that if the procedure terminates, it returns an $n$-polygraph modulo $(\check{R},E,{}_E \check{R}_E)$ that is confluent.

\section{Coherent confluence modulo}
\label{S:CoherentConfluenceModulo}

This section deals with the property of coherent confluence for an $n$-polygraph modulo defined by the adjunction of a square cell for each confluence diagram.
We prove a coherent version of Newman's lemma \cite{Newman42} for polygraphs modulo relating coherent confluence to coherent local confluence.  We prove also a coherent version of the critical branching lemma reducing local coherent confluence to the coherent confluence of a reduced set of critical branchings. 
Let $\Pr=(R,E,S)$ denote an $n$-polygraph modulo.

\subsection{Coherent Newman's lemma modulo}
\label{SS:CoherentNewmanModulo}

\subsubsection{Action on branchings}
\label{SSS:ActionOnSquare}
Let $\Gamma$ be a square extension of $\Pr$. Every $n$-cell $f$ in~$S^\ast$ can be written $f = e_1 \star_{n-1} f_1 \star_{n-1} e_2 \star_{n-1} f_2$ in the free $n$-category $(R\cup E)^\ast$, with $f_1$ in $\Ro$, $f_2$ in~$S^\ast$ such that $\norm{f_2}_S = \norm{f}_S -1$, and $e_1,e_2$ are $n$-cells in $\tck{E}$ possibly identities.
Thus, a $S$-branching $(f,e,g)$ with a $S$-confluence $(f',e',g')$ may correspond to different squares in $\Sq (\tck{E} , S^\ast)$. For example, if the $n$-cell $g$ of $S^\ast$ can be decomposed as $e_1 \star_{n-1} g_1 \star_{n-1} e_2$, the following squares in $\Sq (\tck{E} , S^\ast)$ are different $S$-branchings, 
but we would like them to be equivalent because they provide the same relation among relations when quotienting by $E$:
\[
\xymatrix @R=2em @C=2em {
u
   \ar[r] ^-{f}
   \ar[d] _-{e}
&
v \ar@1@{.>} [r] ^-{f'} & v' \ar@1@{.>} [d] ^-{e'} \\
u
   \ar[r] _-{g}
& w \ar@1@{.>} [r] _-{g'} & w'
} \qquad \raisebox{-6mm}{$\text{and}$} \qquad \xymatrix @R=2em @C=2em {
u
   \ar[r] ^-{f}
   \ar[d] _-{e \star_{n-1} e_1}
&
v \ar@1@{.>} [r] ^-{f'} & v' \ar@1@{.>} [d] ^-{e'} \\
u_1
   \ar[r] _-{g_1 e_2}
& w \ar@1@{.>} [r] _-{g'} & w'
}
\]
These two squares are not equal in the free $n$-category enriched in double groupoids generated by the double $(n+1,n-1)$-polygraph $(E,S,\Gamma\cup \peiffer{\tck{E}}{S^\ast})$. However, in the computation of a coherent extension of $\Pr$, we do not want to add a square cell for each of these confluence diagrams, since they would give the same relation among relations in the quotient modulo the axioms.
We then define a \emph{biaction} of $\tck{E}$ on $\Sq (\tck{E} , S^\ast)$. For all $n$-cells~$e_1,e_2$ in $\tck{E}$ and square $(n+1)$-cell 
\[ 
\xymatrix @R=2em @C=2em{
u 
  \ar[r] ^-{f} ^-{}="1"
  \ar[d] _-{e}
& 
u' \ar[d] ^-{e'}
\\
u
  \ar[r] _-{g} _-{}="2" 
&
v'
\ar@2  "1"!<0pt,-10pt>;"2"!<0pt,10pt> ^-{A} 
}
\]
in $\Sq (\tck{E} , S^\ast)$ satisfying the following conditions
\begin{enumerate}[{\bf i)}]
\item $ \st_{+,n-1}(e_1) = \st_{-,n-1}^h \st_{-,n}^v (A)$ and $ \st_{-,n-1}(e_2) = \st_{+,n-1}^h \st_{-,n}^v (A)$,
\item $ e_1 \st_{-,n}^h (A) \in S $ and $ e_2^- \st_{+,n}^h (A) \in S $,
\end{enumerate}
we define the square $(n+1)$-cell ${}_{e_2}^{e_1} A$ as follows:
\[ 
\xymatrix @R=2em @C=2em{
u_1  
  \ar[r] ^-{e_1 f} ^-{}="1"
  \ar[d] _-{e_1 e e_2}
& 
u' \ar[d] ^-{e'}
\\
u_2
  \ar[r] _-{e_2^- g} _-{}="2" 
&
v'
\ar@2  "1"!<-7pt,-10pt>;"2"!<-7pt,10pt> ^-{{}_{e_2}^{e_1} A}
}
\]
where $u_1 = \st_{-,n-1}(e_1)$ and $u_2 = \st_{+,n-1}(e_2)$. For a square extension $\Gamma$ of $\Pr$, we denote by $E \rtimes \Gamma$ the set containing all elements ${}_{e_2}^{e_1} A$, for $A$ in $\Gamma$ and $e_1,e_2$ $n$-cells in $\tck{E}$, whenever it is well defined.
For all $n$-cells $e_1,e_2$ in $\tck{E}$ and square $A,A'$ in $\Gamma$, the following equalities hold whenever both sides are defined:
\begin{enumerate}[{\bf i)}]
\item $ {}^{e'_1}_{e'_2}  ( {}^{e_1}_{e_2} A) = {}^{e'_1 e_1}_{e_2 e'_2} A $,
\item $ {}^{e_1}_{e_2} (A \cdg^v A') = ({}_{e_2}^{e_1} A) \cdg^v A'$,
\item $ {}^{e_1}_{e_2} (A \cdg^h A') = ({}^{e_1}_{1} A) \cdg^h ({}^{1}_{e_2} A')$.
\end{enumerate}

\subsubsection{Coherent confluence modulo}
Let us denote by
\[
\sqconf{\Gamma} \: := \:  \acy{\Gamma}
\]
the free $(n-1)$-category enriched in double categories, whose vertical $n$-cells are invertible, generated by the double $(n+1)$-polygraph $(E,S, \sqext)$ in $\dbpol_{n+1}^v$.
The $S$-branching~\eqref{E:branchingModulo} is \emph{$\Gamma$-confluent} if there exist $n$-cells~$f',g'$ in~$S^*$, $e'$ in~$\tck{E}$ and an~$(n+1)$-cell $A$ in $\sqconf{\Gamma}$ as in the following diagram:
\[
\raisebox{0.55cm}{
\xymatrix @R=2em @C=2em{
u
  \ar[r] ^-{f}
  \ar[d] _-{e}
&
u' 
  \ar@{.>}[r] ^-{f'} 
  \ar@2 []!<0pt,-8pt>; [d]!<0pt,+8pt> ^-{A}
& 
w
  \ar@{.>}[d] ^-{e'}
\\
v
  \ar[r] _-{g}
&
v'
  \ar@{.>}[r] _-{g'} 
&
w'
}}.
\]

We say that the $n$-polygraph modulo $\Pr$ is 
\begin{enumerate}[{\bf i)}]
\item \emph{$\Gamma$-confluent} (resp. \emph{locally $\Gamma$}-confluent) if every $S$-branching (resp. local $S$-branching) is $\Gamma$-confluent, 
\item \emph{$\Gamma$-convergent} if it is terminating and $\Gamma$-confluent,
\item \emph{$\Gamma$-diconvergent} if it is $\Gamma$-convergent and the $n$-polygraph $E$ is convergent.
\end{enumerate}
When $\Gamma=\Sq(\tck{E},S^\ast)$ (resp. $\Gamma=\Sph(S^\ast)$), the property of $\Gamma$-confluence corresponds to the property of confluence (resp. strict confluence) defined in \eqref{S:ConfluenceModulo}.

In the sequel, the proofs of confluence results will use Huet's double Noetherian induction principle on the $n$-polygraph $\Saux$ defined in \eqref{SSS:InductionPrincipale} and the property $\mathbb{P}$ on $R_{n-1}^\ast \times R_{n-1}^\ast$ defined by
\begin{eqn}{equation}
\label{E:PropertyDoubleInduction}
\mathbb{P}(u,v): \text{\emph{every $S$-branching with source $(u,v)$ is $\Gamma$-confluent},} 
\end{eqn}
for all~$u,v$ in $R_{n-1}^\ast$.

\begin{proposition}[Coherent half Newman's modulo lemma]
\label{P:CoherentHalfNewmanModulo}
Let $\Pr=(R,E,S)$ be a terminating $n$-polygraph modulo, and $\Gamma$ be a square extension of $\Pr$. If $\Pr$ is locally $\Gamma$-confluent, then the following two conditions hold
\begin{enumerate}[{\bf i)}]
\item every $S$-branching of the form $(f,e)$, with $f$ in $\So$ and $e$ in $\tck{E}$, is $\Gamma$-confluent,
\item every $S$-branching of the form $(f,e)$, with $f$ in $S^{\ast}$ and $e$ in $\Eo$, is $\Gamma$-confluent.
\end{enumerate}
\end{proposition}

\begin{proof}
We prove condition {\bf i)}, the proof of condition {\bf ii)} is similar. Suppose that $\Pr$ is locally $\Gamma$-confluent, we proceed by double induction.

We denote by $u$ the source of the branching $(f,e)$. If $u$ is $S$-irreducible, then $f$ is an identity $n$-cell, and the branching is trivially $\Gamma$-confluent. 
Suppose that $f$ is not an identity and assume that for every pair $(u',v')$ of $(n-1)$-cells in $R_{n-1}^\ast$ such that there is an $n$-cell $(u,u) \fl (u',v')$ in $\Saux$, every $S$-branching $(f',e',g')$ of source $(u',v')$ is $\Gamma$-confluent.

Prove that the branching $(f,e)$ is $\Gamma$-confluent.
We proceed by induction on $\norm{e}_E \geq 1$. If $\norm{e}_E = 1$, $(f,e)$ is a local $S$-branching and it is $\Gamma$-confluent by hypothesis. 
Now, let us assume that for $k \geq 1$, every $S$-branching $(f'',e'')$ such that $\norm{e''}_E = k$ is $\Gamma$-confluent, and let us consider a $S$-branching of the form $(f,e)$ with source $u$, such that $\norm{e}_E = k+1$. Let us write $e = e_1 \star_{n-1} e_2$ with $e_1$ in $\Eo$ and $e_2$ in $\tck{E}$. Using local $\Gamma$-confluence on the $S$-branching $(f,e_1)$ of source $u$, there exist $n$-cells $f'$ and $f_1$ in $S^\ast$, an $n$-cell $e'_1: t_{n-1}(f') \fl t_{n-1}(f_1)$ in $\tck{E}$ and an~$(n+1)$-cell $A$ in $\sqconf{\Gamma}$ such that $\st^h_{-,n}(A)=f \star_{n-1} f'$ and $\st^h_{+,n}(A)=f_1$. Then, write $f_1 = f_1^1 \star_{n-1} f_1^2$ with $f_1^1$ in $\So$ and $f_1^2$ in $S^\ast$. Using the induction hypothesis on the $S$-branching $(f_1^1,e_2)$ with source $u_1 := t_{n-1}(e_1) = s_{n-1}(e_2)$, there exist $n$-cells $f'_1$ and $g$ in $S^\ast$, an $n$-cell $e_2: t_{n-1}(f'_1) \fl t_{n-1}(g)$ in $\tck{E}$ and an~$(n+1)$-cell $B$ in $\sqconf{\Gamma}$ such that $\st^h_{-,n}(B)=f_1^1 \star_{n-1} f'_1$ and $\st^h_{+,n}(B)=g$.
This can be represented by the following diagram:
\[
\xymatrix@R=2.25em@C=5.5em{ u \ar [d] _-{e_1} \ar [r] ^-{f} & u' \ar [r] ^-{f'} & u'' \ar [d] ^-{e'_1} \\
u_1 \ar [d] _-{\rotatebox{90}{=}} \ar [r] |-{f_1^1} ^-{}="2" & u'_1 \ar [d] ^-{\rotatebox{90}{=}} \ar [r] |-{f_1^2} & u''_1  \\
u_1 \ar [r] |-{f_1^1} _-{}="3" \ar [d] _ -{e_2} & u'_1 \ar [r] |-{f'_1} & u'_2 \ar [d] ^-{e'_2} \\
v \ar [rr] _-{g} ^-{}="1" & & v' 
\ar@2{} "1,2";"2,2" |{\text{\emph{Local $\Gamma$-confluence}}}
\ar@2{} "3,2";"1" |{\text{\emph{Induction on $\norm{e}_E$}}}
\ar@2{} "2";"3" |{i_1^h(f_1^1)}   } 
\]
Now, there is an $n$-cell $(u,u) \fl (u'_1,u'_1)$ in $\Saux$ given by the composite 
\[ 
(u,u) \fl (u_1,u_1) \fl (u_1,u'_1) \fl (u'_1,u'_1) 
\] 
where the first step exists because $\norm{e_1}_E > 0$ and the remaining composite is as in \eqref{SSS:InductionPrincipale}. Then, we apply double induction on the $S$-branching $(f_1^2, f_1')$ of source $(u'_1,u'_1)$: there exist $n$-cells $f_2$ and $f'_2$ in $S^\ast$ and an $n$-cell $e_3: t_{n-1}(f_2) \fl t_{n-1}(f'_2)$ in $\tck{E}$. By a similar argument, we can apply double induction on the $S$-branchings $(f_2,(e'_1)^-)$ and $(f'_2, e'_2)$, so that there exist $n$-cells $f''$,$f_3$, $f'_3$ and $g'$ in $S^\ast$ and $n$-cells $e''_1: t_{n-1}(f'') \fl t_{n-1}(f_3)$ and $e''_2: t_{n-1}(f'_3) \fl t_{n-1}(g')$  as in the following diagram:
\[
\xymatrix@R=2.25em@C=5.5em{ u \ar [d] _-{e_1} \ar [r] ^-{f} & u' \ar [r] ^-{f'} & u'' \ar [d] ^-{e'_1} \ar [rr]^-{f''} ^-{}="4" & & u''' \ar [d] ^-{e''_1} \\
u_1 \ar [d] _-{\rotatebox{90}{=}} \ar [r] |-{f_1^1} ^-{}="2" & u'_1 \ar [d] ^-{\rotatebox{90}{=}} \ar [r] |-{f_1^2} & u''_1 \ar [r] |-{f_2}  & w_1 \ar [r] |-{f_3} \ar [d] ^-{e_3} & w'_1  \\
u_1 \ar [r] |-{f_1^1} _ -{}="3" \ar [d] _ -{e_2} & u'_1 \ar [r] |-{f'_1} & u'_2 \ar [d] ^-{e'_2} \ar [r] |-{f'_2}  & w_2 \ar [r] |-{f'_3} & w'_2  \ar [d] ^-{e''_2} \\
v \ar [rr] _-{g} ^-{}="1" & & v' \ar [rr] _ -{g'} _-{}="5" & & v'' 
\ar@2{} "1,2";"2,2" |{\text{\emph{Local $\Gamma$-confluence}}}
\ar@2{} "3,2";"1" |{\text{\emph{Induction on $\norm{e}_E$}}} 
\ar@2{} "2";"3" |{i_1^h(f_1^1)}  
\ar@2{} "2,3";"3,3" |{\text{\emph{Double induction}}}
\ar@2{}  "4";"2,4" |{\text{\emph{Double induction}}}
\ar@2{}   "3,4";"5" |{\text{\emph{Double induction}}} } \]   
We can then repeat the same process using double induction on the $S$-branching $(f_3,e_3,f'_3)$ with source $(w_1,w_2)$, and so on. This process terminates in finitely many steps, otherwise it leads to an infinite $S$-rewriting path with source $u_1$, which is not possible since $\Pr$ is terminating. This proves the $\Gamma$-confluence of the branching $(f,e)$.
\end{proof}

\begin{theorem}[Coherent Newman's lemma modulo]
\label{T:ConfluenceTheorem}
Let $\Pr$ be a terminating $n$-polygraph modulo, and $\Gamma$ be a square extension of $\Pr$. If $\Pr$ is locally $\Gamma$-confluent, then it is $\Gamma$-confluent.
\end{theorem}
\begin{proof}
We set $\Pr=(R,E,S)$ and we prove that every $S$-branching $(f,e,g)$ is $\Gamma$-confluent. Let us choose such a $S$-branching and denote by $(u,v)$ its source. We assume that every $S$-branching modulo $(f',e',g')$ with source $(u',v')$ such that there is an $n$-cell $(u,v) \fl (u',v')$ in $\Saux$ is $\Gamma$-confluent. We follow the proof scheme used by Huet in \cite[Lemma 2.7]{Huet80}. Denote by $n := \norm{f}_S$ and $m:= \norm{g}_S$. We assume without loss of generality that $n > 0$ and we set $f = f_1 \star_{n-1} f_2$, with $f_1$ in $\So$ and $f_2$ in $S^\ast$. 

If $m = 0$, by Proposition \ref{P:CoherentHalfNewmanModulo} on the $S$-branching $(f_1,e)$, there exist $n$-cells $f'_1,g'$ in $S^\ast$, an $n$-cell $e': t_{n-1}(f'_1) \fl t_{n-1}(g')$ and an~$(n+1)$-cell $A$ in $\sqconf{\Gamma}$ such that $\st^h_{-,n}(A)=f_1 \star_{n-1} f'_1$ and $\st^h_{+,n}(A)=g'$. Then, since there is an $n$-cell $(u,u) \fl (u_1,u_1)$ in $\Saux$ with $u_1 := t_{n-1}(f_1)$, we can apply double induction on the $S$-branching $(f_2,f'_1)$ as in the following diagram:
\[
\xymatrix@R=2.25em @C=5em{ 
u \ar [r] ^-{f_1} ^-{}="2" \ar [d] _-{\rotatebox{90}{=}} & u_1 \ar [r] ^-{f_2} \ar [d] ^-{\rotatebox{90}{=}} & u_2 \ar [r]^-{f'_2} & u'_2 \ar [d] ^-{}\\
u \ar [d] _ -{e} \ar [r] |-{f_1} _ -{}="3" & u_1 \ar [r] |-{f'_1} & u_2 \ar [r] |-{f''_1} \ar [d] ^-{e'} & u'_2 \\
v \ar [rr] _ -{g'} ^-{}="1" & & v' & 
\ar@2{} "2,2";"1" |{\text{\emph{Proposition \ref{P:CoherentHalfNewmanModulo}}}} \ar@2{} "2";"3" |{i_1^h(f_1)}
\ar@2{} "1,3";"2,3" |{\text{\emph{Double induction}}} } \]
We finish the proof of this case with a similar argument than in \eqref{P:CoherentHalfNewmanModulo}, using repeated double inductions that can not occur infinitely many times since $\Pr$ is terminating.

Now, assume that $m > 0$ and we set $g = g_1 \star_{n-1} g_2$, with $g_1$ in $\So$ and $g_2$ in $S^\ast$. By Step $1$ on the $S$-branching $(f_1,e)$, there exist $n$-cells $f'_1,h_1$ in $S^\ast$, an $n$-cell $e_1 : t_{n-1}(f'_1) \fl t_{n-1}(h_1)$ in $\tck{E}$, and an~$(n+1)$-cell $A$ in $\sqconf{\Gamma}$ such that $\st^h_{-,n}(A)=f_1 \star_{n-1} f'_1$ and $\st^h_{+,n}(A)=h_1$. We distinguish two cases whether $h_1$ is an identity or not. 

\medskip

If $h_1$ is an identity $n$-cell, the $\Gamma$-confluence of the $S$-branching $(f,e,g)$ is given by the following diagram
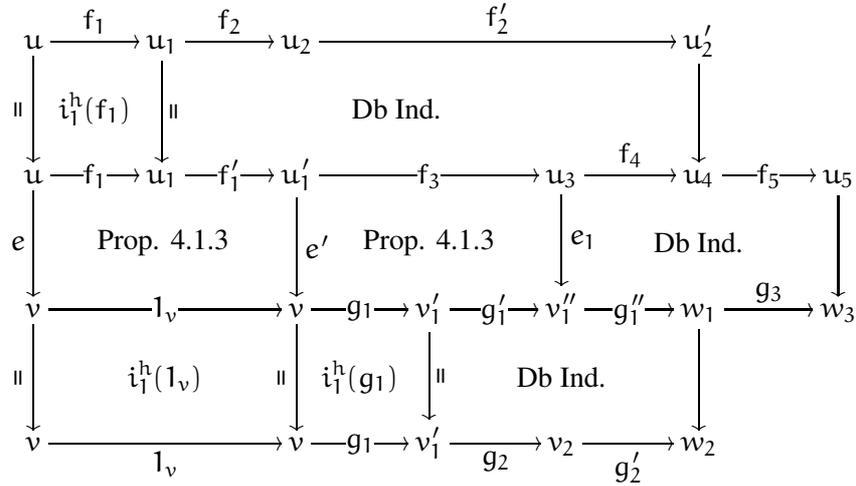
\[
\xymatrix@C=5em@R=2.25em {
u \ar [d] _ -{\rotatebox{90}{=}} \ar [r] ^-{f_1} ^-{}="11" & u_1 \ar [d] ^-{\rotatebox{90}{=}} \ar [r] ^-{f_2} & u_2 \ar [rrr] ^-{f'_2} & & & u'_2 \ar [d] ^-{}    & \\
u \ar [r] |-{f_1} _-{}="12" \ar [d] _ -{e} & u_1 \ar [r] |-{f'_1} & u'_1 \ar [rr] |-{f_3} ^-{}="1" \ar [d] ^-{e'}  &  & u_3 \ar [d] ^-{e_1} \ar [r] ^-{f_4} ^-{}="8" & u_4 \ar [r] |-{f_5} & u_5 \ar [d] ^-{} \\
v \ar [d] _ -{\rotatebox{90}{=}} \ar [rr] |-{1_v} ^-{}="2" & & v \ar [d] _-{\rotatebox{90}{=}} \ar [r] |-{g_1} ^-{}="5" & v'_1 \ar [d] ^-{\rotatebox{90}{=}} \ar [r] |-{g'_1} & v''_1 \ar [r] |-{g''_1} & w_1 \ar [d] ^-{} \ar [r] ^{g_3} & w_3 \\
v \ar [rr] _-{1_v} _-{}="3" & & v \ar[r] |-{g_1} _-{}="6" & v'_1 \ar [r] _-{g_2} & v_2 \ar [r] _ -{g'_2} & w_2 &  
\ar@2{} "3,4";"2,4" |{\emph{\text{Proposition \ref{P:CoherentHalfNewmanModulo}}}}
\ar@2{} "2";"2,2" |{\emph{\text{Proposition \ref{P:CoherentHalfNewmanModulo}}}}
\ar@2{} "11";"12" |{i_1^h(f_1)} 
\ar@2{} "2";"3" |{i_1^h(1_v)} 
\ar@2{} "5";"6" |{i_1^h(g_1)} 
\ar@2{} "1,3";"2,3"!<75pt,0pt> |{\emph{\text{Double induction}}}
\ar@2{} "2,6";"3,6" |{\emph{\text{Double induction}}}
\ar@2{} "3,5";"4,5" |{\emph{\text{Double induction}}}
} \]
where the $S$-branchings $(f_1,e)$ and $(g_1,e')$ are $\Gamma$-confluent by Proposition \ref{P:CoherentHalfNewmanModulo}, double induction applies on the branchings $(f_2,f'_1 \star_{n-1} f_3)$, $(g'_1,g_2)$ and $(f_4,e_1,g''_1)$ since there are $n$-cells 
\[ (u,v) \fl (u,u) \fl (u_1,u_1) \; , \;  (u,v) \fl (v,v) \fl (v,v'_1) \fl (v'_1,v'_1) \; \text{and} \; (u,v) \fl (u_3, v) \fl (u_3,v''_1) \] 
in $\Saux$ and we check that this process of double induction can be repeated, and terminates in a finite number of steps since $\Pr$ is terminating. This gives a $\Gamma$-confluence of the $S$-branching $(f,e,g)$.

If $h_1$ is not an identity $n$-cell, let us write $h_1 = h_1^1 \star_{n-1} h_1^2$ with $h_1^1$ in $\So$ and $h_1^2$ in $S^\ast$. The $\Gamma$-confluence of the $S$-branching $(f,e,g)$ is given by the following diagram:

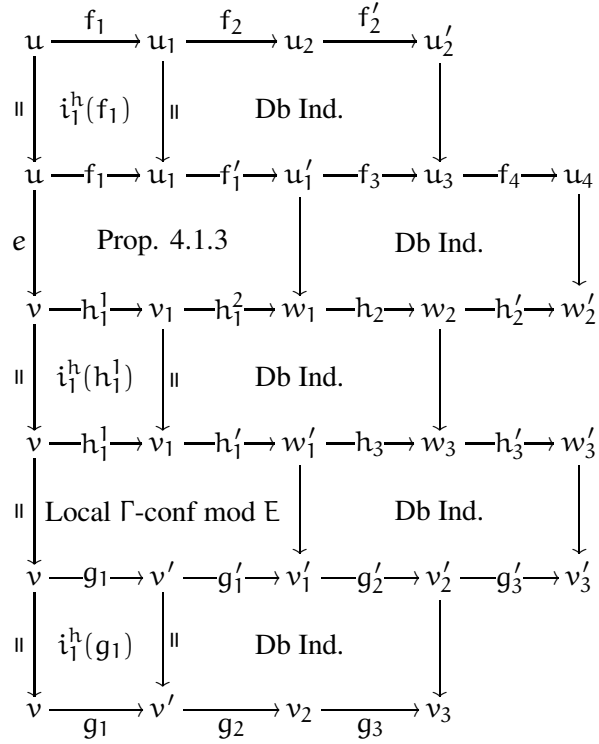
\[ \xymatrix@R=2.25em@C=5em {
u \ar [d] _-{\rotatebox{90}{=}} \ar [r] ^-{f_1} ^-{}="1" & u_1 \ar [d] ^-{\rotatebox{90}{=}}  \ar [r] ^-{f_2} & u_2 \ar [r] ^-{f'_2} & u'_2 \ar [d] ^-{} & \\
u \ar [d] _ -{e} \ar [r] |-{f_1} _-{}="2" & u_1 \ar [r] |-{f'_1} & u'_1 \ar [d] ^-{} \ar [r]|-{f_3} & u_3 \ar [r] |-{f_4} & u_4 \ar [d] ^-{} \\
v \ar [d] _-{\rotatebox{90}{=}} \ar [r] |-{h_1^1} ^-{}="5" & v_1 \ar [d] ^-{\rotatebox{90}{=}} \ar [r] |-{h_1^2} & w_1 \ar [r] |-{h_2} & w_2 \ar [d] ^-{} \ar [r] |-{h'_2} & w'_2 \\
v \ar [d] _-{\rotatebox{90}{=}} \ar [r] |-{h_1^1} _-{}="6" &  v_1 \ar [r] |-{h'_1} & w'_1 \ar [d] ^-{} \ar [r] |-{h_3} & w_3 \ar [r]  |-{h'_3} & w'_3 \ar [d] ^-{} \\
v  \ar [d] _-{\rotatebox{90}{=}} \ar [r] |-{g_1} ^-{}="3" & v' \ar [d] ^-{\rotatebox{90}{=}} \ar [r] |-{g'_1} & v'_1 \ar [r] |-{g'_2} & v'_2 \ar [r] |-{g'_3} \ar [d] ^-{} & v'_3 \\
v \ar [r] _-{g_1} _-{}="4" & v' \ar [r] _-{g_2} & v_2 \ar [r] _-{g_3} & v_3 & 
\ar@2{} "1";"2" |{i_1^h(f_1)} 
\ar@2{} "3";"4" |{i_1^h(g_1)}  
\ar@2{} "5";"6" |{i_1^h(h_1^1)}
\ar@2{} "2,2";"3,2" |{\text{\emph{Proposition \ref{P:CoherentHalfNewmanModulo}}}} 
\ar@2{} "4,2";"5,2" |{\text{\emph{Local $\Gamma$-confluence}}}
\ar@2{} "1,3";"2,3" |{\text{\emph{Double induction}}}
\ar@2{} "2,4";"3,4" |{\text{\emph{Double induction}}}
\ar@2{} "3,3";"4,3" |{\text{\emph{Double induction}}}
\ar@2{} "5,4";"4,4" |{\text{\emph{Double induction}}}
\ar@2{} "5,3";"6,3" |{\text{\emph{Double induction}}}
 } \]
where the $S$-branching $(f_1,e)$ is $\Gamma$-confluent by Proposition \ref{P:CoherentHalfNewmanModulo}, the $S$-branching $(h_1^1,g_1)$ is $\Gamma$-confluent by assumption of local $\Gamma$-confluence of $\Pr$, and we check that double induction applies on the $S$-branchings $(f_2,f'_1)$, $(h_1^2,h'_1)$, $(g'_1,g_2)$, $(f_3,h_2)$ and $(h_3,g'_2)$. This process of double induction can be repeated, and gives a $\Gamma$-confluence of the $S$-branching $(f,e,g)$ in a finite number of steps, since $\Pr$ is terminating.
\end{proof}

\subsection{Coherent critical branching lemma modulo}
In this subsection, we prove coherent local confluence of an $n$-polygraph modulo from coherent confluence of some critical branchings.

\begin{proposition}
\label{P:CoherentCriticalBranchingProposition}
Let $\Pr=(R,E,S)$ be a terminating $n$-polygraph modulo, and $\Gamma$ be a square extension of $\Pr$. Then $\Pr$ is $\Gamma$-locally confluent if, and only if, the following two conditions hold:
\begin{enumerate}[\;\;{\bf a)}]
\item every local strict $S$-branching $(f,g)$ with $f$ in
$S^{\ast(1)}$ and $g$ in $R^{\ast(1)}$ is $\Gamma$-confluent,
\item every local $S$-branching $(f,e)$ with $f$ in
$\So$ and $e$ in $\Eo$ is $\Gamma$-confluent.
\end{enumerate}
\end{proposition}
\begin{proof}
We proceed by double induction.
The only part is trivial because properties {\bf a)} and {\bf b)} correspond to $\Gamma$-confluence of some local $S$-branchings. Conversely, assume that $\Pr$ satisfies properties {\bf a)} and {\bf b)} and prove that every local $S$-branching is $\Gamma$-confluent. We consider a local $S$-branching $(f,e,g)$, and assume without loss of generality that $f$ is a non-trivial $n$-cell in $\So$. There are two cases: either $g$ is trivial, and the local $S$-branching $(f,e)$ is $\Gamma$-confluent by {\bf b)}, or $e$ is trivial. In that case, if $g$ is in $\Ro$, then $\Gamma$-confluence of the branching $(f,g)$ is given by {\bf a)}. Otherwise, let us write $g = e_1 \star_{n-1} g' \star_{n-1} e_2$ with $e_1$,$e_2$ in $\tck{E}$ and $g'$ in $\Ro$. 
Now, let us prove the confluence of the $S$-branching 
\[ 
\xymatrix @R=2em @C=2em{
u \ar [r]^-{f} \ar [d] _-{e_1} & v \\
u' \ar [r] _-{g' e_2} & v'
}
\]
where $g' \star_{n-1} e_2$ is an $n$-cell in $\So$. We will then prove the $\Gamma$-confluence of the branching $(f,g)$ using the biaction of $\tck{E}$ on $\Sq (\tck{E}, S^\ast)$.
Using Proposition \ref{P:CoherentHalfNewmanModulo} on the $S$-branching $(f,e_1)$, there exist $n$-cells $f',f_1$ in $S^\ast$, an $n$-cell $e' : t_{n-1}(f') \fl t_{n-1}(f_1)$ and an~$(n+1)$-cell $A$ in $\sqconf{\Gamma}$ such that $\st^h_{-,n}(A)=f \star_{n-1} f'$ and $\st^h_{+,n}(A)=f_1$. Using property {\bf a)} on the local $S$-branching $(g',g' \star_{n-1} e_2)$ with $g'$ in $\Ro$ and $g' \star_{n-1} e_2$ in $\So$ and the trivial confluence given by the right vertical cell $e_2$, there exists an $(n+1)$-cell $B$ in $\sqconf{\Gamma}$ such that $\st^h_{-,n}(B)=g'$ and $\st^h_{+,n}(B)=g' e_2$. Let us choose a decomposition $f_1 = f_1^1 \star_{n-1} f_1^2$, with $f_1^1$ in $\So$ and $f_1^2$. By property {\bf a)} on the local $S$-branching $(f_1^1,g')$, there exist $n$-cells $f'_1$ and $g'_1$ in $S^\ast$, an $n$-cell $e'' : t_{n-1}(f'_1) \fl t_{n-1}(g'_1)$ and an~$(n+1)$-cell $C$ in $\sqconf{\Gamma}$ such that $\st^h_{-,n}(C)=f_1^1 \star_{n-1} f'_1$ and $\st^h_{+,n}(C)=g' \star_{n-1} g'_1$ as depicted on the following diagram:
\[
\xymatrix@R=2em@C=3em{ u \ar [d] _-{e_1} \ar [r] ^-{f} & u' \ar [r] ^-{f'} & u'' \ar [d] ^-{e'_1} \\
u_1 \ar [d] _-{\rotatebox{90}{=}} \ar [r] |-{f_1^1} ^-{}="2" & u'_1 \ar [d] ^-{\rotatebox{90}{=}} \ar [r] |-{f_1^2} & u''_1  \\
u_1 \ar [r] |-{f_1^1} _-{}="3" \ar [d] _ -{\rotatebox{90}{=}} & u'_1 \ar [r] |-{f'_1} & u'_2 \ar [d] ^-{e'_2} \\
v \ar [r] |-{g'} ^-{}="1" \ar [d]_ -{\rotatebox{90}{=}}  & v_1 \ar [r] _ -{g'_2}  \ar [d] ^-{e_2} & v_2 \\
v \ar [r] _-{g' e_2} ^-{}="9" & v' &
\ar@2 "1,2"!<0pt,-11pt>;"2,2"!<0pt,11pt> ^-{A}
\ar@2 "3,2"!<0pt,-11pt>;"4,2"!<0pt,11pt> ^-{C}
\ar@2 "1"!<0pt,-11pt>;"9"!<0pt,+11pt> ^-{B}
\ar@2{} "2";"3" |{i_1^h(f_1^1)}   } \]

There are $n$-cells $(u,u) \fl (u'_1,u'_1)$ and $(u,u) \fl (v_1,v_1)$ in $\Saux$ given by the following composites
\[ (u,u) \fl (u_1,u_1) \fl (u_1,u'_1) \fl (u'_1,u'_1) \]
\[ (u,u) \fl (u_1,u_1) \fl (u_1,v) \fl (v,v) \fl (v,v_1) \fl (v_1,v_1) \]
so that we can apply double induction on the $S$-branchings $(f_1^2,f'_1)$ and $(g'_2,e_2)$, and we finish the proof of $\Gamma$-confluence of the $S$-branching $(f,e_1,g' e_2)$ using repeated double inductions, terminating in a finite number of steps since $\Pr$ is terminating.

Now, we get the $\Gamma$-confluence of the $S$-branching $(f,g)$ by the following diagram:
\[
\xymatrix@R=2em@C=3em{ u \ar [d] _-{\rotatebox{90}{=}} \ar [r] ^-{f} & u' \ar [r] ^-{f'} & u'' \ar [d] ^-{e'_1} \\
u_1 \ar [d] _-{\rotatebox{90}{=}} \ar [r] |-{e_1 f_1^1} ^-{}="2" & u'_1 \ar [d] ^-{\rotatebox{90}{=}} \ar [r] |-{f_1^2} & u''_1  \\
u_1 \ar [r] |-{e_1 f_1^1} _-{}="3" \ar [d] _ -{\rotatebox{90}{=}} & u'_1 \ar [r] |-{f'_1} & u'_2 \ar [d] ^-{e'_2} \\
v \ar [r] |-{e_1 g'} ^-{}="1" \ar [d]_ -{\rotatebox{90}{=}}  & v_1 \ar [r] _ -{g'_2}  \ar [d] ^-{e_2} & v_2 \\
v \ar [r] _-{e_1 g' e_2} ^-{}="9" & v' &
\ar@2{} "2";"3" |{i_1^h(e_1f_1^1)}  
\ar@2 "1,2"!<0pt,-11pt>;"2,2"!<0pt,11pt> _-{{}^1_{e_1}A}
\ar@2 "3,2"!<0pt,-11pt>;"4,2"!<0pt,11pt> _-{{}^{e_1}_{e_1^-} C}
\ar@2 "1"!<+6pt,-11pt>;"9"!<+6pt,11pt> _{{}^{e_1}_{e_1^-} B} } 
\]
since the top rectangle is by definition tiled by the $(n+1)$-cell ${}^1_{e_1} A$, the bottom rectangle is tiled by the $(n+1)$-cell ${}^{e_1}_{e_1^-} B$ and the remaining rectangle is tiled by the $(n+1)$-cell ${}^{e_1}_{e_1^-} C$. The rest of the confluence diagram is tiled in the same way as above. 
\end{proof}

\subsubsection{Coherent critically confluence}
Given $\Gamma$ a square extension of $\Pr$, we say that $\Pr$ is \emph{$\Gamma$-critically confluent} if it satisfies the following two conditions:
\begin{description}
\item[\;\;$\mathbf{a_0)}$] every strict $S$-critical branching $(f,g)$ with $f$ in
$S^{\ast(1)}$ and $g$ in $R^{\ast(1)}$ is $\Gamma$-confluent,
\item[\;\;$\mathbf{b_0)}$] every $S$-critical branching $(f,e)$ with $f$ in
$\So$ and $e$ in $\Eo$ is $\Gamma$-confluent.
\end{description}

\begin{theorem}[Coherent critical branching lemma modulo]
\label{T:CoherentCriticalBranchingTheorem}
Let $\Pr$ be a terminating $n$-polygraph modulo, and $\Gamma$ be a square extension of $\Pr$. Then $\Pr$ is $\Gamma$-locally confluent if, and only if, it is $\Gamma$-critically confluent.
\end{theorem}
\begin{proof}
By Proposition~\ref{P:CoherentCriticalBranchingProposition}, the local $\Gamma$-confluence is equivalent to both conditions {\bf a)} and {\bf b)}. Let us prove that the condition {\bf a)} (resp. {\bf b)}) holds if, and only if, the condition $\mathbf{a_0)}$ (resp. $\mathbf{b_0)}$) holds. One implication is trivial. 
Suppose that condition $\mathbf{b_0)}$ holds and prove condition {\bf b)}. The proof of the other implication is similar.
We examine all the possible forms of local $S$-branchings given in \eqref{SSS:LocalBranchings}.  Local aspherical $S$-branchings and local Peiffer $S$-branchings of the form~\eqref{E:LocalPeifferModulo} are trivially confluent:
\[
\xymatrix @R=2em @C=2.4em {
u\star_i v 
  \ar[r] ^-{f\star_i v} 
  \ar[d] _-{u\star_i e}
& 
u'\star_i v
  \ar[d] ^-{u'\star_i e}
\\
u\star_i v' 
  \ar[r] _-{f\star_i v'} 
&
u'\star_i v'
}
\qquad\qquad
\xymatrix @R=2em @C=2.4em {
w\star_i u 
  \ar[r] ^-{w\star_i f} 
  \ar[d] _-{e'\star_i u}
& 
w\star_i u'
  \ar[d] ^-{e'\star_i u'}
\\
w'\star_i u 
   \ar[r] _-{w'\star_i f} 
&
w'\star_i u'
}
\]
and $\Gamma$-confluent by definition of $\Gamma$-confluence. The other local $S$-branchings are overlapping $S$-branchings of the form $(f,e) : u \dfl (u',v)$ as in~\eqref{E:OverlappingBranchingModulo}, where $f$ is an $n$-cell of $S^{\ast(1)}$ and $e$ is an $n$-cell of $E^{\top (1)}$. By definition, there exists a whisker $C$ on $R_{n-1}^\ast$ and a critical $S$-branching $(f',e') : u_0 \dfl (u'_0,v_0)$ such that $f=C[f']$ and $e=C[e']$. Following condition $\mathbf{b_0)}$ the branching $(f',e')$ is $\Gamma$-confluent, that is there exists a $\Gamma$-confluence:
\[
\xymatrix @R=2em @C=2em{
u
   \ar[r] ^-{f'}
   \ar[d] _-{e'}
&
v \ar@1@{.>} [r] ^-{f''} & v' \ar@1@{.>} [d] ^-{e''} \\
u'
   \ar@{.>} [rr] _-{g'} _-{}="1"
&  & w
\ar@2@{.>} "1,2"!<0pt,-10pt>;"1"!<0pt,10pt> ^-{A} 
}
\]
inducing a $\Gamma$-confluence for $(f,e)$:
\[
\xymatrix @R=2em @C=2.4em {
C[u]
   \ar[r] ^-{C[f']}
   \ar[d] _-{C[e']}
&
C[v] \ar@1@{.>} [r] ^-{C[f'']} & v' \ar@1@{.>} [d] ^-{C[e'']} \\
C[u']
   \ar@{.>} [rr] _-{C[g']} _-{}="1"
&  & w
\ar@2@{.>} "1,2"!<0pt,-10pt>;"1"!<0pt,10pt> ^-{C[A]} 
}
\]
This proves the condition {\bf b)}. 
\end{proof}
		
\section{Coherent completion modulo}
\label{S:CoherentCompletionModulo}

We construct a double coherent presentation of an $(n-1)$-category, starting with one of its presentations by an $n$-polygraph modulo, and by adding square cells given by the confluence diagrams of some of its critical branchings modulo. 
Let $\Pr=(R,E,S)$ denote a $n$-polygraph modulo.  

\subsection{Coherent completion modulo}
\label{SS:CoherentCompletionModulo}

Let us recall the notion of coherent completion of a convergent $n$-polygraph. We then define the notion of Squier's extension for polygraphs modulo.

\subsubsection{Coherent completion}
\label{SSS:CoherentCompletion}
Recall from \cite{GuiraudMalbos09} that a convergent $n$-polygraph $E$ can be extended into a coherent globular presentation of the $(n-1)$-category $\cl{E}$. Explicitly, if an $n$-polygraph $E$ is critically confluent, we define a \emph{family of generating confluences} of~$E$ as a cellular extension of the free $(n,n-1)$-category~$\tck{E}$ that contains one globular $(n+1)$-cell
\[
\xymatrix @R=0.8em@C=2.8em @!C{
& {v}
	\ar @/^/ [dr] ^{e_1}
	\ar@2 []!<0pt,-15pt>;[dd]!<0pt,15pt> ^-{E_{e,e'}}
\\
{u}
	\ar @/^/ [ur] ^{e}
	\ar @/_/ [dr] _{e'}
&& {w}
\\
& {v'}
	\ar @/_/ [ur] _{e_1'}
}
\]
for every critical branching $(e,e')$ of $E$, where $(e_1,e_1')$ is a chosen confluence.  Any $(n+1,n)$-polygraph obtained from~$E$ by adjunction of a family of generating confluences of $E$ is a globular coherent presentation of the $(n-1)$-category $\cl{E}$, \cite{GuiraudMalbos09}. This result was originally proved by Squier in \cite{Squier94} for~$n=2$.
We will consider a double $(n+1,n-1)$-polygraph $(E,\emptyset,\squier{E})$, where~$\squier{E}$ is a square extension of the $(n,n-1)$-categories $(\tck{E},1)$ seen as an $n$-category enriched in double groupoids that contains exactly one square $(n+1)$-cell
\[ 
\xymatrix@R=2em @C=2em{
u 
  \ar [d] _-{e} \ar [r] ^-{=} ^-{}="1"
& u 
  \ar [d] ^-{e'} \\
v 
  \ar@{.>} [d] _-{e_1} 
& v' 
  \ar@{.>} [d] ^-{e'_1} \\
w \ar@{.>} [r] _-{=} _-{}="2" & w 
\ar@2 "1"!<0pt,-15pt>;"2"!<0pt,15pt> |-{E_{e,e'}} }
\]
for every critical branching $(e,e')$ of $E$, and where $(e_1,e_1')$ is a chosen confluence.

\subsubsection{Squier's extensions modulo}
\label{SSS:SquierExtensionModulo}
A \emph{family of generating confluences} of $\Pr$ is a square extension of $\Pr$ whose elements are the square $(n+1)$-cells $A_{f,g}$ and $B_{f,e}$ of the following forms:
\begin{eqn}{equation}
\label{E:CellsCoherentCompletion}
\xymatrix @R=2em @C=2em{
u
  \ar [r] ^-{f}
  \ar [d] _-{\rotatebox{90}{=}}
&
u'
  \ar [r] ^-{f'}
  \ar@2 []!<0pt,-8pt>; [d]!<0pt,+8pt> ^-{A_{f,g}}
&
w
  \ar[d] ^-{e'}
\\
u
  \ar [r] _-{g}
&
v
  \ar [r] _-{g'}
&
w'
}
\qquad  \qquad
\xymatrix @R=2em @C=2em{
u
  \ar [r] ^-{f}
  \ar [d] _-{e}
&
u'
  \ar [r] ^-{f'}
  \ar @2 []!<0pt,-8pt>; [d]!<0pt,+8pt> ^-{B_{f,e}}
&
w
  \ar[d] ^-{e'}
\\
v
  \ar [rr] _-{g'}
&
&
w'
}
\end{eqn}
for all critical strict $S$-branching $(f,g)$ and critical $S$-branching $(f,e)$, where $f,g$ and $e$ are $n$-cells of $S^{\ast(1)}$, $R^{\ast(1)}$ and $E^{\top(1)}$ respectively. Such a family is not unique in general and depends on the $n$-cells~$f',g',e'$ chosen to obtain the confluence of the critical $S$-branchings.

In the rest of this article, we show how to extend a family of generating confluences $\Gamma$ of $\Pr$ to a coherent extension of $\Pr$. The coherent extension will contain the squares obtained by the biaction of $\tck{E}$ defined in \eqref{SSS:ActionOnSquare} on $\Gamma$, the Peiffer squares defined in \eqref{SSS:SquareExtensions}, and the square extension $\squier{E}$. Therefore, we define a \emph{Squier's extension of $\Gamma$} as a square extension of $\Pr$:
\[ \squier{\Gamma} := \sqext \cup \squier{E},
\]
where $\squier{E}$ is a square extension as in~\eqref{SSS:CoherentCompletion}.

\subsection{Coherence by $E$-normalization}

We construct a coherent extension of $\Pr$ under an assumption of confluence and normalization of $S$ with respect to $E$.

\subsubsection{Normalization in polygraphs modulo}
Let us recall the notion of normalization strategy for an $n$-polygraph $P$. Consider a section $s: \cl{P} \fl P_n^\ast$ of the canonical projection $\pi : P_n^\ast \fl \cl{P}$. For an $(n-1)$-cell $u$ in $\cl{P}$ we denote $\hat{u}:=s(u)$, so that $\pi(\hat{u})=u$. 
When $P$ is convergent, a \emph{normalization strategy for $P$ with respect to $s$} is a map
\[
\sigma : P_{n-1}^\ast \fl P_{n}^\ast
\]
that sends an $(n-1)$-cell $u$ of $P_{n-1}^\ast$ to an~$n$-cell $\sigma_u : u \fl \hat{u}$.

The $n$-polygraph modulo $\Pr$ is \emph{normalizing} if every $(n-1)$-cell $u$ in $R_{n-1}^\ast$ admits at least one $S$-normal form, that is $\nf(S,u)\neq \emptyset$.
It is \emph{$E$-normalizing} if $\nf(S,u)\cap \irr(E)\neq \emptyset$ for every $E$-irreducible $(n-1)$-cell $u$ of $R_{n-1}^\ast$. Note that when $S = \ERE$, if $\Pr$ is normalizing then it is $E$-normalizing.

\begin{theorem}
\label{T:CoherentAcyclicity}
Let $\Pr=(R,E,S)$ be an $E$-normalizing $n$-polygraph modulo, and $\Gamma$ be a square extension of $\Pr$ such that $\Pr$ is $\Gamma$-diconvergent. Then any Squier extension $\squier{\Gamma}$ is coherent.
\end{theorem}
\begin{proof}
Let us denote by $\Cr$ the free $n$-category enriched in double groupoid $\bcy{\Gamma}$ generated by the double $(n+1,n-1)$-polygraph $(E,S,\squier{\Gamma})$.
We denote by $\widetilde{u}$ the unique normal form of an~$(n-1)$-cell $u$ in $R_{n-1}^\ast$ with respect to $E$ and we fix a normalization strategy $\rho_u : u \fl \widetilde{u}$ for $E$.

Since $\Pr$ is terminating, the $n$-polygraph $(R_{\leq n-1},S)$ terminates and thus it is normalizing. Moreover, as $\Pr$ is $E$-normalizing, we can define a normalization strategy $\sigma_u : u \fl \hat{u}$ for the polygraph $(R_{\leq n-1},S)$ such that $\hat{u} \in \nf(S,u)\cap \irr(E)$, for every $u \in \irr(E)$. Consider a square
\begin{eqn}{equation}
\label{E:SquareThmCoherence}
\raisebox{0.6cm}{
\xymatrix@C=2em@R=2em{
u \ar [r] ^-{f} \ar [d] _-{e} & v \ar [d] ^-{e'} \\
u' \ar [r] _-{g} & v' 
}}
\end{eqn}
in $\Cr$. By definition the $n$-cell $f$ in $\tck{S}$ can be decomposed (in general in a non unique way) into a zigzag 
sequence
$ f = f_0 \star_{n-1} f_1^{-} \star_{n-1} \cdots \star_{n-1} f_{2n} \star_{n-1} f_{2n+1}^{-}$, where the~$f_{2k} :u_{2k} \fl u_{2k+1}$ and $f_{2k+1} : u_{2k+2} \fl u_{2k+1}$, for all $0\leq k \leq n$ are $n$-cell of~$S^\ast$, with $u_0=u$ and $u_{2n+2}=v$. 

By $\Gamma$-confluence of $\Pr$, there exist $n$-cells $e_{f_i}$ in $\tck{E}$ and $(n+1)$-cells $\sigma_{f_i}$ in $\Cr$ as in the following diagrams:
\[
\xymatrix @R=2em @C=3em {
u_{2k} 
  \ar[r] ^-{f_{2k}} 
  \ar[d] _-{\rho_{u_{2k}}}
& 
u_{2k+1} \ar [r] ^-{\sigma_{u_{2k+1}}} & \hat{u_{2k+1}} \ar@{.>} [d] ^-{e_{f_{2k}}}
\\
\widetilde{u_{2k}}
  \ar[rr] _-{\sigma_{\widetilde{u_{2k}}}} ^-{}="1" 
& & 
\hatt{u_{2k}}
\ar@2 "1,2"!<0pt,-10pt> ; "1"!<-3pt,10pt> ^-{\sigma_{f_{2k}}}
}
\qquad
\qquad
\xymatrix @R=2em @C=3em {
u_{2k+2}
  \ar[r] ^-{f_{2k+1}} 
  \ar[d] _-{\rho_{u_{2k+2}}}
& 
u_{2k+1} 
  \ar [r] ^-{\sigma_{u_{2k+1}}} & \hat{u_{2k+1}} \ar@{.>} [d] ^-{e_{f_{2k+1}}}
\\
\widetilde{u_{2k+2}}
  \ar[rr] _-{\sigma_{\widetilde{u_{2k+2}}}} ^-{}="1" 
& & 
\hatt{u_{2k+2}}
\ar@2 "1,2"!<0pt,-10pt> ; "1"!<0pt,10pt> ^-{\sigma_{{f}_{2k+1}}}
}
\]
for all $0\leq k \leq n$.
By definition of the normalization strategy $\sigma$, for every $0\leq i \leq 2n+1$, the $(n-1)$-cell~$\hatt{u}$ is an $E$-normal form, and by convergence of the $n$-polygraph $E$ it follows that~$\hatt{u_i} = \hatt{u_{i+1}}$. 

Moreover, for every $1\leq i \leq 2n+1$, there exists a square $(n+1)$-cell in $\Cr$ as in the following diagram:
\[
\xymatrix @R=2em @C=2em{
\hat{u_{i+1}}
  \ar[r] ^-{=} _-{}="1"
  \ar[d] _-{e_{f_{i}}}
& 
\hat{u_{i+1}}
  \ar[d] ^-{e_{f_{i+1}}}
\\
\hatt{u_i}
  \ar[r] _-{=} ^-{}="2" 
& 
\hatt{u_{i+2}}
\ar@2 "1"!<-3pt,-10pt> ; "2"!<0pt,10pt> ^-{E_{i+1}}
}
\]

We define a square $(n+1)$-cell $\sigma_f$ in $\Cr$ as the following $\cdg^v$-composite:
\[
\sigma_{f_0}
\cdg^v
E_1
\cdg^v
\sigma_{f_1}
\cdg^v
\sigma_{f_2}
\cdg^v
\ldots
\cdg^v
\sigma_{f_{2n}}
\cdg^v
E_{2n+1}
\cdg^v
\sigma_{f_{2n+1}},
\]
where for an even integer $i\geq 0$, we have
\[
\xymatrix@R=2em@C=2.3em{ u_{i} 
\ar [d] _-{\rho_{u_i}} \ar [r] ^-{f_i} 
& 
u_{i+1} \ar [r] ^-{\sigma_{u_{i+1}}} 
& \widehat{u_{i+1}} \ar [r] ^-{=} ^-{}="6" \ar [d] _-{e_{f_i}} & \widehat{u_{i+1}} \ar [d] ^-{e_{f_{i+1}}} & u_{i+1} \ar [l] _-{\sigma_{u_{i+1}}} 
& u_{i+2} \ar [l] _{f_{i+1}} \ar[r] ^-{f_{i+2}} \ar [d] ^-{\rho_{u_2}} & u_{i+3} \ar [r] ^-{\sigma_{u_{i+3}}} & \hat{u_{i+3}} \ar [r] ^-{=} ^-{}="8" \ar [d] _-{e_{f_{i+2}}} & \hat{u_{i+3}} \ar [d] ^-{e_{f_{i+3}}} &  & \\ \widetilde{u_i} \ar [rr] _{\sigma_{\widetilde{u_i}}} ^-{}="1" & & \hatt{u_i} \ar [r] _-{=} _-{}="7" & \hatt{u_{i+2}} & & \widetilde{u_{i+2}} \ar [ll] ^-{\sigma_{\widetilde{u_{i+2}}}} ^-{}="2" \ar [rr] _ -{\sigma_{\widetilde{u_{i+2}}}} ^-{}="3" & & \hatt{u_{i+2}} \ar [r] _{=} _-{}="9" & \hatt{u_{i+4}} & 
\ar @2 "1,2"!<2.5pt,-10pt>;"1"!<0pt,10pt> ^-{\sigma_{f_i}}
\ar @2 "1,5"!<0pt,-10pt>;"2"!<0pt,10pt> ^-{\sigma_{f_{i+1}}}
\ar @2 "1,7"!<0pt,-10pt>;"3"!<0pt,10pt> ^-{\sigma_{f_{i+2}}}
\ar @2 "6"!<-7pt,-10pt>;"7"!<-4.5pt,10pt> ^-{E_{i+1}}
\ar @2 "8"!<-3pt,-10pt>;"9"!<-3pt,10pt> ^-{E_{i+3}} }
\]

In this way, we have constructed a square $(n+1)$-cell 
\[
\xymatrix @R=2em @C=2em {
u 
  \ar[r] ^-{f} _-{}="1"
  \ar[d] _-{\rho_u}
& 
v 
  \ar[d] ^-{\rho_v}
\\
\widetilde{u}
  \ar[r] _-{\sigma_{\widetilde{u}}\sigma_{\widetilde{v}}^-} ^-{}="2" 
& 
\widetilde{v}
\ar@2 "1"!<0pt,-10pt> ; "2"!<0pt,10pt> ^-{\sigma_f}
}
\]
Similarly, we construct a square $(n+1)$-cell $\sigma_g$ as follows:
\[
\xymatrix @R=2em @C=2em{
\widetilde{u} 
  \ar[r] ^-{\sigma_{\widetilde{u}}\sigma_{\widetilde{v}}^-} _-{}="1"
& 
\widetilde{v}
\\
u'
  \ar[r] _-{g} ^-{}="2" 
  \ar[u] ^-{\rho_{u'}} 
& 
v'
 \ar[u] _-{\rho_{v'}}
\ar@{<-}@2 "1"!<0pt,-10pt> ; "2"!<0pt,10pt> ^-{\sigma_g}
}
\]
using that $\widetilde{u} =\widetilde{u'}$ and $\widetilde{v} =\widetilde{v'}$ by convergence of the $n$-polygraph $E$. 
We obtain a square $(n+1)$-cell $E_e \cdg^v (\sigma_f \cdg^h \sigma_g^-) \cdg^v E_{e'}$ filling the square~\eqref{E:SquareThmCoherence}, as in the following diagram:
\[ 
\xymatrix@R=2em @C=1.6em{
u \ar [rr] ^-{=} ^-{}="3" \ar [dd] _ -{e} & & u \ar [d] ^-{\rho_u} \ar [rrrr] ^-{f} ^-{}="1" & &  & & v \ar [d] _-{\rho_v} \ar [rr] ^-{=} ^-{}="5" & & v \ar [dd] ^-{e'} \\
& & \widetilde{u} \ar [rr] ^-{\sigma_{\widetilde{u}}} & & \hatt{u} = \hatt{v} & & \widetilde{v} \ar [ll] _-{\sigma_{\widetilde{v}}} & & \\
u' \ar [rr] _{=} _-{}="4" & & u' \ar [rrrr] _-{g} _-{}="2" \ar [u] _-{\rho_{u'}} & & & &  v' \ar [rr] _-{=} _-{}="6" \ar [u] ^-{\rho_{v'}} & & v' 
\ar @2 "1"!<0pt,-10pt>;"2,5"!<0pt,10pt> ^-{\sigma_{f}}
\ar @2 "2"!<0pt,10pt>;"2,5"!<0pt,-10pt> _-{\sigma_{g}} 
\ar @2 "3"!<0pt,-20pt>;"4"!<0pt,20pt> ^-{E_{e}}
\ar @2 "5"!<0pt,-20pt>;"6"!<0pt,20pt> ^-{E_{e'}}
} 
\]
\end{proof}

\begin{corollary}
\label{C:mainApplication}
Let $\Pr=(R,E,S)$ be a diconvergent and $E$-normalizing $n$-polygraph modulo, and $\Gamma$ a family of generating confluences of $\Pr$. Then any Squier extension $\squier{\Gamma}$ is coherent.
\end{corollary}

Note that, when the $n$-polygraph $E$ is such that $E_n$ is empty in Corollary~\ref{C:mainApplication}, we recover 
Squier's coherence theorem {\cite[Thm. 5.2]{Squier94}} for convergent $n$-polygraphs, {\cite[Prop. 4.3.4]{GuiraudMalbos09}}.

\subsubsection{Decreasing orders for $E$-normalization}
We give a method to prove that the set $\irr (E)$ is $E$-normalizing, laying on the definition of a termination order for $R$.
A \emph{decreasing order operator} for an $n$-polygraph $P$ is a family of functions 
\[
\Phi_{p,q}: P_{n-1}^{\ast}(p,q) \to \N^{m(p,q)}, 
\] 
indexed by pairs of $(n-2)$-cells $p$ and $q$ in $P_{n-2}^\ast$ satisfying the following three conditions:
\begin{enumerate}[{\bf i)}]
\item If there exists an~$n$-cell $f : u \fl v$ in $P^\ast_n(p,q)$,  then $\Phi_{p,q}(u) > \Phi_{p,q}(v) $, where $>$ is the lexicographic order on $\N^{m(p,q)}$.
 We denote by $>_{\textrm{lex}}$ the partial order on $P_{n-1}^\ast$ defined by $u >_{\textrm{lex}}  v$ if $u$ and $v$ have same source $p$ and target $q$ and $\Phi_{p,q}(u) > \Phi_{p,q}(v) $, 
 \item For all $u,v$ in $P_{n-1}^\ast$ and whisker $C$ on $P_{n-1}^\ast$,  $u >_{\textrm{lex}} v$ implies that  $C[u] >_{\textrm{lex}} C[v]$,
\item The normal forms in $P_{n-1}^\ast (p,q)$ with respect to $P$ are sent to the tuple $(0,\dots,0)$ in $\N^{m(p,q)}$.
\end{enumerate}

Note that an $n$-polygraph admitting a decreasing order operator is terminating. 
 
\subsubsection{Proving coherence modulo using a decreasing order} 
\label{SSS:ProvingCoherenceDecreasingOrder}
Suppose that the $n$-polygraph $E$ is terminating. 
A decreasing order operator $\Phi$ for $E$ is \emph{compatible with $R$} if for every $n$-cell $f:u \fl v$ in $R_n^\ast$, the inequality $\Phi_{p,q}(u) \geq \Phi_{p,q}(v)$ holds.
In that case, the $n$-polygraph modulo $(R,E,R)$ is $E$-normalizing. Indeed, if $u$ is an $E$-normal form in $R_{n-1}^\ast$, $\Phi_{p,q} (u) = (0,\dots,0)$ in $\N^{m(p,q)}$ and by compatibility with $R$, for every $n$-cell $f: u \fl v$ in $R^\ast$, we get $\Phi_{p,q}(v) = (0, \dots, 0)$ so that $v$ is also an $E$-normal form.
We can also prove that the $n$-polygraph modulo $(R,E,\ER)$ is $E$-normalizing if moreover, for every $R$-irreducible $(n-1)$-cell $u$ in $\irr (E)$, any $(n-1)$-cell $u'$ such that there is an $n$-cell $u \fl u'$ in $\tck{E}$ is also $R$-irreducible.  This is for instance the case if $R$ is \emph{left-disjoint from $E$}, that is for every $(n-1)$-cell $u$ in $s(R)$, we have $G_R(u) \cap E_{n-1} = \emptyset$ where:
\begin{enumerate}[{\bf i)}]
\item $s(R)$ is the set of $(n-1)$-sources in $R_{n-1}^\ast$ of generating $n$-cells in $R_n$,
\item for every $u$ in $R_{n-1}^\ast$, $G_R(u)$ is the set of generating $(n-1)$-cells in $R_{n-1}$ contained in $u$.
\end{enumerate}
With these conditions, we can apply Theorem~\ref{T:CoherentAcyclicity} to obtain coherent extensions of polygraphs modulo $(R,E,R)$ or $(R,E,\ER)$. 

\subsection{Coherence by commutation}
We give another method to compute a coherent extension of $\Pr$ based on normalization strategies for the $n$-polygraphs $(R_{\leq n-1},S)$ and $(R_{\leq n-1},E)$ satisfying a commutation property.

\subsubsection{Commuting normalization strategies}
Let $\sigma$ (resp. $\rho$) be a normalization strategy of the $n$-polygraph $(R_{\leq n-1},S)$ (resp. $(R_{\leq n-1},E)$). We say that $\sigma$ and $\rho$ \emph{weakly commute} if, for every $u$ in $R_{n-1}^\ast$, there exists an $n$-cell~$\eta_u$ in $S^\ast$ as in the following diagram: 
\begin{eqn}{equation}
\label{E:SquareSigmaRhoCommute}
\raisebox{0.8cm}{
\xymatrix @C=2em @R=2em{
u \ar [r] ^-{\sigma_u} \ar [d] _-{\rho_u} & \hat{u} \ar [d] ^-{\rho_{\hat{u}}} \\
\widetilde{u} \ar @{.>} [r] _{\eta_u} & \tilda{u}
}}
\end{eqn}
We then denote by $N(\sigma,\rho)$ the square extension of $\Pr$ made of squares of the form~\eqref{E:SquareSigmaRhoCommute}, for every $(n-1)$-cell $u$ in $R_{n-1}^\ast$.
We say that $\sigma$ and $\rho$ \emph{commute} if $\eta_u=\sigma_{\widetilde{u}}$ holds for every $(n-1)$-cell~$u$ in $R_{n-1}^\ast$. The definition is motivated by the fact that $\sigma$ and $\rho$ commute if, and only if, the equality $ \hatt{u} = \tilda{u} $ holds for every $(n-1)$-cell $u$ of $R_{n-1}^\ast$.

\begin{theorem} 
\label{T:CoherentAcyclicity2}
Let $\Pr=(R,E,S)$ be an $n$-polygraph modulo, and $\Gamma$ be a square extension of $\Pr$ such that $\Pr$ is $\Gamma$-diconvergent. If $\sigma$ and~$\rho$ are weakly commuting normalization strategies for $S$ and $E$ respectively, then any square extension $\squier{\Gamma} \cup N(\sigma,\rho)$ is coherent.
\end{theorem}

\begin{proof}
Denote by $\Cr$ the free $n$-category enriched in double groupoids $\dck{(E,S,\squier{\Gamma} \cup N(\sigma,\rho))}$.
For $u$ in $R_{n-1}^\ast$, we denote by $N_u$ the square $(n+1)$-cell in $\Cr$ corresponding to the square~\eqref{E:SquareSigmaRhoCommute}.
We prove that for every $n$-cell $f: u \fl v$ in $S^\ast$, there exists a square $(n+1)$-cell $\widetilde{\sigma_f}$ in $\Cr$ of the following form
\[
\xymatrix @R=2em @C=2em{
\rep{u}
  \ar [d] _-{\rho_{\rep{u}}}
&
u
  \ar[r] ^-{f} _-{}="src"
  \ar[l] _-{\sigma_u} 
&
v
  \ar[r] ^-{\sigma_v}
&
\rep{v}
  \ar [d] ^-{\rho_{\rep{v}}}
\\
\widetilde{\rep{u}}  
\ar [rrr] _-{=} ^-{}="tgt"
& & &
\widetilde{\rep{v}}
\ar@2 "src"!<0pt,-10pt>; "tgt"!<0.5pt,10pt> ^-{\widetilde{\sigma_f}} }
\]
The square $(n+1)$-cell $\widetilde{\sigma_f}$ is obtained as the following composite:
\[
\xymatrix @R=2em @C=4em {
\rep{u}
  \ar [d] _-{\rho_{\rep{u}}}
&
u
  \ar[r] ^-{f} _-{}="src"
  \ar[l] _-{\sigma_u} ^-{}="1"
  \ar[d] ^-{\rho_u}
&
v
  \ar[r] ^-{\sigma_v}
&
\rep{v}
  \ar [r] ^-{=}
&
\rep{v} 
  \ar [r] ^-{=} ^-{}="3"
  \ar@{.>} [d] _-{e_{\eta_u}}
&
\rep{v}
   \ar@{.>} [d] ^-{e_{\hat{v}}} 
   \ar [r] ^-{=} ^-{}="5"
& \rep{v} 
   \ar [d] ^-{\rho_{\rep{v}}} 
\\
\widetilde{\rep{u}}  
& \widetilde{u} \ar [l] ^-{\eta_u} _-{}="2" \ar [rr] _-{\eta_u}
&
&
\widetilde{\rep{v}}
  \ar [r] _-{\sigma_{\widetilde{\rep{v}}}}
&
\widehat{\hatt{v}} \ar [r] _-{=} _-{}="4" 
& 
\widehat{\hatt{v}}
& 
\hatt{v} 
  \ar [l] ^-{\sigma_{\hatt{v}}} _-{}="6"
\ar@2 "1"!<0pt,-10pt>; "2"!<0pt,10pt> ^-{N_u}
\ar@2 "1,4"!<-20pt,-10pt>; "2,4"!<-20pt,10pt> ^-{\eta_f}
\ar@2 "3"!<0pt,-8pt>; "4"!<0pt,7pt> |-{E_{e_{\eta_u}, e_{\hat{v}}}}
\ar@2 "5"!<0pt,-10pt>; "6"!<0pt,10pt> ^-{\gamma_v} }
\]
where the $n$-cell $e_{\eta_u}$ and the square $(n+1)$-cell $\eta_f$ (resp. the $n$-cell $e_{\hat{v}}$ and the square $(n+1)$-cell $\gamma_v$) belongs to $\Cr$ by $\Gamma$-confluence of $\Pr$, and the square $(n+1)$-cell $E_{e_{\eta_u}, e_{\hat{v}}}$ belongs to $\squier{E}$.

Now, let consider a square
\begin{eqn}{equation}
\label{E:SquareThmCoherence2}
\xymatrix@C=2em@R=2em{
u \ar [r] ^-{f} \ar [d] _-{e} & v \ar [d] ^-{e'} \\
u' \ar [r] _-{g} & v' }
\end{eqn}
in $\Cr$. By definition, the $n$-cell $f$ in $\tck{S}$ can be decomposed (in general in a non unique way) into a zigzag 
sequence
\[
f=f_0 \star_{n-1} f_1^{-} \star_{n-1} \cdots \star_{n-1} f_{2n} \star_{n-1} f_{2n+1}^{-},
\]
where the~$f_{2k} :u_{2k} \fl u_{2k+1}$ and $f_{2k+1} : u_{2k+2} \fl u_{2k+1}$, for all $0\leq k \leq n$ are $n$-cells of~$S^\ast$, with $u_0=u$ and $u_{2n+2}=v$. We define a square $(n+1)$-cell $\sigma_f$ as the following vertical composite:
$$ N_u \cdg^v \widetilde{\sigma_{f_0}} \cdg^v  \widetilde{\sigma_{f_1}} \cdg^v \dots \cdg^v  \widetilde{\sigma_{f_{2n+1}}} \cdg^v N_v $$
as depicted on the following diagram
\[
\xymatrix@R=2em@C=2em{u_0 \ar [r]^-{\sigma_{u_0}} ^-{}="1" \ar [d] _-{\rho_{u_0}} & \hat{u_0} \ar[d] ^-{\rho_{\widehat{u_0}}} & u_{0} \ar [l] _-{\sigma_{u_0}}  \ar [r] ^-{f_0} 
& 
u_{1} \ar [r] ^-{\sigma_{u_{1}}} 
& \widehat{u_{1}}  \ar [d] _-{\rho_{\hat{u_1}}}  & u_{1} \ar [l] _-{\sigma_{u_{1}}} 
& u_{2} \ar [l] _{f_{1}} \ar[r] ^-{\sigma_{u_2}} & \rep{u_2} \ar [r] ^-{\sigma_{u_2}} \ar [d] ^-{\rho_{\rep{u_2}}} & u_2   \ar[r] ^-{f_{2}} & u_{3} \ar [r] ^-{\sigma_{u_{3}}} & \hat{u_{3}} \ar [d] _-{\rho_{\rep{u_3}}} & \cdots \\ 
\widetilde{u_0} \ar[r] _-{\eta_{u_0}} _-{}="2" & \tilda{u_0} \ar [rrr] _-{=} _-{}="3" & & & \tilda{u_1} \ar [rrr] _-{=} _-{}="4" & &  & \tilda{u_2} \ar [rrr] _ -{=} ^-{}="5" &  &  & \tilda{u_{3}}  & \cdots  
\ar@2 "1"!<-5pt,-10pt>; "2"!<-5pt,10pt> ^-{N_{u_0}}
\ar@2 "1,3"!<20pt,-10pt>; "3"!<0pt,10pt> ^-{\widetilde{\sigma_{f_0}}}
\ar@2 "1,6"!<20pt,-10pt>; "4"!<0pt,10pt> ^-{\widetilde{\sigma_{f_1}}}
\ar@2 "1,9"!<20pt,-10pt>; "5"!<0pt,10pt> ^-{\widetilde{\sigma_{f_2}}} }
\]
In this way, we have constructed a square $(n+1)$-cell 
\[
\xymatrix @R=2em @C=2em{
u 
  \ar[r] ^-{f} _-{}="1"
  \ar[d] _-{\rho_u}
& 
v 
  \ar[d] ^-{\rho_v}
\\
\widetilde{u}
  \ar[r] _-{\eta_{u}\eta_{v}^-} ^-{}="2" 
& 
\widetilde{v}
\ar@2 "1"!<0pt,-10pt> ; "2"!<0pt,10pt> ^-{\sigma_f}
}
\]
Similarly, we construct a square $(n+1)$-cell $\sigma_g$ as follows:
\[
\xymatrix @R=2em @C=2em{
\widetilde{u} 
  \ar[r] ^-{\eta_{u}\eta_{v}^-} _-{}="1"
& 
\widetilde{v}
\\
u'
  \ar[r] _-{g} ^-{}="2" 
  \ar[u] ^-{\rho_{u'}} 
& 
v'
 \ar[u] _-{\rho_{v'}}
\ar@{<-}@2 "1"!<0pt,-10pt> ; "2"!<0pt,10pt> ^-{\sigma_g}
}
\]
using that $\widetilde{u} =\widetilde{u'}$ and $\widetilde{v} =\widetilde{v'}$ by convergence of the $n$-polygraph $E$. We obtain a square $(n+1)$-cell filling the square~\eqref{E:SquareThmCoherence2}, as in the proof of Theorem \ref{T:CoherentAcyclicity}.
\end{proof}

\subsubsection{Remarks}
When $\sigma$ and $\rho$ are commuting, $\Pr$ is $E$-normalizing. Indeed, the equality $\hatt{u} = \tilda{u}$ implies that the $S$-normal form $\hatt{u}$ is $E$-irreducible. Then Theorem \ref{T:CoherentAcyclicity} applies to prove that $\Sr(\Gamma)$ is acyclic.
We recover the fact that with the hypothesis of Theorem \ref{T:CoherentAcyclicity2} and the assumption that the equality $\eta_u = \sigma_{\widetilde{u}}$ holds for every $u$ in $R_{n-1}^\ast$, we do not need the square $(n+1)$-cells $N_u$ in the coherent extension of $\Pr$, using the following lemma on the square~\eqref{E:SquareSigmaRhoCommute}.

\begin{lemma}
Let $\Pr$ be a terminating $n$-polygraph modulo, and $\Gamma$ be a square extension of $\Pr$ such that $\Pr$ is $\Gamma$-confluent. Then every square in $\sqconf{\Gamma}$ of the form 
\begin{eqn}{equation}
\label{E:SquareConfluent}
\raisebox{+0.7cm}{
\xymatrix@C=2em@R=2em{
u \ar [r] ^-{f} \ar [d] _-{e} & v \ar [r] ^-{f'} & w \ar [d] ^-{e'} \\
u' \ar [r] _-{g} & v' \ar [r] _-{g'} & w' }
}
\end{eqn}
such that $w$ and $w'$ are $S$-normal forms is the boundary of a square $(n+1)$-cell in~$\sqconf{\Gamma}$.
\end{lemma}
\begin{proof}
Consider the square~\eqref{E:SquareConfluent}. By $\Gamma$-confluence of $\Pr$, there exists a $\Gamma$-confluence of the $S$-branching $(f,e,g)$, as in the following diagram:
\[ \xymatrix@C=2em@R=2em{
u \ar [r] ^-{f} \ar [d] _-{e} & v \ar [r] ^-{f_1} & v_1 \ar [d] ^-{e''} \\
u' \ar [r] _-{g} & v' \ar [r] _-{g_1} & v'_1
\ar@2 "1,2"!<0pt,-10pt>;"2,2"!<0pt,10pt> ^-{A} }
\]
By $\Gamma$-confluence of $\Pr$ on the branchings $(f',f_1)$ and $(g_1,g')$ of $S$, there exist square $(n+1)$-cells $B$ and $B'$ as follows:
\[ \xymatrix@R=3em@C=3em {
u \ar [d] _-{\rotatebox{90}{=}} \ar [r] ^-{f} ^-{}="1" & v \ar [d] ^-{\rotatebox{90}{=}}  \ar [rr] ^-{f'} ^-{}="3" &  & w \ar [d] ^-{e_1} & \\
u \ar [d] _ -{e} \ar [r] |-{f} _-{}="2" & v \ar [r] |-{f_1} & v_1 \ar [d] ^-{e_2} \ar [r]|-{f_2} & v_2  \\
u' \ar [d] _-{\rotatebox{90}{=}} \ar [r] |-{g} ^-{}="5" & v' \ar [d] ^-{\rotatebox{90}{=}} \ar [r] |-{g_1} & v'_1 \ar [r] |-{g_2} & v'_2 \ar [d] ^-{e_3} \\
u' \ar [r] |-{g} _-{}="6" &  v' \ar [rr] |-{g'} _-{}="4" &  & w' \\
\ar@2{} "1";"2" |{i_1^h(f)} 
\ar@2{} "5";"6" |{i_1^h(g)}  
\ar@2 "2,2"!<0pt,-10pt>;"3,2"!<0pt,10pt> ^{A}
\ar@2 "3"!<0pt,-10pt>;"2,3"!<0pt,10pt> ^{B}
\ar@2 "3,3"!<0pt,-10pt>;"4"!<0pt,10pt> ^{B'}  } \]
Since $\Pr$ is terminating, one can them assume that $v_2$ and $v'_2$ are $S$-normal forms, and by $\Gamma$-confluence of $\Pr$, there exist an $n$-cell $e'$ in $\tck{E}$ with source $v_2$ and target $v'_2$. By double induction as in Section \ref{S:CoherentConfluenceModulo}, we prove that the square
\[ 
\xymatrix@R=2em @C=2em{
v_1 \ar [d] _-{e_2} \ar [r] ^-{f_2} & v_2 \ar [d] ^-{e'} \\
v'_1 \ar [r] _-{g_2} & v'_2} \]
is the boundary of a square $(n+1)$-cell in $\sqconf{\Gamma}$.
\end{proof}

\section{Globular coherence from double coherence}
\label{S:GlobularCoherenceFromDoubleCoherence}

In this section, we construct a globular coherent presentation for an $n$-category from a double coherent presentation generated by a polygraph modulo. We apply this construction in the situation of commutative monoids, pivotal monoidal categories, and groups.

\subsection{Globular coherence by convergence modulo}
\label{SS:GlobularCoherentConvergenceModulo}

Let $\Pr=(R,E,S)$ be an $n$-polygraph modulo and $\Gamma$ be a square extension of $\Pr$.
Consider the double $(n+1,n-1)$-polygraph $(E,S,\squier{\Gamma})$, where $\squier{\Gamma}$ is the square extension defined in~\eqref{SSS:SquierExtensionModulo}. Let us denote by $((P_i)_{0\leq i \leq n+1}, (Q_i)_{1\leq i \leq n+1})$ the associated $(n+1,n-1)$-dipolygraph $V(E,S,\squier{\Gamma})$, where the functor~$V$ is defined in \eqref{E:QuotientFunctorV}. 
The cellular extension $S$ being defined modulo $E$, we adapt the construction of the $n$-functor $F$ in the definition of the quotient functor $V$ \eqref{SSS:CoherenceFromDoubleCoherence}-{\bf vi)} as follows. 
\begin{enumerate}[{\bf a)}]
\item $F$ is the identity functor on the underlying $(n-2)$-category $R_{n-2}^\ast$, that coincides with $E_{n-2}^\ast$,
\item $F$ sends an $(n-1)$-cell $u$ in $R_{n-1}^\ast$ to its equivalence class $[u]^v$ modulo $E_{n}$,
\item $F$ sends an $n$-cell $f:u \fl v$ in $\tck{S}$ to the $n$-cell $[f]^v : [u]^v \fl [v]^v$ in $(R_{n-1}^\ast)_{E_n}(P_n)$ as defined in \eqref{SSS:CoherenceFromDoubleCoherence}, {\bf iv)-c)} and by setting 
\[
[f]^v = [f_1]^v \star_{n-1} [f_2]^v \star_{n-1} \ldots \star_{n-1} [f_k]^v,
\] 
for $f = e_1\star_{n-1} f_1 \star_{n-1} e_2 \star_{n-1} f_2 \star_{n-1} \ldots \star_{n-1} e_k \star_{n-1} f_k$ in $\tck{S}$, where the $n$-cells $e_i$ and $f_i$ belong to $\tck{E}$ and $\tck{R}$ respectively and may be identity cells. 
\end{enumerate}

As a consequence of Proposition \ref{P:QuotientAcyclic} and Corollary~\ref{C:mainApplication}, we get the following result:

\begin{proposition}
\label{P:QuotientCoherentPresentation}
Let $\Pr=(R,E,S)$ be an $E$-normalizing diconvergent $n$-polygraph modulo, and $\Gamma$ a family of generating confluences of $\Pr$. Then, the $(n+1,n-1)$-dipolygraph $V(E,S,\squier{\Gamma})$ is a globular coherent presentation of the $(n-1)$-category $\cl{\Pr}$. 
\end{proposition}

\begin{theorem}
\label{T:QuotientCoherentPresentation}
Let $\Pr=(R,E,S)$ be an $E$-normalizing diconvergent $n$-polygraph modulo, and $\Gamma$ a family of generating confluences of $\Pr$. Then, the cellular extension 
\[
[\Gamma]^v := \{ [A]^v\;|\; A\in \Gamma\;\}
\]
extends the $n$-category $(R_{n-1}^\ast)_{E_n}(R_n)$ into a globular coherent presentation of the $(n-1)$-category  $\cl{\Pr}$.
\end{theorem}
\begin{proof}
The quotient functor $V$ sends the cellular extension $\Sr(\Gamma)$ to $[\Gamma]^v$. Indeed, any square $(n+1)$-cell $E_{e,e'}$ in $\squier{E}$ yields an identity $(n+1)$-cell in the $(n+1)$-category $(R_{n-1}^\ast)_{E_n}(P_n)(P_{n+1})$:
\[ 
\raisebox{1.2cm}{
\xymatrix@R=2em @C=2em{
u 
  \ar [d] _-{e} \ar [r] ^-{=} ^-{}="1"
& u 
  \ar [d] ^-{e'} \\
v 
  \ar@{.>} [d] _-{e_1} 
& v' 
  \ar@{.>} [d] ^-{e'_1} \\
w \ar@{.>} [r] _-{=} _-{}="2" & w 
\ar@2 "1"!<0pt,-15pt>;"2"!<0pt,15pt> |-{E_{e,e'}} }
}
\qquad
\rightsquigarrow
\qquad
\xymatrix{ [u]^v = [w]^v 
     \ar @/^5ex/ [rr] ^{[i_0^h(u)]^v} _-{}="1"
	\ar @/_5ex/ [rr] _{[i_0^h(w)]^v}  ^-{}="2" 
	 & & 
	 [u]^v = [w]^v
\ar@2 "1"!<0pt,-10pt>;"2"!<0pt,10pt> ^-{}
} 
\] 
Similarly, any $(n+1)$-cell in $\peiffer{\tck{E}}{S^\ast}$ yields an identity $(n+1)$-cell in the $(n+1)$-category $(R_{n-1}^\ast)_{E_n}(P_n)(P_{n+1})$. 
Finally, two square $(n+1)$-cells in the same orbit for the biaction of the $(n,n-1)$-category $\tck{E}$ on $\Sq(\tck{E},S^\ast)$ are sent on the same globular $(n+1)$-cell in $(R_{n-1}^\ast)_{E_n}(P_n)(P_{n+1})$. 
\end{proof}

\subsubsection{Gobular coherent completion procedure for $\ER$} 
\label{SSS:KnuthBendixSquierModulo}
Given a diconvergent $n$-polygraph modulo~$\Pr=(R,E,S)$, Corollary \ref{C:mainApplication} constructs a coherent extension of $\Pr$. In many applications, this result is applied with $S=\ER$ and in situations where $\ER$ is not confluent modulo $E$. When $\ER$ is equipped with a termination order compatible with $R$ modulo $E$, we apply procedure~\eqref{SSS:CompletionProcedure} to complete $R$ into an $n$-polygraph $\check{R}$ such that the $n$-polygraph modulo $(\check{R},E,{}_{E}\check{R})$ is confluent. Moreover, following Corollary~\ref{T:QuotientCoherentPresentation} the only square cells that we have to consider in the construction of the globular coherent presentation through the quotient functor $V$ are the square cells $A_{f,g}$ and $B_{f,e}$ of~\eqref{E:CellsCoherentCompletion}. When $S=\ER$, we do not have to consider square cells of the form $B_{f,e}$. Indeed, the critical $S$-branchings $(f,e)$ where $f$ is an $n$-cell in $\So$ and $e$ is an $n$-cell in $\Eo$ are trivially confluent from~\eqref{SSS:RemarquesERconfluentModuloE}, and the square $(n+1)$-cell $B_{f,e}$ obtained by the following choice of a confluence:
\begin{eqn}{equation}
\label{E:InducedBranchings}
\xymatrix @R=2em @C=2em{
u
  \ar [rr] ^-{f}
  \ar [d] _-{e}
& 
  \ar @2 []!<0pt,-8pt>; [d]!<0pt,+8pt> ^-{B_{f,e}}
&
v
  \ar[d] ^-{\rotatebox{90}{=}}
\\
u'
  \ar [rr] _-{e^- \cdot f}
&
&
v
}
\end{eqn} yields an identity $(n+1)$-cell 
\[ \xymatrix@C=5em{ [u]^v = [u']^v 
     \ar @/^4ex/ [r] ^{[f]^v } _-{}="1"
	\ar @/_4ex/ [r] _{[e^- \cdot f]^v = [f]^v}  ^-{}="2" 
	 &
	 [v]^v
\ar@2 "1"!<0pt,-10pt>;"2"!<0pt,10pt> ^-{i_{[f]^v}}
} 
\]
in the $(n+1)$-category $((R_{n-1}^\ast)_{E_n}(P_n))(P_{n+1})$.
As a consequence, we only need to consider the square $(n+1)$-cells 
\[ 
\xymatrix @R=2em @C=2em{
u
  \ar [r] ^-{f}
  \ar [d] _-{\rotatebox{90}{=}}
&
u'
  \ar [r] ^-{f'}
  \ar@2 []!<0pt,-8pt>; [d]!<0pt,+8pt> ^-{A_{f,g}}
&
w
  \ar[d] ^-{e'}
\\
u
  \ar [r] _-{g}
&
v
  \ar [r] _-{g'}
&
w'
} \] for a choice of a family of generating confluences of strict $S$-branchings $(f,g)$, where $f$ is an $n$-cell of~$\ER^{\ast (1)}$ and $g$ is an $n$-cell of $\Ro$. 
Applying the quotient functor $V$ of \eqref{E:QuotientFunctorV} on the set of square $(n+1)$-cells $A_{f,g}$, following Theorem~\ref{P:QuotientCoherentPresentation}, we obtain an acyclic extension of the $n$-category $(R_{n-1}^\ast)_{E_n}(P_n)$ given by 
\[
\{\;[A_{f,g}]^v \; | \; \text{$(f,g)$ is a critical branching of $S$ modulo $E$}\;\},
\]
where bracket notation $[-]^v$ is defined in~\eqref{SSS:NotationCrochetV}.

\subsection{Commutative monoids}
\label{SS:CommutativeMonoids}
Consider a presentation $(X,R)$ of a commutative monoid $M$, and the associated $(2,1)$-dipolygraph $((P_0,P_1,P_2),(Q_1,Q_2))$ as in \eqref{SSS:CoherentPresentationCommutativeMonoid}. Using completion procedure~\eqref{SSS:KnuthBendixSquierModulo}, we compute a coherent presentation of $M$ by considering the $2$-polygraph modulo $(R,E,\ERE)$, where $E$ and $R$ are the $2$-polygraphs $(P_0,P_1,Q_2)$ and $(P_0,(P_1^\ast)_{Q_2},P_2)$ respectively. 
The $2$-cells of the $2$-polygraph $E$ are oriented with respect to a deglex order induced by a total order on $X$ making it terminating. It is also confluent, by confluence of each of its critical branchings having the following form:
\[ 
\xymatrix@C=3em@R=1em{
     & x_i x_k x_j 
       \ar@2 [r] ^-{\alpha_{i,k} x_j} & x_k x_i x_j \ar@2 [dr] ^-{x_k \alpha_{i,j}} & \\
      x_i x_j x_k 
        \ar@2 [ur] ^-{x_i \alpha_{j,k}}
        \ar@2 [dr] _-{\alpha_{i,j} x_k} & & &  x_k x_j x_i \\
       & x_j x_i x_k \ar@2 [r] _-{x_j \alpha_{i,k}} & x_j x_k x_i \ar@2 [ur] _-{\alpha_{j,k} x_i} &
 } \]
for all $x_i, x_j, x_k$ in $X$ such that $x_i > x_j > x_k$, and the $2$-cells $\alpha_{-,-}$ defined in~\eqref{SSS:CoherentPresentationCommutativeMonoid}. 

\subsubsection{Example}
Suppose that $X=\{ x_1 , x_2 , x_3, x_4 \}$ and $R_2 = \{ x_1 x_3 \overset{\beta}{\dfl} x_2 x_4, \; x_1 x_2 \overset{\gamma}{\dfl} x_1\}$.
There is a critical branching of $\ERE$ modulo $E$ given by
\begin{eqn}{equation}
\label{E:branchingx1x2x3}
\raisebox{0.7cm}{
\xymatrix@R=2em @C=2em{x_1 x_2 x_3 
\ar@2 [rr] ^-{\alpha_{2,3}^- \cdot \beta}
\ar@2 [d] _-{\rotatebox{90}{=}} &   & x_2 x_4  x_2 \\
x_1 x_2 x_3 \ar@2 [r] _-{\gamma} & x_1 x_3 \ar@2 [r]_-{\beta} & x_2 x_4 }
}
\end{eqn}
where $\alpha_{2,3}^- \cdot \beta$ is the rewriting step of $\ERE$ defined by $\xymatrix{x_1 x_2 x_3 \ar@2 [r] ^-{\alpha_{2,3}^-} & x_1 x_3 x_2 \ar@2 [r] ^-{\beta x_2} & x_2 x_4 x_2}$.  
As any permutation of the $x_i$ in $x_2 x_4 x_2$ and $x_2 x_4$ is $R_2$-irreducible, the $1$-cells $x_2 x_4 x_2$ and $x_2 x_4$ are $\ERE$-normal forms. Hence, the branching~\eqref{E:branchingx1x2x3} is not confluent modulo $E$. We add the following $2$-cell 
\[ 
\delta: x_2 x_2 x_4 \dfl x_2 x_4,
\]
and we set $R:=R\cup \{\delta\}$. 
The degree lexicographic order induced by $x_1 > x_2 > x_3 > x_4$ is a termination order compatible with $R_2$ modulo $E$, so that $\ERE$ is terminating and $\irr (E)$ is trivially $E$-normalizing with respect to $\ERE$. 
Moreover, the $2$-polygraph modulo $\ERE$ is confluent modulo $E$. Indeed, all its critical branchings modulo, depicted in~\eqref{E:criticalBranchingExample1} and~\eqref{E:criticalBranchingExample2}, are confluent modulo:
\begin{eqn}{equation}
\label{E:criticalBranchingExample1}
\raisebox{0.7cm}{
\xymatrix@C=3em@R=3em{
x_1 x_2 x_3 
    \ar@2 [r] ^-{\alpha_{2,3}^- \cdot \beta}
    \ar@2 [d] _-{\rotatebox{90}{=}}  
   & 
   x_2 x_4  x_2 
   \ar@2 [r] ^-{\alpha_{2,4}^- \cdot \delta} 
   & 
   x_2 x_4 
     \ar@2 [d] ^-{\rotatebox{90}{=}} \\
x_1 x_2 x_3 
     \ar@2 [r] _-{\gamma} 
     &
      x_1 x_3 \ar@2 [r] _-{\beta} 
      & 
      x_2 x_4 
\ar@3 "1,2"!<0pt,-8pt>;"2,2"!<0pt,8pt> ^-{A}  
}}
\qquad
\raisebox{0.7cm}{
\xymatrix@C=3em@R=3em{
x_2 x_2 x_4 x_1 
    \ar@2 [r] ^-{\alpha_{2,4} \cdot \gamma}
    \ar@2 [d] _-{\rotatebox{90}{=}}  
   & 
   x_2 x_4  x_1 
   \ar@2 [r] ^-{\alpha_{1,4}^- \alpha_{1,2}^- \cdot \gamma} 
   & 
   x_2 x_4 
     \ar@2 [d] ^-{\rotatebox{90}{=}} \\
x_2 x_2 x_4 x_1 
     \ar@2 [r] _-{\delta x_1} 
     &
      x_2 x_4 x_1 
     \ar@2 [r] _-{\alpha_{1,4}^- \alpha_{1,2}^- \cdot \gamma} 
      & 
      x_2 x_4 
\ar@3 "1,2"!<0pt,-8pt>;"2,2"!<0pt,8pt> ^-{B}  
}}
\end{eqn}
\begin{eqn}{equation}
\label{E:criticalBranchingExample2}
\raisebox{0.7cm}{
\xymatrix@C=3em@R=3em{
x_2 x_4 x_2 x_4 x_2
    \ar@2 [r] ^-{\alpha_{2,4}^- \cdot \delta}
    \ar@2 [d] _-{\rotatebox{90}{=}}  
   & 
   x_2 x_4  x_4 x_2 
   \ar@2 [r] ^-{(\alpha_{2,4}^-)^2 \cdot \delta} 
   & 
   x_2 x_4 x_4
     \ar@2 [d] ^-{\rotatebox{90}{=}} \\
x_2 x_4 x_2 x_4 x_2 
     \ar@2 [r] _-{\alpha_{2,4}^- \cdot \delta} 
     &
      x_2 x_4 x_2 x_4 
     \ar@2 [r] _-{\alpha_{2,4}^- \cdot \delta} 
      & 
      x_2 x_4 x_4
\ar@3 "1,2"!<0pt,-8pt>;"2,2"!<0pt,8pt> ^-{C}  
}}
\end{eqn}
Following procedure~\eqref{SSS:KnuthBendixSquierModulo}, we show that an acyclic extension of the commutative monoid  presented by~$(X,R_2)$ can be computed from the the square extension $\{A,B,C\}$ of $(\tck{E},\tck{\ERE})$. This acyclic extension is made of the following $3$-cells:
\[ 
\scalebox{0.9}{
\xymatrix{ [x_1 x_2 x_3] 
     \ar@2 @/^5ex/ [rr] ^{[\beta] \star_1 [\delta]} _-{}="1"
	\ar@2 @/_5ex/ [rr] _{[\gamma] \star_1 [\beta]}  ^-{}="2" 
	 & & 
	 [x_2 x_4]
\ar@3 "1"!<0pt,-10pt>;"2"!<0pt,10pt> ^-{[A]}
}}
\qquad 
\scalebox{0.9}{
\xymatrix{ [x_1 x_2 x_2 x_4] 
     \ar@2 @/^5ex/ [rr] ^{[\delta] \star_1 [\gamma]} _-{}="1"
	\ar@2 @/_5ex/ [rr] _{[\delta] \star_1 [\gamma] }  ^-{}="2" 
	 & & 
	 [x_2 x_4]
\ar@3 "1"!<0pt,-10pt>;"2"!<0pt,10pt> ^-{[B]}
}}
\qquad
\scalebox{0.9}{
\xymatrix{ [x_2 x_2 x_2 x_4 x_4] 
     \ar@2 @/^5ex/ [rr] ^{[\delta] \star_1 [\delta]} _-{}="1"
	\ar@2 @/_5ex/ [rr] _{[\delta] \star_1 [\delta]}  ^-{}="2" 
	 & & 
	 [x_2 x_4 x_4]
\ar@3 "1"!<0pt,-10pt>;"2"!<0pt,10pt> ^-{[C]}
}}
\]

Note that if we take the commutation $2$-cells as rewriting rules, the Knuth-Bendix completion is infinite, requiring to add a $2$-cell $\varepsilon_n  : x_4 x_3^n x_2 x_2 \dfl x_4 x_3^n x_2$ for each $n \geq 0$. This yields an acyclic extension made of an infinite set of $3$-cells
\[ 
\xymatrix@R=0.8em@C=2.5em{
 & x_4 x_3^{n+1} x_2 x_2 
     \ar@2 @/^2ex/ [dr] ^-{\varepsilon_{n+1}} 
 &  \\
  x_4 x_3^n x_2 x_2 x_3 
     \ar@2 @/^2ex/ [ur] ^-{\alpha_{2,3}^2} 
     \ar@2 @/_2ex/ [dr] _-{\varepsilon_n x_3} & & x_4 x_3^{n+1} x_2 \\
  & x_4 x_3^n x_2 x_3 
     \ar@2 @/_2ex/ [ur] _-{\alpha_{2,3}} & 
\ar@3 "1,2"!<0pt,-20pt>;"3,2"!<0pt,20pt> ^-{\;D_n}
}
\]

\subsection{Pivotal monoidal categories}
\label{SS:ExampleDiagrammatic}

We present an application of Theorem~\ref{T:CoherentAcyclicity} in the context of pivotal monoidal categories, seen as a $2$-categories with only one $0$-cell, and thus presented by $3$-polygraphs. The pivotal structure implies that two isotopic string diagrams correspond to the same $2$-cell of the $2$-category. 
However, the isotopy rules produce many critical branchings with primary rules of a given presentation.
Using the structure of polygraphs modulo, we show how to manage primary rules with respect to isotopy rules in the computation of a coherent presentation of the given monoidal category. 
Let us illustrate the method on an example of pivotal  monoidal category admitting one generating $1$-cell and its bidual, and whose $2$-cells are subject to some relations including the relations of the symmetric category, together with one more relation \eqref{E:ThreeCellToyEx}. We also discuss some perspectives to extend our results to presentations of monoidal categories including a relation of the form \eqref{E:relationsPermutations}, such as Khovanov-Lauda's $2$-category~\cite{KhovanovLauda08}, which categorifies quantum groups associated with symmetrizable Kac-Moody algebras, or Heisenberg categories \cite{Khovanov2010,Brundan2017}, presented by rewriting systems that are terminating up to rewriting cycles.

\subsubsection{Example}
\label{SSS:ToyExample}
Let $P$ be the $3$-polygraph defined by the following data:
\begin{enumerate}[{\bf i)}]
\item only one generating $0$-cell,
\item two generating $1$-cells $\curlywedge$ and $\curlyvee$,
\item eight generating $2$-cells pictured by 
\begin{eqn}{equation}
\label{E:TwoCellsDotCrossings}
\udott{}\;, \qquad \crossup{}{}\;, \qquad \ddott{}\;, \qquad \crossdn{}{}\;,
\end{eqn}
\begin{eqn}{equation}
\label{E:TwoCellsCupCap}
\capl{}\;, \qquad  \cupl{}\;, \qquad \capr{}\;, \qquad \cupr{}\;,
\end{eqn}
\item the generating $3$-cells are given by the following families:
\begin{enumerate}[{\bf a)}]
\item the three families of generating $3$-cells of the $3$-polygraph of pearls from~\cite{GuiraudMalbos09}:
\begin{eqn}{equation} 
\label{E:Isotopy1}
\mathord{
\begin{tikzpicture}[baseline = 0, scale=0.911]
  \draw[->,thick,black] (0.3,0) to (0.3,.4);
	\draw[-,thick,black] (0.3,0) to[out=-90, in=0] (0.1,-0.4);
	\draw[-,thick,black] (0.1,-0.4) to[out = 180, in = -90] (-0.1,0);
	\draw[-,thick,black] (-0.1,0) to[out=90, in=0] (-0.3,0.4);
	\draw[-,thick,black] (-0.3,0.4) to[out = 180, in =90] (-0.5,0);
  \draw[-,thick,black] (-0.5,0) to (-0.5,-.4);
   \node at (-0.5,-.5) {};
\end{tikzpicture}
}
 \Rrightarrow
\mathord{\begin{tikzpicture}[baseline=0, , scale=0.911]
  \draw[->,thick,black] (0,-0.4) to (0,.4);
   \node at (0,-.5) {};
\end{tikzpicture}
},\qquad
\mathord{
\begin{tikzpicture}[baseline = 0, , scale=0.911]
  \draw[->,thick,black] (0.3,0) to (0.3,-.4);
	\draw[-,thick,black] (0.3,0) to[out=90, in=0] (0.1,0.4);
	\draw[-,thick,black] (0.1,0.4) to[out = 180, in = 90] (-0.1,0);
	\draw[-,thick,black] (-0.1,0) to[out=-90, in=0] (-0.3,-0.4);
	\draw[-,thick,black] (-0.3,-0.4) to[out = 180, in =-90] (-0.5,0);
  \draw[-,thick,black] (-0.5,0) to (-0.5,.4);
   \node at (-0.5,.5) {};
\end{tikzpicture}
}
\Rrightarrow
\mathord{\begin{tikzpicture}[baseline=0, , scale=0.911]
  \draw[<-,thick,black] (0,-0.4) to (0,.4);
   \node at (0,.5) {};
\end{tikzpicture}
}, \qquad 
\mathord{
\begin{tikzpicture}[baseline = 0]
  \draw[-,thick,black] (0.3,0) to (0.3,-.4);
	\draw[-,thick,black] (0.3,0) to[out=90, in=0] (0.1,0.4);
	\draw[-,thick,black] (0.1,0.4) to[out = 180, in = 90] (-0.1,0);
	\draw[-,thick,black] (-0.1,0) to[out=-90, in=0] (-0.3,-0.4);
	\draw[-,thick,black] (-0.3,-0.4) to[out = 180, in =-90] (-0.5,0);
  \draw[->,thick,black] (-0.5,0) to (-0.5,.4);
   \node at (0.3,-.5) {};
\end{tikzpicture}
}
\Rrightarrow
\mathord{\begin{tikzpicture}[baseline=0, , scale=0.911]
  \draw[->,thick,black] (0,-0.4) to (0,.4);
   \node at (0,-.5) {};
\end{tikzpicture}
}, \qquad
\mathord{
\begin{tikzpicture}[baseline = 0, , scale=0.911]
  \draw[-,thick,black] (0.3,0) to (0.3,.4);
	\draw[-,thick,black] (0.3,0) to[out=-90, in=0] (0.1,-0.4);
	\draw[-,thick,black] (0.1,-0.4) to[out = 180, in = -90] (-0.1,0);
	\draw[-,thick,black] (-0.1,0) to[out=90, in=0] (-0.3,0.4);
	\draw[-,thick,black] (-0.3,0.4) to[out = 180, in =90] (-0.5,0);
  \draw[->,thick,black] (-0.5,0) to (-0.5,-.4);
   \node at (0.3,.5) {};
\end{tikzpicture}
}
\Rrightarrow
\mathord{\begin{tikzpicture}[baseline=0, , scale=0.911]
  \draw[<-,thick,black] (0,-0.4) to (0,.4);
   \node at (0,.5) {};
\end{tikzpicture}
},
\end{eqn}
\begin{eqn}{equation}
\label{E:Isotopy2}
\suld{} \Rrightarrow \dpd{}, \qquad \surdd{} \Rrightarrow \upd{}, \qquad \sdrd{} \Rrightarrow \dpd{}, \qquad \sdld{} \Rrightarrow \upd{},
\end{eqn}
\begin{eqn}{equation}
\label{E:Isotopy3}
\cuprdl{}{}  \Rrightarrow \cuprdr{}{}, \qquad \caprdl{}{} \Rrightarrow \caprdr{}{}, \qquad
\cupldl{}{}  \Rrightarrow \cupldr{}{}, \qquad \capldl{}{} \Rrightarrow \capldr{}{},
\end{eqn}
\item the generating $3$-cells of permutations for both upward and downward orientations of strands:
\begin{eqn}{equation} 
\label{E:relationsPermutations}
\hspace{-0.5cm}
\raisebox{-0.9cm}{
\doublecroisementhaut{}{} \raisebox{7mm}{$\overset{\alpha_+}{\Rrightarrow}$}\doubleidentitehaut{}{}\raisebox{0.9cm}{,} 
 \quad 
 \doublecroisementbas{}{} \raisebox{7mm}{$\overset{\alpha_-}{\Rrightarrow}$} \doubleidentitebas{}{} \raisebox{0.9cm}{,}
 \quad
 \ybgauchehaut{}{}{} \raisebox{8mm}{$\overset{\beta_+}{\Rrightarrow}$} \ybdroithaut{}{}{}
 \raisebox{0.9cm}{,}
 \quad
 \ybgauchebas{}{}{} \raisebox{8mm}{$\overset{\beta_-}{\Rrightarrow}$} \ybdroitbas{}{}{}\raisebox{0.9cm}{,}
 }
\end{eqn}
\item a generating $3$-cell 
\begin{eqn}{equation} 
\label{E:ThreeCellToyEx} 
 \raisebox{-8mm}{$\begin{tikzpicture}[scale=0.45]
\draw[-,thick,black] (0,0) to (1,0.9);
\draw[-,thick,black] (1,0) to (0,0.9);
\draw[-,thick,black] (1,1.1) to (0,2);
\draw[-,thick,black] (0,1.1) to (1,2);
\draw[-,thick,black] (0,0.9) to (0,1.1);
\draw[-,thick,black] (1,0.9) to (1,1.1);
\draw[-,thick] (0,2) to[out=90, in=0] (-0.4,2.4);
\draw[->,thick] (-0.4,2.4) to[out = 180, in = 90] (-0.8,2);
\draw[-,thick,black] (0,0) to[out=-90, in=0] (-0.4,-0.4);
	\draw[-,thick,black] (-0.4,-0.4) to[out = 180, in = -90] (-0.8,0);
	\draw[-,thick,black] (-0.8,0) to (-0.8,2);
	\draw[->,thick,black] (1,2) to (1,2.4);
	\draw[-,thick,black] (1,0) to (1,-0.4);
\end{tikzpicture}$} \quad \raisebox{-3mm}{$\overset{\gamma}{\Rrightarrow}$} \quad \raisebox{-4mm}{$\begin{tikzpicture}[baseline = 0, scale = 1.2]
  \draw[-,thick,black] (0,0.4) to[out=180,in=90] (-.2,0.2);
  \draw[->,thick,black] (0.2,0.2) to[out=90,in=0] (0,.4);
 \draw[-,thick,black] (-.2,0.2) to[out=-90,in=180] (0,0);
  \draw[-,thick,black] (0,0) to[out=0,in=-90] (0.2,0.2);
 \end{tikzpicture}$} \: \: \raisebox{-3mm}{$\mathord{
\begin{tikzpicture}[baseline = 0,scale=0.8]
  \draw[->,thick,black] (0,-0.4) to (0,0.6);
\end{tikzpicture}}$} \raisebox{-0.25cm}{\quad.}
\end{eqn}
\end{enumerate}
\end{enumerate}

Note that the relations~(\ref{E:Isotopy1} -- \ref{E:Isotopy3}) state that the generating $1$-cells~$\curlyvee$ and $\curlywedge$ are biadjoints in  the $2$-category $\cl{P}$ presented by $P$, and cups and caps $2$-cells are units and counits for these adjunctions. Relations implying dots also ensure that the dot $2$-cell is a cyclic $2$-morphism in the sense of \cite{CockettKoslowskiSeely00} for the biadjunction $\curlyvee \vdash \curlywedge \vdash \curlyvee$, making $\cl{P}$ into a \emph{pivotal $2$-category}.

\subsubsection{Confluence modulo isotopy} 
We consider the $3$-polygraph $E$ defined by the following data
\begin{enumerate}[{\bf i)}]
\item $E_{\leq 1} :=P_{\leq 1}$,
\item it has the four generating $2$-cells given in \eqref{E:TwoCellsCupCap} and the two dot generating $2$-cells in \eqref{E:TwoCellsDotCrossings},
\item the isotopy $3$-cells (\ref{E:Isotopy1} -- \ref{E:Isotopy3}) of the $3$-polygraph of pearls.
\end{enumerate}
Let $R$ be a $3$-polygraph such that $R_{\leq 2} := P_{\leq 2}$, and whose $3$-cells are given by $(\alpha_{\pm}, \beta_{\pm} , \gamma)$ of (\ref{E:relationsPermutations} -- \ref{E:ThreeCellKLR}), and let us consider the $3$-polygraph modulo $\ER{}{}$. Following \eqref{SSS:RemarquesERconfluentModuloE}, the only critical branchings we have to consider are those of the form $(f,g)$ with $f$ in $\ER^{\ast (1)}$ and $g$ in $\Ro$. This set of critical branchings can be reduced to the branchings $(f,g)$ with $f,g$ in $\Ro$. There is no critical branching modulo between $\gamma$ and $\alpha_{\pm}$ or $\beta_{\pm}$, thus the only critical branchings we have to consider are those of the $3$-polygraph of permutations described in {\cite[5.4.4]{GuiraudMalbos09}}, with both upward and downward direction of strands.

\subsubsection{Decreasing order operator for $E$-normalization}
The $3$-polygraph $R' := (R_0,R_1,R_2,R_3 \backslash \{ \gamma \})$ is left-disjoint from $E$, since no caps and cups $2$-cells appear in the sources of the generating $3$-cells of $R'$. Moreover, starting from an $R$-irreducible $2$-cell of $R_2^\ast$, using the $3$-cell $\gamma$ one can only create isotopies that do not overlap with the source of $\gamma$, that one can remove in $\ER$-rewriting paths, so that one can still reach an $E$-normal form. Therefore, one can use \eqref{SSS:ProvingCoherenceDecreasingOrder} to prove that the polygraph modulo $(R,E,\ER)$ is $E$-normalizing using a decreasing order operator $\Phi$ for $E$ compatible with $R$.

\begin{lemma}
There exists a decreasing operator order $\Phi$ for $E$ compatible with $R$.
\end{lemma}
\begin{proof}
For all $1$-cells $p$ and $q$ in $R_1^\ast$, we set $m(p,q)=2$, and for every $2$-cell $u:p\dfl q$ of $R_2^\ast$, we set $\Phi_{p,q}(u) = (\text{ldot} (u), \text{I}(u))$, where:
\begin{enumerate}[{\bf i)}]
\item $\text{ldot}(u)$ counts the number of left-dotted caps and cups, adding for such cap and cup the number of dots on it. In particular, for every $n$ in $\N^\ast$, we have
\[ 
\text{ldot} \left( \: \begin{tikzpicture}[baseline = 0]
	\draw[-,thick,black] (0.4,-0.1) to[out=90, in=0] (0.1,0.4);
	\draw[-,thick,black] (0.1,0.4) to[out = 180, in = 90] (-0.2,-0.1);
      \node at (-0.14,0.25) {$\color{black}\bullet$};
      \node at (-0.35,0.25) {$n$};
\end{tikzpicture} \: \right) = \text{ldot} \left( \: \begin{tikzpicture}[baseline = 0]
	\node at (-0.16,0.15) {$\color{black}\bullet$};
	\node at (-0.35,0.15) {$n$};
	\draw[-,thick,black] (0.4,0.4) to[out=-90, in=0] (0.1,-0.1);
	\draw[-,thick,black] (0.1,-0.1) to[out = 180, in = -90] (-0.2,0.4);
\end{tikzpicture} \; \right) = n+1 \]
for both orientations of strands.

\item $\text{I}(u)$ counts the number of instances of one of the following $2$-cells of $R_2^\ast$ in $u$: 
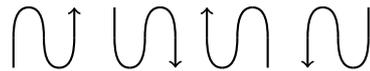
\[ \mathord{
\begin{tikzpicture}[baseline = 0]
  \draw[->,thick,black] (0.3,0) to (0.3,.4);
	\draw[-,thick,black] (0.3,0) to[out=-90, in=0] (0.1,-0.4);
	\draw[-,thick,black] (0.1,-0.4) to[out = 180, in = -90] (-0.1,0);
	\draw[-,thick,black] (-0.1,0) to[out=90, in=0] (-0.3,0.4);
	\draw[-,thick,black] (-0.3,0.4) to[out = 180, in =90] (-0.5,0);
  \draw[-,thick,black] (-0.5,0) to (-0.5,-.4);
   \node at (-0.5,-.5) {};
\end{tikzpicture}
}
\quad 
\mathord{
\begin{tikzpicture}[baseline = 0]
  \draw[->,thick,black] (0.3,0) to (0.3,-.4);
	\draw[-,thick,black] (0.3,0) to[out=90, in=0] (0.1,0.4);
	\draw[-,thick,black] (0.1,0.4) to[out = 180, in = 90] (-0.1,0);
	\draw[-,thick,black] (-0.1,0) to[out=-90, in=0] (-0.3,-0.4);
	\draw[-,thick,black] (-0.3,-0.4) to[out = 180, in =-90] (-0.5,0);
  \draw[-,thick,black] (-0.5,0) to (-0.5,.4);
   \node at (-0.5,.5) {};
\end{tikzpicture}
}
\quad 
\mathord{
\begin{tikzpicture}[baseline = 0]
  \draw[-,thick,black] (0.3,0) to (0.3,-.4);
	\draw[-,thick,black] (0.3,0) to[out=90, in=0] (0.1,0.4);
	\draw[-,thick,black] (0.1,0.4) to[out = 180, in = 90] (-0.1,0);
	\draw[-,thick,black] (-0.1,0) to[out=-90, in=0] (-0.3,-0.4);
	\draw[-,thick,black] (-0.3,-0.4) to[out = 180, in =-90] (-0.5,0);
  \draw[->,thick,black] (-0.5,0) to (-0.5,.4);
   \node at (0.3,-.5) {};
\end{tikzpicture}
}
\quad
\mathord{
\begin{tikzpicture}[baseline = 0]
  \draw[-,thick,black] (0.3,0) to (0.3,.4);
	\draw[-,thick,black] (0.3,0) to[out=-90, in=0] (0.1,-0.4);
	\draw[-,thick,black] (0.1,-0.4) to[out = 180, in = -90] (-0.1,0);
	\draw[-,thick,black] (-0.1,0) to[out=90, in=0] (-0.3,0.4);
	\draw[-,thick,black] (-0.3,0.4) to[out = 180, in =90] (-0.5,0);
  \draw[->,thick,black] (-0.5,0) to (-0.5,-.4);
   \node at (0.3,.5) {};
\end{tikzpicture}
} \]
\end{enumerate} 
For every $3$-cell $u \tfl v$ in $E$, we have $\Phi(u) > \Phi(v)$ and that $\Phi(u,u) = (0,0)$ when $u$ in $\irr (E)$. Moreover, $\Phi$ is compatible with $R$ because $R$-rewriting steps do not make the dot $2$-cell move around a cup or a cap, or create sources of isotopies.
\end{proof}

We deduce from Theorem~\ref{T:CoherentAcyclicity} a coherent extension of the polygraph modulo $(R,E,\ER)$. 
This square extension is made of the ten elements given by the diagrams of the homotopy basis for the $3$-polygraph of permutations from \cite[5.4.4]{GuiraudMalbos09} for both upward and downward orientations of strands and the $16$ elements given by the diagrams of the homotopy basis or the $3$-polygraph of pearls in \cite[Section 5.5.3]{GuiraudMalbos09} for both orientations of strands.

\begin{remark}
If we consider the linear~$(3,2)$-polygraph $P'$ whose $i$-cells are those of $P$ for $0 \leq i \leq 3$, but the $3$-cell $\gamma$ is replaced by the following $3$-cell:
\begin{eqn}{equation}
\label{E:ThreeCellKLR}
\mathord{
		\begin{tikzpicture}[baseline = 0,scale=0.7]
		\draw[<-,thick,black] (0.3,.5) to (-0.3,-.5);
		\draw[-,thick,black] (-0.2,.2) to (0.2,-.3);
		\draw[-,thick,black] (0.2,-.3) to[out=130,in=180] (0.5,-.4);
		\draw[-,thick,black] (0.5,-.4) to[out=0,in=270] (0.8,.5);
		\draw[-,thick,black] (-0.2,.2) to[out=130,in=0] (-0.5,.5);
		\draw[->,thick,black] (-0.5,.5) to[out=180,in=-270] (-0.8,-.5);
		
		\draw[<-,thick,black] (-0.3,-.5) to (0.3,-1.5);
		\draw[-,thick,black] (-0.5,-1.5) to[out=180,in=-90] (-0.8,-.5);
		\draw[-,thick,black] (-0.2,-1.2) to[out=230,in=0] (-0.5,-1.5);
		\draw[-,thick,black] (-0.2,-1.2) to (0.2,-0.7); 
		\draw[-,thick,black] (0.2,-0.7) to[out=50,in=180] (0.5,-0.6);
		\draw[->,thick,black] (0.5,-0.6) to[out=0,in=90] (0.7,-1.5);
		\end{tikzpicture}
} \; \raisebox{-3mm}{$\overset{\gamma'}{\Rrightarrow}$}  \quad \mathord{
\begin{tikzpicture}[baseline = 0,scale=2]
	\draw[<-,thick,black] (0.08,-.5) to (0.08,.2);
	\draw[->,thick,black] (-0.18,-.5) to (-0.18,.2);
\end{tikzpicture}
}  
\end{eqn}
which is relation arising in many presentations of monoidal categories appearing in representation category, see for instance Khovanov-Lauda's $2$-category introduced in~\cite{KhovanovLauda08} or in the Heisenberg categories defined by Khovanov in~\cite{Khovanov2010}, and extended by Brundan in~\cite{Brundan2017}. 
Note that with this new relation creating branchings with the isotopy relations, the $3$-polygraph $P'$ is not confluent. Indeed, the branching 
\begin{eqn}{equation}
\label{branch}
\raisebox{1.7cm}{
\xymatrix @C=5em @R=0.05em{ 
& \;\; \mathord{
\begin{tikzpicture}[baseline = 0,scale=0.611]
	\draw[<-,thick,black] (0.3,.5) to (-0.3,-.5);
	\draw[-,thick,black] (-0.2,.2) to (0.2,-.3);
        \draw[-,thick,black] (0.2,-.3) to[out=130,in=180] (0.5,-.4);
        \draw[-,thick,black] (-0.2,.2) to[out=130,in=0] (-0.5,.5);
        \draw[->,thick,black] (-0.5,.5) to[out=180,in=-270] (-0.8,-.5);        
      \draw[<-,thick,black] (-0.3,-.5) to (0.3,-1.5);
         \draw[-,thick,black] (-0.5,-1.5) to[out=180,in=-90] (-0.8,-.5);
         \draw[-,thick,black] (-0.2,-1.2) to[out=230,in=0] (-0.5,-1.5);
         \draw[-,thick,black] (-0.2,-1.2) to (0.2,-0.7); 
         \draw[-,thick,black] (0.2,-0.7) to[out=50,in=180] (0.5,-0.6);
          \draw[->,thick,black] (0.5,-0.6) to[out=0,in=90] (0.7,-1.5);
\end{tikzpicture}
} \\
\mathord{
\begin{tikzpicture}[baseline = 0,scale=0.611]
	\draw[<-,thick,black] (0.3,.5) to (-0.3,-.5);
	\draw[-,thick,black] (-0.2,.2) to (0.2,-.3);
        \draw[-,thick,black] (0.2,-.3) to[out=130,in=180] (0.5,-.4);
        \draw[-,thick,black] (0.5,-.4) to[out=0,in=270] (0.8,.5);
        \draw[-,thick,black] (-0.2,.2) to[out=130,in=0] (-0.5,.5);
        \draw[->,thick,black] (-0.5,.5) to[out=180,in=-270] (-0.8,-.5);        
      \draw[<-,thick,black] (-0.3,-.5) to (0.3,-1.5);
         \draw[-,thick,black] (-0.5,-1.5) to[out=180,in=-90] (-0.8,-.5);
         \draw[-,thick,black] (-0.2,-1.2) to[out=230,in=0] (-0.5,-1.5);
         \draw[-,thick,black] (-0.2,-1.2) to (0.2,-0.7); 
         \draw[-,thick,black] (0.2,-0.7) to[out=50,in=180] (0.5,-0.6);
          \draw[->,thick,black] (0.5,-0.6) to[out=0,in=90] (0.7,-1.5);
          	\draw[-,thick,black] (1.4,0.5) to[out=90, in=0] (1.1,0.9);
	\draw[->,thick,black] (1.1,0.9) to[out = 180, in = 90] (0.8,0.5);
	\draw[-,thick,black] (1.4,0.5) to (1.4,-1.2);
\end{tikzpicture}
}
\ar@3 @<-12pt>@/^2ex/  [ur] ^{} 
\ar@3@/_2ex/  [dr] _{} & \\
& \;\; \mathord{
\begin{tikzpicture}[baseline = 0,scale=0.911]
    \draw[->,thick,black] (0,0) to (0,0.4);
	\draw[-,thick,black] (0.8,0) to[out=90, in=0] (0.5,0.4);
	\draw[->,thick,black] (0.5,0.4) to[out = 180, in = 90] (0.2,0);
\end{tikzpicture}
}}
}
\end{eqn}
is not confluent. Moreover, solving this obstruction to confluence using Knuth-Bendix completion may lead to adding a great number of new relations, making analysis of confluence from critical branchings inefficient. To tackle this issue, this is convenient to rewrite modulo the isotopy relations.
In that case, there are critical branchings modulo isotopy of the form $(\Ro, \ER^{\ast (1)})$ between $\gamma'$ and $\alpha_+$ (resp. $\beta_+$) with respective source
\begin{eqn}{equation}
\label{E:BranchingsModuloIsotopy}
\mathord{
		\begin{tikzpicture}[baseline = 0,scale=0.7]
		\draw[<-,thick,black] (0.3,.5) to (-0.3,-.5);
		\draw[-,thick,black] (-0.2,.2) to (0.2,-.3);
		\draw[-,thick,black] (-0.2,.2) to[out=130,in=0] (-0.5,.5);
		\draw[->,thick,black] (-0.5,.5) to[out=180,in=-270] (-0.8,-.5);     
		\draw[<-,thick,black] (-0.3,-.5) to (0.3,-1.5);
		\draw[-,thick,black] (-0.5,-1.5) to[out=180,in=-90] (-0.8,-.5);
		\draw[-,thick,black] (-0.2,-1.2) to[out=230,in=0] (-0.5,-1.5);
		\draw[-,thick,black] (-0.2,-1.2) to (0.2,-0.7); 
		\draw[-,thick,black] (0.2,-0.7) to (0.19,-0.26);
		\end{tikzpicture}
	}\: \raisebox{-5mm}{$\sim$} \: \mathord{
		\begin{tikzpicture}[baseline = 0,scale=0.7]
		\draw[<-,thick,black] (0.3,.5) to (-0.3,-.5);
		\draw[-,thick,black] (-0.2,.2) to (0.2,-.3);
		\draw[-,thick,black] (-0.2,.2) to[out=130,in=0] (-0.5,.5);
		\draw[->,thick,black] (-0.5,.5) to[out=180,in=-270] (-0.8,-.5);     
		\draw[<-,thick,black] (-0.3,-.5) to (0.3,-1.5);
		\draw[-,thick,black] (-0.5,-1.5) to[out=180,in=-90] (-0.8,-.5);
		\draw[-,thick,black] (-0.2,-1.2) to[out=230,in=0] (-0.5,-1.5);
		\draw[-,thick,black] (-0.2,-1.2) to (0.2,-0.7); 
		\draw[-,thick,black] (0.2,-.3) to[out=130,in=180] (0.5,-.4);
		\draw[-,thick,black] (0.2,-0.7) to[out=50,in=180] (0.5,-0.6);
		\draw[-,thick,black] (0.5,-.4) to[out=0,in=270] (0.8,.3); 
		\draw[-,thick,black] (1.4,0.3) to[out=90, in=0] (1.1,0.7);
		\draw[-,thick,black] (1.1,0.7) to[out = 180, in = 90] (0.8,0.3);
		\draw[-,thick,black] (0.5,-0.6) to[out=0,in=90] (0.7,-1.2);
		\draw[-,thick,black] (1.3,-1.2) to[out=-90, in=0] (1,-1.6); %
		\draw[-,thick,black] (1,-1.6) to[out = 180, in = -90] (0.7,-1.2);
		\draw[-,thick,black] (1.3,-1.2) -- (1.4, 0.3);
		\end{tikzpicture}
	}  \: \raisebox{-5mm}{$,$} \: \qquad\raisebox{-1mm}{$\mathord{
		\begin{tikzpicture}[baseline = 0,scale=0.7]
		\draw[<-,thick,black] (0.3,.5) to (-0.3,-.5);
		\draw[-,thick,black] (-0.2,.2) to (0.2,-.3);
		\draw[-,thick,black] (-0.2,.2) to[out=130,in=0] (-0.5,.5);
		\draw[->,thick,black] (-0.5,.5) to[out=180,in=-270] (-0.8,-.5);     
		\draw[<-,thick,black] (-0.3,-.5) to (0.3,-1.5);
		\draw[-,thick,black] (-0.5,-1.5) to[out=180,in=-90] (-0.8,-.5);
		\draw[-,thick,black] (-0.2,-1.2) to[out=230,in=0] (-0.5,-1.5);
		\draw[-,thick,black] (-0.2,-1.2) to (0.2,-0.7); 
		\draw[-,thick,black] (0.19,-0.26) to (0.59, -0.76);
		\draw[-,thick,black] (0.2,-0.7) to (0.6, -0.2);
		\draw[->,thick,black] (0.6,-0.2) to (0.6, 0.5);
		\draw[-,thick,black] (0.59,-0.76) to (0.6, -1.5);
		\end{tikzpicture}
	}$} \: \raisebox{-5mm}{$\sim$} \: \raisebox{-2mm}{$\mathord{
			\begin{tikzpicture}[baseline = 0,scale=0.7]
			\draw[<-,thick,black] (0.3,.5) to (-0.3,-.5);
			\draw[-,thick,black] (-0.2,.2) to (0.2,-.3);
			\draw[-,thick,black] (-0.2,.2) to[out=130,in=0] (-0.5,.5);
			\draw[->,thick,black] (-0.5,.5) to[out=180,in=-270] (-0.8,-.5);     
			\draw[<-,thick,black] (-0.3,-.5) to (0.3,-1.5);
			\draw[-,thick,black] (-0.5,-1.5) to[out=180,in=-90] (-0.8,-.5);
			\draw[-,thick,black] (-0.2,-1.2) to[out=230,in=0] (-0.5,-1.5);
			\draw[-,thick,black] (-0.2,-1.2) to (0.2,-0.7); 
			\draw[-,thick,black] (0.2,-.3) to[out=130,in=180] (0.5,-.4);
			\draw[-,thick,black] (0.2,-0.7) to[out=50,in=180] (0.5,-0.6);
			\draw[-,thick,black] (0.5,-.4) to[out=0,in=270] (0.8,.3); 
			\draw[-,thick,black] (1.4,0.3) to[out=90, in=0] (1.1,0.7);
			\draw[-,thick,black] (1.1,0.7) to[out = 180, in = 90] (0.8,0.3);
			\draw[-,thick,black] (0.5,-0.6) to[out=0,in=90] (0.7,-1.2);
			\draw[-,thick,black] (1.3,-1.2) to[out=-90, in=0] (1,-1.6); %
			\draw[-,thick,black] (1,-1.6) to[out = 180, in = -90] (0.7,-1.2);
			\draw[-,thick,black]  (1.3,-1.2) to (1.3,-0.7);
			\draw[-,thick,black] (1.4,0.3) -- (1.4,-0.26);
			\draw[-,thick,black]  (1.4,-0.26) to (1.8,-0.76);
			\draw[-,thick,black]  (1.3,-0.7) to (1.75, -0.2);
			\draw[-,thick,black]  (1.75,-0.2) to (1.75,0.5);
			\draw[-,thick,black]  (1.8,-0.76) to (1.8,-1.5);  	
			\end{tikzpicture}
		}$} \: \raisebox{-5mm}{$,$} 
\end{eqn}
and to get confluence of these branchings, we have to add a bubble slide relation in $R$ of the form:
\[ \raisebox{-1mm}{$\begin{tikzpicture}[baseline = 0, scale = 1.2]
  \draw[-,thick,black] (0,0.4) to[out=180,in=90] (-.2,0.2);
  \draw[->,thick,black] (0.2,0.2) to[out=90,in=0] (0,.4);
 \draw[-,thick,black] (-.2,0.2) to[out=-90,in=180] (0,0);
  \draw[-,thick,black] (0,0) to[out=0,in=-90] (0.2,0.2);
 \end{tikzpicture}$} \: \: \mathord{
\begin{tikzpicture}[baseline = 0,scale=0.8]
  \draw[->,thick,black] (0,-0.4) to (0,0.6);
\end{tikzpicture}} \overset{}{\Rrightarrow} \mathord{
\begin{tikzpicture}[baseline = 0,scale=0.8]
  \draw[->,thick,black] (0,-0.4) to (0,0.6);
\end{tikzpicture}} \: \: \raisebox{-2mm}{$\begin{tikzpicture}[baseline = 0, scale=1.2]
  \draw[->,thick,black] (0,0.4) to[out=180,in=90] (-.2,0.2);
  \draw[-,thick,black] (0.2,0.2) to[out=90,in=0] (0,.4);
 \draw[-,thick,black] (-.2,0.2) to[out=-90,in=180] (0,0);
  \draw[-,thick,black] (0,0) to[out=0,in=-90] (0.2,0.2);
 \end{tikzpicture}$} \]
 As a consequence, following \cite[Ex. 4.3.2]{Alleaume16}, $\ER$ cannot be terminating. 
 The study of this type of examples requires to extend our results to the cases of \emph{quasi-terminating} polygraphs modulo, that is terminating up to rewriting loops.
 \end{remark}
 
\subsection{Groups}
We apply the coherent completion procedure to the case of group rewriting systems, by considering the notion of positive rewriting step introduced in \cite{ChenavierDupontMalbos20}. We expect that this would allow to compute generating syzygies for group presentations, but this would require to extend the coherence results of this article to positive coherence, that is coherence from squares that are obtained from positive rewriting paths. This procedure is a first approach towards coherence in groups by rewriting, and should be improved to deal with non-necessary syzygies, and branchings that need not to be considered.

\subsubsection{Positive group rewriting} 
In \cite{ChenavierDupontMalbos20}, a notion of group rewriting system was introduced, and based on rewriting modulo the inverse axioms. In order to avoid termination obstructions, rewriting are defined with respect to \emph{positive} rewriting steps as follows. 
Let $(X,R)$ be a presentation of a group $G$, and $>$ be a total order on $X$. 
We define the order $\succ$ on the free group $F(X)$ by $u \succ v$ if:
\begin{enumerate}[{\bf i)}]
\item $\ell(u) > \ell(v)$, where $\ell(u)$ corresponds to the minimal number of elements of $X$ needed to write $u$, or
\item $\ell(u) = \ell(v)$ and $\hat{u} >_{\text{lex}} \hat{v}$, where $\hat{u}$ (resp. $\hat{v}$) is the normal form of $u$ (resp. $v$) with respect to group relations $xx^- \dfl 1$ and $x^-x \dfl 1$ for every $x$ in $X$, and $>_{\text{lex}}$ is the lexicographic order on $>$.
\end{enumerate}   
Positive rewriting steps are then reductions of the form
\[ u r_1 v \dfl u r_2^{-} v \]
where $u,v$ are elements of the free group $F(X)$ generated by $X$, $r = r_1 r_2$ is an element of $R \cup R^-$, and such that $u r_1 v \succ u r_2^- v$.

Let us consider a $(2,1)$-dipolygraph $((P_0,P_1,P_2),(Q_1,Q_2))$ presenting $G$ as in \eqref{SSS:CoherentPresentationGroups}.
We will consider the $2$-polygraph modulo $(R,E,\ER)$, where $E$ is the $2$-polygraph $(P_0,P_1,P_2)$, and $R$ is the $2$-polygraph $(P_0,P_1,P'_2)$, where $P'_2$ is the cellular extension of $P_1^\ast$ made of elements of the form $r \dfl 1$ and $r^- \dfl 1$, for every $r \in R$. The $2$-polygraph $E$ is convergent, indeed it is terminating since its rules strictly decrease the length of words, and its confluence is ensured by the confluence of its critical branchings as follows:
\[
\xymatrix{
xx^- x \ar@2@/^3ex/ [rr]
\ar@2@/_3ex/ [rr] & & x 
}.
\]

\medskip
Note that we say that an $\ER$-rewriting step $(e,f)$ for $e$ in $\tck{E}$ and $f \in \Ro$ is \emph{positive} if $f$ is a \emph{positive} $R$-rewriting step as defined above. Following \cite{ChenavierDupontMalbos20}, we will study confluence properties for the $2$-polygraph modulo $(R,E,\ER)$ with respect to positive rewriting paths.

\subsubsection{Example: the braid group}
Let us illustrate the procedure \eqref{SSS:KnuthBendixSquierModulo} for the braid groups on three strands, presented with three generators $s,t,a$ ordered by $s > t > a$, and the rules $f:sta^- \dfl 1$, $g:tas^-a^- \dfl 1$.

The rules induce $\ER$-rewriting steps of the form $st \equiv sta^- a \dfl a$ or $tas^- \equiv tas^- a^- a \dfl a$, $ta \equiv tas^- a^- as \dfl as$. Note that these $\ER$-rewriting steps or positive since  $st \succ a$, $tas^- \succ a$ and $ta \succ as$. However, the $\ER$-rewriting step $s \equiv sta^- a t^- \dfl a t^-$ is not positive. These induced reductions yield critical branchings of the form 
\[ 
\xymatrix{
sta^-a \ar@2[d] _-{e} \ar@2[r]^-{fa} & a \ar@2[d] ^{\rotatebox{90}{=}} \\
st \ar@2[r] _-{e^- \cdot fa} & a 
} \]
However, these critical branchings do not need to be considered in the computation of a coherent extension as explained in \eqref{SSS:KnuthBendixSquierModulo}, since they are branchings of the form  \eqref{E:InducedBranchings}. Therefore, we only have to consider critical branchings of the form $(f,g)$, where $f$ is an $\ER$-rewriting step and $g$ is an $R$-rewriting step. There is such a critical branching given by
\[ 
\xymatrix{
stas^- a^- \ar@2[r] ^-{sg} \ar@2[d] _-{\rotatebox{90}{$\sim$}} & s \\
sta^- a a s^- a^- \ar@2[r] _-{faas^-a^-} & aas^- a^- 
}
\]
which is not confluent, so that we have to add the $h:aas^- a^- s^- \dfl 1$ and its inverse $h^-:sasa^- a^- \dfl 1$ in $R$.
This critical branching could be written differently, for instance 
\[ 
\xymatrix@C=5em{
sta^- \ar@2[r] ^-{f} \ar@2 [d] _-{\rotatebox{90}{$\sim$}} & 1 \\
stas^-a^-asa^- a^-  \ar@2[r] _-{sgasa^-a^-} & sasa^- a^-
}
\]
but these branchings are equivalent for the bi-action of $\tck{E}$ defined in \eqref{SSS:ActionOnSquare}, so that we only need to consider one of them.
There is an $(\ER,R)$-critical branching involving $f$ and $h$ as follows:
\[ 
\xymatrix@C=4em{
 sasta^- \ar@2[r] ^-{saf} \ar@2 [d] _-{\rotatebox{90}{$\sim$}} & sa \\
sasa^-a^-aata^- \ar@2[r] _-{haata'} & aata^-
}
\]
that is not confluent, so we have to add a new rule $k:aatap -a^-s^- \dfl 1$ and its inverse in $R$. One then checks that with the $2$-polygraph $\check{R}$ defined by $\check{R}_{\leq 1} = R_{\leq 1}$ and $\check{R}_2$ contains as generating $2$-cells $f,g,h,k$ and their inverses satisfies that the $2$-polygraph $(\check{R},E,{}_E \check{R})$ is confluent. One also checks that the following square cells form a family of generating confluence for critical $\ER$-branchings of the form $(f,g)$, with $f$ in $\Ro$ and $g$ in $\ER^{\ast (1)}$:
\[ 
\xymatrix@C=5em{
sta^- \ar@2[rr] ^-{f} ^-{}="src" \ar@2 [d] _-{\rotatebox{90}{$\sim$}} & &  1  \ar@2 [d] ^-{\rotatebox{90}{=}}\\
stas^-a^-asa^- a^-  \ar@2[r] _-{sgasa^-a^-} & sasa^- a^- \ar@2[r] _-{h^-} & 1
\ar@3 "src"!<8pt,-10pt>;"2,2"!<-7pt,-10pt> ^-{A}
} \quad 
\xymatrix@C=4em{
 sasta^- \ar@2[rr] ^-{}="src" ^-{saf} \ar@2 [d] _-{\rotatebox{90}{$\sim$}} & & sa \ar@2 [d] ^-{\rotatebox{90}{=}}\\
sasa^-a^-aata^- \ar@2[r] _-{haata'} & aata^- \ar@2 [r] _-{ksa} & sa
\ar@3 "src"!<2pt,-10pt>;"2,2"!<-7pt,-10pt> ^-{B}
}
\]
\[ 
\xymatrix@C=4em{
aas^-a^-s^- \ar@2[rr] ^-{h} ^-{}="src" \ar@2 [d] _-{\rotatebox{90}{$\sim$}} & & 1 \ar@2 [d] ^-{\rotatebox{90}{=}}\\
at^-tas^-a^-s^- \ar@2[r] _-{at^-gs^-} & at^-s^- \ar@2 [r] _-{f^-} & 1 
\ar@3 "src"!<2pt,-10pt>;"2,2"!<-2pt,-10pt> ^-{C}
} 
\quad 
\xymatrix@C=4em{
 aata^-a^-s^- \ar@2[rr] ^-{k} ^-{}="src" \ar@2 [d] _-{\rotatebox{90}{$\sim$}} & & 1 \ar@2 [d] ^-{\rotatebox{90}{=}}\\
aas^-sta^-a^-s^- \ar@2[r] _-{aas^-fa^-s^-} & aas^-a^-s^- \ar@2 [r] _-{h} & 1
\ar@3 "src"!<2pt,-10pt>;"2,2"!<-2pt,-10pt> ^-{D}
}
\]

\[ 
\xymatrix@C=4em{
sasasa^-a^-aa \ar@2[r] ^-{sah^-aa} \ar@2 [d] _-{\rotatebox{90}{$\sim$}} & saaa \ar@2 [r]^-{k^- aata} ^-{}="src" & aata \ar@2 [r] ^-{aagas} & aaas \ar@2 [d] ^-{\rotatebox{90}{=}}\\
sasa^-a^-aaas \ar@2[rrr] _-{h^-aaas} ^-{}="tgt" & & & aaas 
\ar@3 "src"!<2pt,-10pt>;"tgt"!<2pt,10pt> ^-{E} 
} 
\]

\[ 
\xymatrix@C=4em{
sasaat^-a^-a^-aat \ar@2[r] ^-{sak^-aat} \ar@2 [d] _-{\rotatebox{90}{$\sim$}} & saaat \ar@2 [r]^-{k^- aatat} & aatat \ar@2 [r] ^-{aagast} & aaast \ar@2 [r] ^-{aaafa} & aaaa \ar@2 [d] ^-{\rotatebox{90}{=}}\\
sasa^-a^-aaaa \ar@2[rrrr] _-{h^-aaaa} ^-{}="tgt" & & & & aaaa
\ar@3 "1,3"!<0pt,-10pt>;"tgt"!<3pt,10pt> ^-{F}  
} 
\]
together with their six inverses, given by the critical $\ER$-branchings above with all the $\ER$-rewriting steps replaced by their inverse. Note that the square cells $A$ and $C^-$ (resp. $B$ and $Da^-s^-$) give the same relation among relation in the quotient by the axioms of group, so that one can remove $C$ and $D$. Furthermore, applying an homotopical reduction procedure as in \cite{GaussentGuiraudMalbos15} shows that all the square cells $A$, $B$, $C$ and $D$ could be removed, to obtain an empty coherent extension of the presented group.

\begin{remark}
Note that there are other critical $\ER$-branchings that we do not consider here, since they do not imply positive $\ER$-rewriting steps. For instance, we have an $\ER$-branching of the form $(h,aas^-t^-s^-fs^-)$, and the rule $aas^-t^-s^-fs^-$ is not positive since its reduced source has length equals to $5$, while its reduced target has length $6$. This is due to the fact that this overlapping between $h$ and $f$ is on a subword  whose length is smaller than $\frac{1}{2} \text{min}(\ell(sta^-), \ell(aas^-a^-s^-))$.
\end{remark}

\begin{small}
\renewcommand{\refname}{\Large\textsc{References}}
\bibliographystyle{plain}
\bibliography{biblioCURRENT}
\end{small}

\clearpage

\quad

\vfill

\begin{footnotesize}
\bigskip
\auteur{Benjamin Dupont}{bdupont@math.univ-lyon1.fr}
{Univ Lyon, Universit\'e Claude Bernard Lyon 1\\
CNRS UMR 5208, Institut Camille Jordan\\
43 blvd. du 11 novembre 1918\\
F-69622 Villeurbanne cedex, France}

\bigskip
\auteur{Philippe Malbos}{malbos@math.univ-lyon1.fr}
{Univ Lyon, Universit\'e Claude Bernard Lyon 1\\
CNRS UMR 5208, Institut Camille Jordan\\
43 blvd. du 11 novembre 1918\\
F-69622 Villeurbanne cedex, France}
\end{footnotesize}

\vspace{1.5cm}

\begin{small}---\;\;\today\;\;-\;\;\hhmm\;\;---\end{small}
\end{document}